\documentclass[reqno]{amsart}
\usepackage{amssymb,amsmath,amsfonts,amsthm,mathrsfs}
\usepackage[colorlinks=true, citecolor=blue, anchorcolor=red]{hyperref}
\usepackage{enumerate}
\usepackage{bbm}
\usepackage{tikz-cd}

\newcommand{\R}{\mathbb R}
\newcommand{\N}{\mathbb N}
\newcommand{\C}{\mathbb C}
\newcommand{\Z}{\mathbb Z}

\renewcommand{\Re}{\mathop{\text{\upshape{Re}}}}
\renewcommand{\Im}{\mathop{\text{\upshape{Im}}}}

\renewcommand{\Im}{\operatorname{Im}}

\renewcommand{\epsilon}{\varepsilon}
\renewcommand{\bar}[1]{\overline{#1}}

\renewcommand{\tilde}{\widetilde}
\renewcommand{\phi}{\varphi}

\def\cO{\mathcal{O}}

\def\cS{\mathcal{S}}

\def\txtd{{\textnormal{d}}}
\def\txte{{\textnormal{e}}}
\def\txti{{\textnormal{i}}}

\def\txtD{{\textnormal{D}}}

\def\Re{{\textnormal{Re}}}
\def\Im{{\textnormal{Im}}}

\def\ra{\rightarrow}

\newtheorem{theorem}{Theorem}[section]
\newtheorem{lemma}[theorem]{Lemma}
\newtheorem{corollary}[theorem]{Corollary}
\newtheorem{proposition}[theorem]{Proposition}
\theoremstyle{definition}

\newtheorem{remark}[theorem]{Remark}
\newtheorem{example}[theorem]{Example}

\numberwithin{equation}{section}

\def\XXint#1#2#3{{\setbox0=\hbox{$#1{#2#3}{\int}$ }
\vcenter{\hbox{$#2#3$ }}\kern-.6\wd0}}

\newcommand{\be}{\begin{equation}}
\newcommand{\ee}{\end{equation}}
\newcommand{\benn}{\begin{equation*}}
\newcommand{\eenn}{\end{equation*}}
\newcommand{\bea}{\begin{eqnarray}}
\newcommand{\eea}{\end{eqnarray}}
\newcommand{\beann}{\begin{eqnarray*}}
\newcommand{\eeann}{\end{eqnarray*}}

\makeatletter
\@namedef{subjclassname@2020}{\textup{2020} Mathematics Subject Classification}
\makeatother

\begin{document}

\title[Slow Manifolds for Infinite-Dimensional Evolution Equations]
{Slow Manifolds for Infinite-Dimensional Evolution Equations}
\author{Felix Hummel}
\address{Technical University of Munich\\ Department of Mathematics \\ Boltzmannstra{\ss}e 3\\ 85748 Garching bei M\"unchen \\ Germany}
\email{hummel@ma.tum.de}

\author{Christian Kuehn}
\address{Technical University of Munich\\ Department of Mathematics \\ Boltzmannstra{\ss}e 3\\ 85748 Garching bei M\"unchen \\ Germany}
\email{ckuehn@ma.tum.de}


\begin{abstract}
	We extend classical finite-dimensional Fenichel theory in two directions to infinite dimensions. Under comparably weak assumptions we show that the solution of an infinite-dimensional fast-slow system is approximated well by the corresponding slow flow. After that we construct a two-parameter family of slow manifolds $S_{\epsilon,\zeta}$ under more restrictive assumptions on the linear part of the slow equation. The second parameter $\zeta$ does not appear in the finite-dimensional setting and describes a certain splitting of the slow variable space in a fast decaying part and its complement. The finite-dimensional setting is contained as a special case in which $S_{\epsilon,\zeta}$ does not depend on $\zeta$. Finally, we apply our new techniques to three examples of fast-slow systems of partial differential equations.
\end{abstract}

\maketitle

\section{Introduction}
\label{sec:intro}

In this work, we study infinite-dimensional fast-slow evolution equations of the form
\be
\label{eq:maineq}
\begin{array}{rcl}
\epsilon \partial_t u^\epsilon &=& A u^\epsilon + f(u^\epsilon,v^\epsilon),\\
\partial_t v^\epsilon &=& B v^\epsilon + g(u^\epsilon,v^\epsilon), 
\end{array}
\ee
where $\varepsilon \geq 0$ is a small parameter, $A$ and $B$ are linear operators on Banach spaces $X$ and $Y$ respectively, $f,g$ are sufficiently regular nonlinearities, and $(u^\epsilon,v^\epsilon)=(u^\epsilon(t),v^\epsilon(t))\in X\times Y$ are the unknown functions, where the superscript indicates the dependence of the solution on $\epsilon$. In particular, the class of systems~\eqref{eq:maineq} are multiscale evolution equations, where the small parameter $\epsilon$ hints at a formal time-scale separation between the variables $u^\epsilon$ and $v^\epsilon$. 

The motivation to study~\eqref{eq:maineq} is best explained via the finite-dimensional setting, where $(u^\epsilon,v^\epsilon)\in\R^m\times \R^n$, $A\in\R^{m\times m}$, $B\in\R^{n\times n}$, and one often assumes that $f,g$ are sufficiently smooth. Multiple time scale ordinary differential equations (ODEs) are employed across broad areas of mathematics~\cite{KuehnBook} and form one of the few classes of higher-dimensional dynamical systems, where analytical results about nonlinear dynamics can be obtained due to the time scale separation structure. If we let $\epsilon\rightarrow 0$ in~\eqref{eq:maineq} we obtain the slow subsystem (or reduced system)
\be
\label{eq:slowsub}
\begin{array}{rcl}
0 &=& A u^0 + f(u^0,v^0),\\
\partial_t v^0 &=& B v^0 + g(u^0,v^0), 
\end{array}
\ee
which is a differential-algebraic equation defined on the critical set
\benn
S_0:=\{(u^0,v^0)\in\R^m\times \R^n:0=A u^0 + f(u^0,v^0)\},
\eenn 
which we shall assume to be a manifold referred to as the critical manifold. If $\cS_0\subseteq S_0$ is compact and normally hyperbolic submanifold, i.e., all eigenvalues of the matrix $A+\txtD_uf(z)\in\R^{m\times m}$ have nonzero real part for all $z\in\cS_0$, then Fenichel-Tikhonov theory~\cite{Fenichel4,Tikhonov} guarantees the existence of a locally invariant slow manifold $\cS_\epsilon$. Of course, for practical applications, the case of a critical manifold, which attracting in the fast directions, is the most frequently encountered. This case occurs when all eigenvalues of $A+\txtD_uf(z)$ have negative real part and we shall focus on the attracting setting here. For any normally hyperbolic critical manifold, the flow on $\cS_\epsilon$ is approximated well by the slow subsystem flow of~\eqref{eq:maineq}; see also~\cite{Jones,KuehnBook,WigginsIM} for detailed expositions of Fenichel theory. One reason to intuitively expect such an approximation result in finite dimensions is better visible on the fast time scale $r:=t/\epsilon$, which leads upon substitution in~\eqref{eq:maineq} to 
\be
\label{eq:fastsys}
\begin{array}{rcl}
\partial_{r} u^\epsilon &=& A u^\epsilon + f(u^\epsilon,v^\epsilon),\\
\partial_r v^\epsilon &=& \epsilon (B v^\epsilon + g(u^\epsilon,v^\epsilon)). 
\end{array}
\ee
Indeed, sending $\epsilon\ra 0$ in~\eqref{eq:fastsys} yields the fast subsystem (or layer equations)
\be
\label{eq:fastsub}
\begin{array}{rcl}
\partial_{r} u^0 &=& A u^0 + f(u^0,v^0),\\
\partial_r v^0 &=& 0. 
\end{array}
\ee
The full fast-slow system on $\R^m\times \R^n$ can then be treated near $\cS_0$ as a bounded perturbation of the fast subsystem since $B$ and $g$ satisfy local bounds due to the assumptions of sufficient regularity on $g$, so the fast linear hyperbolic dynamics driven by $A+\txtD_uf(z)$ for $z\in \cS_0$ dominates near $z$. To make this intuition precise is already difficult in the finite-dimensional setting with Fenichel theory providing the comprehensive standard~\cite{Fenichel4}, even for multiple time scale dynamical systems, which cannot be written directly~\cite{Wechselberger4} in the standard form~\eqref{eq:maineq}.\medskip   

Transferring the finite-dimensional situation to general evolution equations on Banach spaces turns out to be challenging. At first sight, one may hope that the classical Fenichel approach to show the existence of $\cS_\epsilon$ via a Lyapunov-Perron method or via a Hadamard graph transform~\cite{Fenichel4,WigginsIM} can still be applied utilizing variants/extensions of infinite-dimensional center manifold theory~\cite{VanderbauwhedeIooss}. So far, the best available results in this direction are due to Bates et al.~\cite{Bates_Lu_Zeng_1998,Bates_Lu_Zeng_2008}, who cover the case of semiflows, when the perturbation induced by the slow dynamics is bounded. In particular, this includes the case of partially dissipative systems, where $A=\Delta$ is the Laplacian and $B=0$ so that the slow variable dynamics is an ODE. Yet, even for quite standard reaction-diffusion systems~\cite{Grindrod,Henry,KuehnBook1} with $A=\Delta$ and $B=\Delta$ on bounded domains, there has been no major progress to generalize Fenichel's theory from the 1970s. The main problem is that on the fast time scale we can never view $\epsilon Bv^\epsilon$ as a bounded perturbation if $B$ is a differential operator (this statement will be made precise below); indeed, for differential operators we encounter the formal limit $0\cdot \infty$ since $B$ is an unbounded operator. Furthermore, the classical concept of normal hyperbolicity is problematic since $\epsilon Bv^\epsilon$ is not necessarily ``small'' in any norm compared to the linear part of the $u^\epsilon$-variable. For example, when $B=\Delta$ on a bounded domain, a spectral Galerkin decomposition shows that the $v^\epsilon$-variable may have fast decaying components in its linear part. This implies that the case of hyperbolic operators for $B$ (which we include here as well) is somewhat easier. In fact, a very special case of fast-slow invariant manifold theory was carried out for the Maxwell-Bloch equations in~\cite{MenonHaller}, where $u^\epsilon$ is governed by an ODE and $B$ is a first-order partial derivative.

Another hope might be that one can adapt the theory of inertial manifolds~\cite{Robinson1,Temam}, which has been used to constructed low-dimensional attracting invariant manifolds for several classes of partial differential equations (PDEs). Yet, inertial manifold theory is based on global dissipation and compact embeddings to construct reduced lower-dimensional invariant manifolds. For the fast-slow evolution system~\eqref{eq:maineq}, we are not interested in global reduction but local persistence/perturbation of manifolds. In fact, we shall see below that our slow manifold can even grow upon perturbation in a suitable sense in comparison to the critical manifold.\medskip

In this work, we provide a quite general fast-slow invariant manifold theory for the evolution equations~\eqref{eq:maineq}. We briefly outline our results in a non-technical form:

\begin{itemize}
 \item We identify the key problems with Fenichel theory in infinite dimensions via several explicit examples including the problems with unbounded and differential operators $B$ as well as with the notion of normal hyperbolicity; see Section~\ref{Section:Bates_Difficult}. 
 \item We assume that $A$ is the generator of a $C_0$ semigroup having zero in its resolvent and that the nonlinearity is (locally) Lipschitz. Then we prove an approximation result that the flow of the full evolution equation for sufficiently small $\epsilon>0$ is, near $S_0$, well-approximated by the flow of the slow subsystem on $S_0$; see Theorem~\ref{Thm:Original_and_Modified_Flow_Close}.
 \item Under suitable regularity assumptions on $B$ and $g$, we prove the existence of a two-parameter family of slow manifolds $S_{\epsilon,\zeta}$. The second small parameter $\zeta>0$ controls additional ``fast'' contributions of the $v^\epsilon$-dynamics. We also prove differentiability of $S_{\epsilon,\zeta}$ if $f$ is $C^1$, we show estimates on the distance of $S_{\epsilon,\zeta}$ to the critical manifold, and a result regarding local attraction of trajectories near $S_{\epsilon,\zeta}$; see Section~\ref{sec:slowfinal}.      
\end{itemize}      

In the proofs, there are several important new technical steps. The approximation result given in Theorem~\ref{Thm:Original_and_Modified_Flow_Close} does not provide a slow manifold, and is hence weaker than classical Fenichel theory but it also uses weaker assumptions. It shows that there exists a very general result that the slow subsystem can be used to approximate the full dynamics in a suitable sense near $S_0$. In fact, the proof of this result seems to be difficult to achieve on the fast time scale, or even directly with the original full evolution equations~\eqref{eq:maineq} on the slow time scale. We use an intermediate approximating evolution equation (see also the calculations starting from equation \eqref{Eq:Modified_Fast}), which changes the right-hand side of the fast component as follows
\be
\label{eq:intermediate}
\begin{array}{rcl}
\epsilon \partial_t u^{\epsilon,0} &=& A u^{\epsilon,0} + f(u^{\epsilon,0},v^0) -\epsilon\partial_tA^{-1}f(h^0(v^0),v^0),\\
\partial_t v^0 &=& B v^0 + g(h^0(v^0),v^0), 
\end{array}
\ee
where $h^0:Y\rightarrow X$ is a local parametrization of the critical manifold. On the finite-dimensional level, when $X=\R^m$ and $Y=\R^n$ one can nicely see, why this choice might be helpful. Looking formally at different orders of $\cO(\epsilon^k)$ one has for $k=0,1$ from the first equation 
\benn
A u^{0,0} + f(u^{0,0},v^0)=0\quad \text{and}\quad u^{0,0}+A^{-1}f(h^0(v^0),v^0) =\textnormal{constant},
\eenn 
so upon using an initial condition with $h^0(v^0)=u^{0,0}$ one just obtains the condition of the critical manifold twice, to leading-order and to first order in $\epsilon$. This means that our intermediate system \eqref{eq:intermediate} is likely to be a locally better approximation to the full fast-slow dynamics near $S_0$ and it is a regularization of the slow subsystem. Other important ingredients to obtain the approximation result are the use of interpolation-extrapolation scales and suitably adapted Gronwall-type arguments involving mild solutions. 

For the construction of the slow manifold family $S_{\epsilon,\zeta}$, we use a re-partitioning the slow dynamics into two parts, which can formally be expressed as
\benn
Y=Y^\zeta_F\oplus Y^\zeta_S.
\eenn
The part $Y^\zeta_S$ comes from modes/directions, where $\epsilon B$ yields a sufficiently small perturbation so that these modes are slow. The other part $Y^\zeta_F$ comes from modes, which are fast as $B$ dominates the small parameter $\epsilon$. The control of this splitting leads to a doubly-singularly perturbed problem with a second small parameter $\zeta>0$. Evidently, such a splitting relies on having a certain spectral gap of the slow dynamics, which we need to impose. Having this splitting available, we then proceed to set up a Lyapunov-Perron functional iteration to obtain the existence of $S_{\epsilon,\zeta}$. The dynamical properties of $S_{\epsilon,\zeta}$ can be established using relatively long estimates in combination with mild solution representations, time differentiation of the manifold parametrization along solutions, and contraction mapping arguments.\medskip

The paper is structured as follows: In Section~\ref{sec:prelim} we collect technical background results regarding interpolation-extrapolation scales of Banach spaces and operators on these spaces, as well as suitable variants of Gronwall-type lemmas. In Section~\ref{Section:Bates_Difficult}, we illustrate the difficulties of the classical Fenichel viewpoint and the barriers to generalize the bounded perturbation results for semiflows. In Section~\ref{sec:genfs}, we prove the general result on slow flow approximation for semiflows, while in Section~\ref{sec:slowfinal} we obtain the slow manifold family and its precise dynamic properties. We present three illustrating examples in Section~\ref{sec:examples} and conclude with an outlook in Section~\ref{sec:outlook}.   


\section{Preliminaries}
\label{sec:prelim}

\subsection{Interpolation-Extrapolation Scales}

We briefly introduce some required notions and results in connection with interpolation-extrapolation scales. As a general reference, we would like to mention \cite[Chapter V]{Amann_1995}.\\
 Let $T\colon X\supset D(T)\to X$ be a densely defined closed linear operator on a Banach space $X$ with $0\in\rho(T)$. Moreover, for $\theta\in(0,1)$ let $(\cdot,\cdot)_{\theta}$ be an exact admissible interpolation functor, i.e. an exact interpolation functor such that $X_1$ is dense in $(X_0,X_1)_{\theta}$ whenever $X_1\stackrel{d}{\hookrightarrow} X_0$. We define a family of Banach spaces $(X_{\alpha})_{\alpha\in[-1,\infty)}$ and a family of operators $(T_{\alpha})_{\alpha\in[-1,\infty)}\in\mathcal{B}(X_{\alpha},X_{\alpha+1})$ as follows :
 \begin{itemize}
  \item For $k\in\N_0$ we choose $X_k:=D(T^k)$ endowed with $\|x\|_{X_k}:=\|T^kx\|_X$ $(x\in D(T^k))$. In particular, $X_0=D(T^0)=D(\operatorname{id}_X)=X$. Moreover, $T_k:=T\vert_{E_{k+1}}$.
  \item $X_{-1}$ is defined as the completion of $X=X_0$ with respect to the norm $\|x\|_{X_{-1}}=\| T^{-1} x \|_{X_0}$. The operator $T_{0}=T$ is then closable on $X_{-1}$ and $T_{-1}$ is defined to be the closure. One can also define  $(X_{-k},T_{-k})$ for $k\in\N$ by iteration, but we do not go beyond $k=-1$ in this paper.
  \item For $k\in\N_0\cup\{-1\}$, $\theta\in(0,1)$ and $\alpha=k+\theta$ we define $X_{\alpha}:=(X_{k},X_{k+1})_{\theta}$ and $T_{\alpha}=T_k\vert_{D(T_{\alpha})}$ where
  \[
   D(T_{\alpha})=\{x\in X_{k+1}: T_kx\in X_{\alpha}\}.
  \]
 \end{itemize}
The family $(X_{\alpha},T_{\alpha})_{\alpha\in[-1,\infty)}$ is a densely injected Banach scale in the sense that
\[
 X_{\alpha}\stackrel{d}{\hookrightarrow} X_{\beta}
\]
whenever $\alpha\geq\beta$ (i.e. the injection is continuous with dense range).  and
\[
 T_{\alpha}\colon X_{\alpha+1}\to X_{\alpha}
\]
is an isomorphism for all $\alpha\in\R$. Moreover $T_{\alpha}\colon X_{\alpha}\supset X_{\alpha+1}\to X_{\alpha}$ is a densely defined closed linear operator with $0\in\rho(T_{\alpha})$ for all $\alpha\in\R$. The family $(X_{\alpha},T_{\alpha})_{\alpha\in\R}$ is an interpolation-extrapolation scale.

One of the nice things about interpolation-extrapolation scales is that semigroups can be shifted along these scales. More precisely, we have the following (c.f. \cite[Chapter V, Theorem 2.1.3]{Amann_1995}):
\begin{theorem}\label{Thm:Semigroup_in_Scales}
 Let $T$ be the generator of a $C_0$-semigroup $(S(t))_{t\geq0}$ and let $\omega_S\in\R$ be the growth bound of $S$, i.e.
 \[
	 \omega_S:=\{\omega\in\R\,\vert\,\exists M>0~\forall t\geq0: \|S(t)\|_{\mathcal{B}(X)}\leq M\txte^{\omega t} \}.
 \]
Then $T_{\alpha}\colon X_{\alpha}\supset X_{\alpha+1}\to X_{\alpha}$ also generates a $C_0$ semigroup $(S_{\alpha}(t))_{t\geq0}$ with the same growth bound and for all $\alpha,\beta \in[-1,\infty)$, $\alpha\geq \beta$, the diagram
\[
\begin{tikzcd}
X_{\alpha} \arrow{r}{S_{\alpha}(t)} \arrow[hookrightarrow]{d} & X_{\alpha} \arrow[hookrightarrow]{d} \\
X_{\beta} \arrow{r}{S_{\beta}(t)}& X_{\beta}
\end{tikzcd}
\]
 commutes. Moreover, if $(S(t))_{t\geq0}$ is holomorphic then the same holds for $(S_{\beta}(t))_{t\geq0}$ and for all $\omega> \omega_S$ there is a constant $C$ also depending on $\alpha$ and $\beta$ such that
 \[
  \| S_{\beta}(t) \|_{\mathcal{B}(E_{\beta},E_{\alpha})}\leq Ct^{\beta-\alpha}\txte^{-\omega t}\quad(t>0).
 \]
\end{theorem}

\subsection{Estimates for the Incomplete Gamma Function}

In this paper we frequently encounter terms of the form
\[
	\int_0^t\frac{\txte^{\epsilon^{-1}\omega(t-s)}}{\epsilon^{\gamma}(t-s)^{1-\gamma}}\,\txtd s
\]
with $\gamma\in(0,1]$, $\omega<0$ and $\epsilon>0$. In the following, we derive certain elementary estimates which we use several times. They might not be of great importance on their own, but being able to refer to them will be useful at some places. Note that the substitution $r=-\epsilon^{-1}\omega(t-s)$ yields
\[
	\int_0^t\frac{\txte^{\epsilon^{-1}\omega(t-s)}}{\epsilon^{\gamma}(t-s)^{1-\gamma}}\,\txtd s=\frac{1}{|\omega|^{\gamma}}\int_0^{\epsilon^{-1}\omega t}\frac{\txte^{-r}}{r^{1-\gamma}}\,\txtd r=\frac{\tilde{\Gamma}(\gamma,\epsilon^{-1}\omega t)}{|\omega|^{\gamma}},
\]
where $\tilde{\Gamma}(\gamma,t):=\int_0^t\frac{\txte^{-r}}{r^{1-\gamma}}\,\txtd r$ denotes the incomplete gamma function.

\begin{lemma}\label{Lemma:Incomplete_Gamma_1}
	For all $t\geq0$, $\epsilon>0$, $\gamma\in(0,1]$ and $\omega<0$ it holds that
	\[
		\int_0^t\frac{\txte^{\epsilon^{-1}\omega(t-s)}}{\epsilon^{\gamma}(t-s)^{1-\gamma}}\,\txtd s\leq \min\left\{\frac{t^\gamma}{\gamma\epsilon^{\gamma}},\frac{\Gamma(\gamma)}{|\omega|^{\gamma}}\right\}.
	\]
	Here, $\Gamma$ denotes the gamma function.
\end{lemma}
\begin{proof}
 H\"older's inequality yields
 \[
  \int_0^t\frac{\txte^{\epsilon^{-1}\omega(t-s)}}{\epsilon^{\gamma}(t-s)^{1-\gamma}}\leq \frac{1}{\epsilon^{\gamma}}\int_0^t\frac{1}{(t-s)^{1-\gamma}}\,\txtd s=\frac{t^\gamma}{\gamma\epsilon^{\gamma}}.
 \]
 On the other hand, since $\tilde{\Gamma}(\gamma,t)$ is increasing in $t$, it follows that
 \[
 	\int_0^t\frac{\txte^{\epsilon^{-1}\omega(t-s)}}{\epsilon^{\gamma}(t-s)^{1-\gamma}}\,\txtd s=\frac{\tilde{\Gamma}(\gamma,\epsilon^{-1}\omega t)}{|\omega|^{\gamma}}\leq \lim_{t\to\infty}\frac{\tilde{\Gamma}(\gamma,\epsilon^{-1}\omega t)}{|\omega|^{\gamma}}=\frac{\Gamma(\gamma)}{|\omega|^{\gamma}}.
 \]
\end{proof}

\begin{lemma}\label{Lemma:Incomplete_Gamma_2}
	For all $t\geq0$, $\epsilon>0$, $\gamma\in(0,1]$ and $\omega<\tilde{\omega}$ it holds that
	\[
		\txte^{\epsilon^{-1}\tilde{\omega}t}\int_0^t\frac{\txte^{\epsilon^{-1}(\omega-\tilde{\omega})s}}{\epsilon^{\gamma}s^{1-\gamma}}\,\txtd s\leq\frac{\txte^{\gamma}}{\gamma^{1-\gamma}|\tilde{\omega}|^{\gamma}}
	\]
\end{lemma}
\begin{proof}
	By Lemma \ref{Lemma:Incomplete_Gamma_1} it holds that
		\[
		\txte^{\epsilon^{-1}\tilde{\omega}t}\int_0^t\frac{\txte^{\epsilon^{-1}(\omega-\tilde{\omega})s}}{\epsilon^{\gamma}s^{1-\gamma}}\,\txtd s\leq \txte^{\epsilon^{-1}\tilde{\omega}t}\frac{t^\gamma}{\gamma\epsilon^{\gamma}}.
	\]
	The right hand side attains its maximum at $t=|\gamma\epsilon\tilde{\omega}^{-1}|$. This yields the assertion.
\end{proof}

\begin{lemma}\label{Lemma:Incomplete_Gamma_3}
	For all $t\geq0$, $\epsilon>0$, $\gamma\in(0,1]$ and $\omega<\tilde{\omega}<0$ it holds that
	\[
		\int_0^t \epsilon^{-1}|\omega|\txte^{\epsilon^{-1}\tilde{\omega}(t-s)}\int_0^s\frac{\txte^{\epsilon^{-1}\omega r}}{\epsilon^{\gamma}r^{1-\gamma}}\,\txtd r\,\txtd s\leq \frac{\Gamma(\gamma)|\omega|^{1-\gamma}}{\tilde{\omega}}
	\]
\end{lemma}
\begin{proof}
	Using Lemma \ref{Lemma:Incomplete_Gamma_1} we obtain
	\begin{align*}
	\int_0^t \epsilon^{-1}|\omega|\txte^{\epsilon^{-1}\tilde{\omega}(t-s)}\int_0^s\frac{\txte^{\epsilon^{-1}\omega r}}{\epsilon^{\gamma}r^{1-\gamma}}\,\txtd r\,\txtd s&\leq \epsilon^{-1}|\omega|^{1-\gamma}\Gamma(\gamma)\int_0^t\txte^{\epsilon^{-1}\tilde{\omega}(t-s)}\,\txtd s\\
	&\leq \frac{\Gamma(\gamma)|\omega|^{1-\gamma}}{\tilde{\omega}}
	\end{align*}
\end{proof}

\begin{corollary}\label{Cor:Incomplete_Gamma_1}
For all $t\geq0$, $\epsilon>0$, $\gamma\in(0,1]$ and $\omega<\tilde{\omega}<0$ it holds that
	\[
		\int_0^t \left(\frac{\txte^{\epsilon^{-1}\omega s}}{\epsilon^{\gamma}s^{1-\gamma}}+\epsilon^{-1}|\omega|\int_0^s\frac{\txte^{\epsilon^{-1}\omega r}}{\epsilon^{\gamma}r^{1-\gamma}}\,\txtd r\right)\txte^{\epsilon^{-1}\tilde{\omega}(t-s)}\,\txtd s\leq \left(\frac{\txte^{\gamma}}{\gamma^{1-\gamma}}+\Gamma(\gamma)\left|\frac{\omega}{\tilde{\omega}}\right|^{1-\gamma}\right)\frac{1}{\tilde{\omega}^\gamma}
	\]
\end{corollary}
\begin{proof}
	This follows from summing up the estimates of Lemma \ref{Lemma:Incomplete_Gamma_2} and Lemma \ref{Lemma:Incomplete_Gamma_3}.
\end{proof}

\begin{lemma}\label{Lemma:Not_Incomplete_Gamma}
	Let $\omega<0$ and $\gamma\in(0,1]$. Then it holds that
	\[
		\int_0^t\frac{\txte^{\omega s}}{(t-s)^{1-\gamma}}\,ds\leq \frac{\txte^{1+\omega t}+\gamma}{\gamma|\omega|^{\gamma}}.
	\]
\end{lemma}
\begin{proof}
	This follows from
	\begin{align*}
	\int_0^t\frac{\txte^{\omega s}}{(t-s)^{1-\gamma}}\txtd s&=e^{\omega t}\int_0^t\frac{\txte^{-\omega s}}{s^{1-\gamma}}\txtd s=\frac{\txte^{\omega t}}{|\omega|^{\gamma}}\int_0^{|\omega| t} \frac{\txte^{r}}{r^{1-\gamma}}\txtd r\\
	&\leq \frac{\txte^{\omega t}}{|\omega|^{\gamma}}\left(\int_0^{1}\frac{\txte^{r}}{r^{1-\gamma}}\txtd r+\int_1^{\max\{1,|\omega|t\}}\frac{\txte^{r}}{r^{1-\gamma}}\txtd r\right)\\
	&\leq \frac{\txte^{\omega t}}{|\omega|^{\gamma}}\left(\frac{\txte}{\gamma}+e^{-\omega t}\right)= \frac{\txte^{1+\omega t}+\gamma}{\gamma|\omega|^{\gamma}}.
	\end{align*}
\end{proof}

\subsection{Some Gronwall Type Inequalities}
In most of the proofs of this paper, Gronwall type inequalities are essential ingredients. Here, we collect the versions which we use throughout this work.

\begin{lemma}\label{Lemma:Gronwall_Version}
 Let $T>0$, $u,v,c\colon [0,T]\to[0,\infty)$ be continuous and suppose that $c'$ is locally integrable. If $v(t)\leq c(t)+\int_0^t u(s) v(s)\,\txtd s$ for all $t\in[0,T]$, then
 \[
  v(t)\leq c(0)\exp\bigg(\int_0^t u(s)\,\txtd s\bigg)+\int_0^t c'(s)\exp\bigg(\int_{s}^t u(r)\,\txtd r\bigg)\,\txtd s\quad(t\in[0,T]).
 \]

\end{lemma}
\begin{proof}
    This is a well-known version of Gronwall's inequality. A proof can for example be found in \cite[Corollary 2]{Dragomir_2003}. The statement therein is formulated for $c$ being differentiable, but the argument relies on integration by parts and thus, also the asserted version holds true.
\end{proof}

\begin{lemma}\label{Lemma:Gronwall_Specific}
 Let $x\in\R$, $\epsilon,N,T>0$, $\gamma\in(0,1]$, $p\in(1,\infty)$ and let $p'=\frac{p}{p-1}$ be the conjugated H\"older index. Let further $v,c\colon [0,T]\to[0,\infty)$ be continuous. Suppose that $c'$ is locally integrable and that $[t\mapsto \txte^{-\epsilon^{-1}xt}c(t)]$ is non-decreasing. If
 \[
  v(t)\leq c(t)+N\int_0^t\frac{\txte^{\epsilon^{-1}x(t-s)}}{\epsilon^{\gamma}(t-s)^{1-\gamma}}v(s)\,\txtd s
 \]
for all $t\in[0,T]$, then we have the estimate
\[
 v(t) \leq pc(0)\txte^{\epsilon^{-1}\tilde{x}t}+p\int_0^t (c'(s)-\epsilon^{-1}xc(s))\txte^{\epsilon^{-1}\tilde{x}(t-s)}\,\txtd s\quad(t\in[0,T])
\]
where $\tilde{x}:=x+p N^{\tfrac{1}{\gamma}}(\tfrac{p'}{\gamma})^{\frac{1-\gamma}{\gamma}}$.
\end{lemma}
\begin{proof}
 Let $\theta(t):= \sup_{0\leq s\leq t}\txte^{-\epsilon^{-1}xs}v(s)$. Then we have the estimate
 \[
  \txte^{-\epsilon^{-1}xt}v(t)\leq c(t)\txte^{-\epsilon^{-1}xt}+N\int_0^t\frac{1}{\epsilon^{\gamma}(t-s)^{1-\gamma}}\theta(s)\,\txtd s
 \]
If we choose $\sigma=(\frac{\gamma}{p'N})^{1/\gamma}\epsilon$, then we obtain
\begin{align*}
 \txte^{-\epsilon^{-1}xt}v(t)&\leq c(t)\txte^{-\epsilon^{-1}xt}+N\int_0^{[t-\sigma]_+}\frac{1}{\epsilon^{\gamma}(t-s)^{1-\gamma}}\theta(s)\,\txtd s\\
 &\qquad\qquad+N\int_{[t-\sigma]_+}^{t}\frac{1}{\epsilon^{\gamma}(t-s)^{1-\gamma}}\theta(t)\,\txtd s\\
 &\leq c(t)\txte^{-\epsilon^{-1}xt}+\frac{N}{\epsilon^{\gamma}\sigma^{1-\gamma}}\int_0^{t}\theta(s)\,\txtd s-\frac{N}{\gamma\epsilon^\gamma}\big[(t-s)^\gamma\big]_{s=[t-\sigma]_+}^t\theta(t)\\
 &\leq c(t)\txte^{-\epsilon^{-1}xt}+\frac{N}{\epsilon^{\gamma}\sigma^{1-\gamma}}\int_0^{t}\theta(s)\,\txtd s+\frac{1}{p'}\theta(t)
\end{align*}
By the monotonicity of the right hand side, it follows that we can replace $\txte^{-\epsilon^{-1}xt}v(t)$ by $\theta(t)$ on the left hand side. Therefore, we obtain
\[
 \theta(t)\leq pc(t)\txte^{-\epsilon^{-1}xt}+\frac{pN}{\epsilon^{\gamma}\sigma^{1-\gamma}}\int_0^{t}\theta(s)\,\txtd s
\]
so that Lemma \ref{Lemma:Gronwall_Version} implies
\begin{align*}
 \theta(t)&\leq pc(0)\exp\left(\frac{pN}{\epsilon^{\gamma}\sigma^{1-\gamma}}t\right)\\
 &\qquad\qquad+p\int_0^t (c'(s)-\epsilon^{-1}xc(s))\exp\left(-\epsilon^{-1}xs+\frac{pN}{\epsilon^{\gamma}\sigma^{1-\gamma}}(t-s)\right)\,\txtd s
\end{align*}
and therefore
\begin{align*}
 v(t)&\leq pc(0)\exp\left(\bigg(\epsilon^{-1}x+\frac{pN}{\epsilon^{\gamma}\sigma^{1-\gamma}}\bigg)t\right)\\
 &\qquad\qquad+p\int_0^t (c'(s)-\epsilon^{-1}xc(s))\exp\left(\bigg(\epsilon^{-1}x+\frac{pN}{\epsilon^{\gamma}\sigma^{1-\gamma}}\bigg)(t-s)\right)\,\txtd s\\
 &=pc(0)\txte^{\epsilon^{-1}\tilde{x}t}+p\int_0^t (c'(s)-\epsilon^{-1}xc(s))\txte^{\epsilon^{-1}\tilde{x}(t-s)}\,\txtd s.
\end{align*}
\end{proof}

\begin{remark}\label{Rem:Gronwall_Specific}
	For the sake of simplicity, we will apply Lemma~\ref{Lemma:Gronwall_Specific} with $p=2$ most of the time. However, this is not optimal in many cases. In particular, if $\gamma=1$ then it is actually better to take $p$ close to $1$. This way, we may actually take $\omega_f>\omega_A+C_AL_f$ instead of $\omega_f=\omega+(2C_AL_{f})^{\frac{1}{\gamma}} (\frac{1}{\gamma})^{\frac{1-\gamma}{\gamma}}$ later in this paper. This might be of importance if one wants $\omega_f$ to be as small as possible.
\end{remark}

\begin{lemma}\label{Lemma:Gronwall_Specific_Sum}
 Let $x,y\in\R$, $\epsilon,N,M,T>0$ as well as $\gamma,\delta\in(0,1]$. Let further $v,c\colon [0,T]\to[0,\infty)$ be continuous. Suppose that $c'$ is locally integrable and that $[t\mapsto \txte^{-yt}c(t)]$ is non-decreasing. If $0<\frac{N\Gamma(\gamma)}{(\epsilon y-x)^{\gamma}}<1$ and if
 \[
  v(t)\leq c(t)+N\int_0^t\frac{\txte^{\epsilon^{-1}x(t-s)}}{\epsilon^{\gamma}(t-s)^{1-\gamma}}v(s)\,\txtd s+M\int_0^t\frac{\txte^{y(t-s)}}{(t-s)^{1-\delta}}v(s)\,\txtd s
 \]
for all $t\in[0,T]$, then for all $\mu\in(0,1-\frac{N\Gamma(\gamma)}{(\epsilon y-x)^{\gamma}})$ we have the estimate
\[
 v(t) \leq \frac{1}{1-\mu-\frac{N\Gamma(\gamma)}{(\epsilon y-x)^{\gamma}}}\left[c(0)\txte^{\tilde{y}t}+\int_0^t (c'(s)-yc(s))\txte^{\tilde{y}(t-s)}\,\txtd s\right]\quad(t\in[0,T])
\]
where $\tilde{y}:=y+M^{\frac{1}{\delta}}(\delta\mu)^{\frac{\delta-1}{\delta}}(1-\mu-\frac{N\Gamma(\gamma)}{(\epsilon y-x)^{\gamma}})^{-1}$.
\end{lemma}
\begin{proof}
	The proof is similar to the one of Lemma~\ref{Lemma:Gronwall_Specific}. We define $$\theta(t):=\sup_{0\leq s \leq t} \txte^{-ys}v(s)$$ so that we obtain
	\begin{align*}
		\txte^{-yt}v(t)&\leq \txte^{-yt}c(t)+N\int_0^t\frac{\txte^{(\epsilon^{-1}x-y)(t-s)}}{\epsilon^{\gamma}(t-s)^{1-\gamma}}\theta(s)\,\txtd s+M\int_0^t\frac{1}{(t-s)^{1-\delta}}\theta(s)\,\txtd s\\
		&\leq \txte^{-yt}c(t)+N\int_0^t\frac{\txte^{(\epsilon^{-1}x-y)(t-s)}}{\epsilon^{\gamma}(t-s)^{1-\gamma}}\,\txtd s\;\theta(t)+M\int_0^t\frac{1}{(t-s)^{1-\delta}}\theta(s)\,\txtd s\\
		&\leq \txte^{-yt}c(t)+\frac{N\Gamma(\gamma)}{(\epsilon y-x)^{\gamma}}\theta(t)+M\int_0^t\frac{1}{(t-s)^{1-\delta}}\theta(s)\,\txtd s,
	\end{align*}
	where we used Lemma~\ref{Lemma:Incomplete_Gamma_1}.
	For some $\sigma\geq0$ we split again
	{\allowdisplaybreaks{\begin{align*}
		\txte^{-yt}v(t)&\leq \txte^{-yt}c(t)+\frac{N\Gamma(\gamma)}{(\epsilon y-x)^{\gamma}}\theta(t)+\int_0^{[t-\sigma]_+}M\frac{1}{(t-s)^{1-\delta}}\theta(s)\,\txtd s\\
		&\quad+\int_{[t-\sigma]_+}^{t}M\frac{1}{(t-s)^{1-\delta}}\,\txtd s\,\theta(t)\\
		&\leq \txte^{-yt}c(t)+\left(\frac{N\Gamma(\gamma)}{(\epsilon y-x)^{\gamma}}+\frac{M\sigma^\delta}{\delta}\right)\theta(t)+\frac{M}{\sigma^{1-\delta}}\int_0^{t}\theta(s)\,\txtd s.
	\end{align*}}}
	Now we choose $\mu\in(0,1-\frac{N\Gamma(\gamma)}{(\epsilon y-x)^{\gamma}})$ and $\sigma=\left(\frac{\delta\mu}{M}\right)^{\frac{1}{\delta}}$. If we also use the monotonicity of the right-hand side, then we obtain
	\begin{align*}
		\theta(t)\leq \txte^{-yt}c(t)+\left(\frac{N\Gamma(\gamma)}{(\epsilon y-x)^{\gamma}}+\mu\right)\theta(t)+M^{\frac{1}{\delta}}(\delta\mu)^{\frac{\delta-1}{\delta}}\int_0^{t}\theta(s)\,\txtd s.
	\end{align*}
	Since $0<\frac{N\Gamma(\gamma)}{(\epsilon y-x)^{\gamma}}+\mu<1$ this yields
	\begin{align*}
		\theta(t)\leq \frac{1}{1-\mu-\frac{N\Gamma(\gamma)}{(\epsilon y-x)^{\gamma}}}\txte^{-yt}c(t)+\frac{M^{\frac{1}{\delta}}(\delta\mu)^{\frac{\delta-1}{\delta}}}{1-\mu-\frac{N\Gamma(\gamma)}{(\epsilon y-x)^{\gamma}}}\int_0^{t}\theta(s)\,\txtd s.
	\end{align*}
	Hence, the assertion follows from Lemma~\ref{Lemma:Gronwall_Version}.
\end{proof}

\section{Problems with Fast-Slow Systems in Infinite Dimensions}
\label{Section:Bates_Difficult}

Here we give some reasons why it is difficult to apply perturbation theorems for normally hyperbolic invariant manifolds in infinite dimensions such as the ones in \cite{Bates_Lu_Zeng_1998, Bates_Lu_Zeng_2008} to infinite-dimensional fast-slow systems.
\subsection{Problems with Small Perturbations}
\label{Subection:Bates_Difficult:Small_Perturbations}

In finite dimensions, the usual approach to show the existence of slow manifolds is to show that the flow of the fast-slow system on the fast time scale is a small perturbation of the flow generated by the fast subsystem. Then the existence of slow manifolds follows from the persistence of normally hyperbolic invariant manifolds under small perturbation. But even though such persistence results are also available in infinite dimensions (see for example \cite{Bates_Lu_Zeng_1998, Bates_Lu_Zeng_2008}), this approach does not work directly for many interesting infinite-dimensional examples. Consider for instance the following situation:\\
Let $X,Y$ be Banach spaces. Suppose that $A\colon X\supset D(A)\to X$ and $B\colon Y\supset D(B)\to Y$ are generators of $C_0$-semigroups $(T_A(t))_{t\geq0}\subset\mathcal{B}(X)$ and $(T_B(t))_{t\geq0}\subset\mathcal{B}(Y)$, respectively. Let further $L_1\in\mathcal{B}(Y,X)$ and $L_2\in\mathcal{B}(X,Y)$ be bounded linear operators. Then the operator
\[
    \begin{pmatrix} A & L_1 \\ \epsilon L_2 & \epsilon B\end{pmatrix} \colon X\times Y\supset D(A)\times D(B)\to X\times Y
\]
generates a $C_0$-semigroup $(T_{\epsilon}(t))_{t\geq0}$ for all $\epsilon\geq0$. Hence, for all $u_0\in X$, $v_0\in Y$ and all $\epsilon\geq0$ there is a unique solution to the fast-slow system
\begin{align}
   \begin{aligned}\label{Eq:FastSlow:SlowTime}
     \partial_{t}u^{\epsilon} &= Au^{\epsilon}+L_1v^{\epsilon},\\
    \partial_{t}v^{\epsilon} &= \epsilon Bv^{\epsilon} + \epsilon L_2u^{\epsilon},\\
    u^{\epsilon}(0)&=u_0,\quad v^{\epsilon}(0)= v_0
    \end{aligned}
\end{align}
on the fast time scale which is given by a semiflow
\[
    \begin{pmatrix} u^{\epsilon}(t) \\ v^{\epsilon}(t) \end{pmatrix} = T_{\epsilon}(t)\begin{pmatrix} u_0 \\ v_0 \end{pmatrix}.
\]
For the sake of argument, we assume that the embedding 
\[
 D(A)\times D(B)\to X\times Y
\]
is compact so that the intersection of the critical subspace
\[
  S_0:=\{(u,v)\in D(A)\times D(B): Au+L_1v=0\}
\]
with the ball $B(0,R)$ in $D(A)\times D(B)$ around $0$ with arbitrary radius $R>0$ is relatively compact in $X\times Y$. Note that this assumption is frequently satisfied for differential operators on bounded domains. We are thus in a similar situation as in finite dimensions. One would hope that one can apply the theorem given in the introduction of \cite{Bates_Lu_Zeng_1998} to $S_0\cap B(0,R)$. However, if one wants to apply this theorem in order to perturb the critical subspace $S_0$ for \eqref{Eq:FastSlow:SlowTime} with $\epsilon=0$ to a slow submanifold $S_{\epsilon}$ for \eqref{Eq:FastSlow:SlowTime} with $\epsilon>0$, one would - among other assumptions - need that
\begin{align}\label{Eq:OperatorNormConvergence}
 \| T_0(t)-T_{\epsilon}(t) \| _{\mathcal{B}(X\times Y)} \to 0\quad (\epsilon\to0).
\end{align}
for some $t>0$. In fact, one just needs
\[
  \| T_0(t)-T_{\epsilon}(t) \|_{C^1(N; X\times Y)}\to 0\quad (\epsilon\to0).
\]
for a suitable neighbourhood $N$ of $S_0\cap B(0,R)$. But since such a neighbourhood already contains a ball in $X\times Y$ around $0$ with small radius, this implies \eqref{Eq:OperatorNormConvergence} by linearity. However, \eqref{Eq:OperatorNormConvergence} is not satisfied if $B$ is an unbounded operator. This can be seen as follows:\\
One can use the variation of constants formula together with a standard version of Gronwall's inequality in order to show that there is a constant $C>0$ such that
\[
	\sup_{\epsilon,t\in[0,1]}\big(\|u^{\epsilon}(t)\|_{X}+\|v^{\epsilon}(t)\|_Y\big)\leq C (\|u_0\|_X+\|v_0\|_Y).
\]
Therefore, if \eqref{Eq:OperatorNormConvergence} holds then we have that
{\allowdisplaybreaks{
\begin{align*}
	0&=\lim_{\epsilon\to0}\sup_{\|(u_0,v_0)^T\|_{X\times Y}=1}\left\|\operatorname{pr}_Y(T_{\epsilon}(1)-T_0(1))(u_0, v_0)^T\right\|_Y\\
	&=\lim_{\epsilon\to0}\sup_{\|(u_0,v_0)^T\|_{X\times Y}=1}\left\|v^{\epsilon}(1)-v_0 \right\|_Y\\
	&=\lim_{\epsilon\to0}\sup_{\|(u_0,v_0)^T\|_{X\times Y}=1}\left\|(T_B(\epsilon)-\operatorname{id}_Y)v_0+\epsilon \int_0^1 T_B(\epsilon(1-s))L_2 u^{\epsilon}(s)\,\txtd s\right\|_Y\\
	&\geq \lim_{\epsilon\to0}\sup_{\|(u_0,v_0)^T\|_{X\times Y}=1}\left(\left\|(T_B(\epsilon)-\operatorname{id}_Y)v_0\right\|_Y\right)\\
	&\qquad\qquad-\lim_{\epsilon\to0}\sup_{\|(u_0,v_0)^T\|_{X\times Y}=1}\epsilon \left\|\int_0^1 T_B(\epsilon(1-s))L_2 u^{\epsilon}(s)\,\txtd s\right\|_Y\\
	&=\lim_{\epsilon\to0}\sup_{\|(u_0,v_0)^T\|_{X\times Y}=1}\left(\left\|(T_B(\epsilon)-\operatorname{id}_Y)v_0\right\|_Y\right).
\end{align*}}}
Hence, we have
\[
 \| T_B(\epsilon )- \operatorname{id}_Y \|_{\mathcal{B}(Y)}\to 0 \quad(\epsilon\to 0),
\]
i.e. the semigroup generated by $B$ is norm-continuous at $t=0$. But this holds if and only if $B$ is a bounded linear operator on $Y$, see for example \cite[Theorem I.3.7]{Engel_Nagel_2000}. Therefore, one can not apply \cite{Bates_Lu_Zeng_1998} directly to fast-slow systems, in which the dynamics of the slow variable are given by a partial differential equation.

\subsection{Problems with the Notion of Normal Hyperbolicity}\label{Subsection:Bates_Difficult:Normal_Hyperbolicity}

One of the central objects in classical Fenichel theory is the notion of a normally hyperbolic invariant manifold. The important properties of such a manifold $M$ are that it is invariant under the given (semi-) flow $(T^t)_{t\geq0}$ on the space $X$ and that for each $m\in M$ it admits a splitting
\[
	X=X_m^c\oplus X_m^s \oplus X_m^u
\]
such that
\begin{enumerate}[(i)]
	\item $X_m^c $ is the tangent space to $M$ at $m$.
	\item The splitting is invariant under the linearized flow $\txtD T^t(m)$.
	\item \label{NHIM} $\txtD T^t(m)\vert_{X_m^u}$ expands, $\txtD T^t(m)\vert_{X_m^s}$ contracts and both do it to a greater degree than $\txtD T^t(m)\vert_{X_m^c}$.
\end{enumerate}
Perturbation results for such normally hyperbolic invariant manifolds in infinite dimensions have been obtained in \cite{Bates_Lu_Zeng_1998}. Therein, Property \eqref{NHIM} includes on a formal level the condition
\begin{align}\label{Eq:Some_Condition_For_Normal_Hyperbolicity}
 \lambda\min\{1,\inf\{\|\txtD T^t(m)x^c\|_{X^c_m}:x^c\in X^c_m,\,|x^c|=1\}\}>\| \txtD T^t(m)|_{X^s_m}\|_{\mathcal{B}(X^s_m)}
\end{align}
for some $\lambda\in(0,1)$. However, if we consider the uncoupled, linear case of a fast-slow system, i.e. \eqref{Eq:FastSlow:SlowTime} with $L_1=0$ and $L_2=0$, then the center direction $X^c_m$ on the critical manifold will be given by
\[
 X^c_m=\{(x,y)\in X\times Y: Ax=0\}\supset\{(x,y)\in X\times Y: x=0\}.
\]
Thus, if $B$ is a standard parabolic operator as the Laplacian $\Delta$ on $L_p(\R^d)$ or the Dirichlet Laplacian $\Delta_D$ on $L_p(\mathcal{O})$ with $\mathcal{O}$ being a smooth domain, then the left hand side of \eqref{Eq:Some_Condition_For_Normal_Hyperbolicity} equals to $0$ so that normal hyperbolicity in the sense of \cite[Page 11]{Bates_Lu_Zeng_1998} can not be satisfied.

\subsection{Problems with the Splitting in Fast and Slow Time}\label{Subsection:Bates_Difficult:Fast_Slow_Times}

In infinite dimensions, one has to be careful with the interpretation of the notion ``fast-slow system''. Many interesting cases can (locally) be written as
\begin{align}\begin{aligned}\label{Eq:Fast_Slow_Explain_Normal_Hyperbolicity}
 \epsilon\partial_t u^{\epsilon}&=Au^{\epsilon}+ f(u^{\epsilon},v^{\epsilon}),\\
 \partial_t v^{\epsilon}&=Bv^{\epsilon}+ g(u^{\epsilon},v^{\epsilon}),\\
 u^{\epsilon}(0)&=u_0,\;v^{\epsilon}(0)=v_0,
 \end{aligned}
\end{align}
where in infinite dimensions the operators $A$ and $B$ are unbounded operators on the Banach spaces $X$ and $Y$, the Lipschitz continuous nonlinearities $f,g$ have Lipschitz constants which are not too large and $u_0,v_0$ are certain initial conditions; note that in many examples one may cut off the nonlinearity to make it Lipschitz due to invariant regions~\cite{Smoller} or due to global dissipation~\cite{Temam,Robinson1}.

Already in finite dimensions, the speed of evolution of the fast variable can only be considered faster than the one of the slow variable if they are related to their norms. Obviously, if $\|v_0\|_Y$ is very large, then $v^{\epsilon}(t)$ may change quickly compared to $u^{\epsilon}(t)$, even if $\epsilon$ is very small. However, in infinite dimensions $\|\cdot\|_X$ and $\|\cdot\|_Y$ may not be suitable for such a comparison for several reasons. First of all, unlike in finite dimensions, not all norms are equivalent and thus, comparing $\|\cdot\|_X$ and $\|\cdot\|_Y$ might not be very meaningful. But even if ones takes $(X,\|\cdot\|_X)=(Y,\|\cdot\|_Y)$, one may run into difficulties. For the sake of argument, we assume for the moment that there is no coupling, i.e. $f=0$ and $g=0$. Since $B$ is unbounded in many interesting cases, we may take $u_0\in D(A)$ with $\|u_0 \|_{X}=1$ and $v_0\in Y$ with $\|v_0\|_Y=1$ such that $\|Bv_0\|_Y> \epsilon^{-1}\|Au_0\|_X$. Then we have
\[
 \| \partial_t u^{\epsilon}(0) \|_{X}=\epsilon^{-1}\|Au_0\|_X<\|Bv_0\|_Y=\| \partial_t v^{\epsilon}(0) \|_{Y}.
\]
Therefore, one could argue that $v^{\epsilon}(t)$ is faster around $t=0$ than $u^{\epsilon}(t)$, even though it is called ``slow variable''. Note that this argument breaks down if one takes $u_0$ and $v_0$ to have graph norms of the same size, i.e. $\|u_0\|_{D(A)}=\|v_0\|_{D(B)}=1$. But then we have to problem the other way round: $\| \partial_t v^{\epsilon}(0) \|_{Y}$ might be smaller than $\| \partial_t u^{\epsilon}(0) \|_{X}$ only because $\|v_0\|_Y$ is much smaller than $\|u_0\|_X$. In order to illustrate this, let us consider an example:

\begin{example}
 We take $X=L_2(\R^d)$, $Y=H^{-2}(\R^d)$, $A=\Delta-1$ with domain $H^2(\R^d)$ and $B=\Delta-1$ with domain $L_2(\R^d)$. Again, we take $f=0$ and $g=0$ so that we obtain the system
 \begin{align*}
 \epsilon\partial_t u^{\epsilon}&=(\Delta-1)u^{\epsilon},\\
 \partial_t v^{\epsilon}&=(\Delta-1)v^{\epsilon},\\
 u^{\epsilon}(0)&=u_0,\;v^{\epsilon}(0)=v_0.
\end{align*}
Now, we take $u_0:=\mathscr{F}^{-1}[\xi\mapsto \frac{1}{1+|\xi|^2}\mathbbm{1}_{[0,1]^d}(\xi)]$ and $v_0:=\mathscr{F}^{-1}[\xi\mapsto \mathbbm{1}_{[0,1]^d}(\xi-\xi_0)]$ for a certain $\xi_0\in\R^d$. Then we have
\begin{align*}
 \|u_0\|_{D(A)}&=\|u_0\|_{L_2(\R^d)} +\| (\Delta-1) u_0\|_{L_2(\R^d)}\\
 &\eqsim \|\mathscr{F}^{-1}(1+|\xi|^2)\mathscr{F}u_0\|_{L_2(\R^d)}=\|\mathbbm{1}_{[0,1]^d}\|_{L_2(\R^d)}=1
\end{align*}
and
\begin{align*}
 \|v_0\|_{D(B)}&=\|v_0\|_{H^{-2}(\R^d)} +\| (\Delta-1) v_0\|_{H^{-2}(\R^d)}\\
 &\eqsim \|v_0\|_{L_2(\R^d)}=\|\mathbbm{1}_{[0,1]^d}(\cdot-\xi_0)\|_{L_2(\R^d)}=1.
\end{align*}
But it holds that
\begin{align*}
 \| u^{\epsilon}(t)\|_{L_{2}(\R^d)}=\|\mathscr{F}^{-1}\txte^{-\epsilon^{-1}(1+|\xi|^2)t}\mathscr{F}u_0\|_{L_2(\R^d)}\geq \txte^{-2\epsilon^{-1}t}\|u_0\|_{L_2(\R^d)}
\end{align*}
and
\begin{align*}
 \| v^{\epsilon}(t)\|_{H^{-2}(\R^d)}=\|\mathscr{F}^{-1}\txte^{-(1+|\xi|^2)t}\mathscr{F}v_0\|_{H^{-2}(\R^d)}\leq \txte^{-|\xi_0|^2 t} \|v_0\|_{H^{-2}(\R^d)}.
\end{align*}
Hence, $v^{\epsilon}(t)$ decays faster in relation to $\|v_0\|_{H^{-2}(\R^d)}$ than $u^{\epsilon}(t)$ in relation to $\|u_0\|_{L_2(\R^d)}$ if $|\xi_0|^2>2\epsilon^{-1}$, even though $\|u_0\|_{D(A)}=\|v_0\|_{D(B)}=1$.
\end{example}

We also want to point out that norms can be a bad indicator of different time scales in a system. Suppose that $B$ generates a unitary group $(\txte^{tB})_{t\in\R}$ on a Hilbert space $Y$ and $A$ generates an exponentially stable $C_0$-semigroup of contractions $(\txte^{tA})_{t\geq0}$ on $X$. Since $(\txte^{tB})_{t\in\R}$ is a family of isometric isomorphisms on $Y$, we obviously have that
\[
 1=\|\txte^{tB}v_0\|_Y>\| \txte^{\epsilon^{-1}tA}u_0\|_X
\]
for all choices of $t>0$, $v_0\in Y$ with $\|v_0\|_Y=1$ and $u_0\in X$ with $\|u_0\|_X=1$. But still, the trajectories of $(\txte^{tB})_{t\in\R}$ can have changes which are much faster than the exponential decay caused by $(\txte^{\epsilon^{-1}tA})_{t\geq0}$ for certain initial values. Take for example $B=\frac{d}{dx}$ on $H^{-1}(\R)$ with domain $L_2(\R)$. The corresponding group is given by the family of shifts $\txte^{tB}v=v(\cdot+t)$. If we take $v_k=\sqrt{k}\mathbbm{1}_{[0,k^{-1}]}$, then we have
\[
 \|v_{k}\|_{L_2(\R)}=1,\quad \| \txte^{k^{-1}B}v_k-v_k \|_{L_2(\R)}=\sqrt{2}.
\]
Thus, no matter how small $|t|$ is, there will always be an initial value $v_0$ with $\|v_0\|_{L_2(\R)}=1$ such that $\txte^{tB}v_0$ and $v_0$ have a distance of $\sqrt{2}$.\\ \text{ }\\
In principle, the fact that small $\epsilon$ does not provide an intuitive splitting in fast and slow time does not necessarily mean that carrying over the results from the finite to the infinite dimensional setting has to cause problems. However, it shows that both cases are different not only from a technical but also from a conceptual point of view. Looking at the above examples one could even discuss whether using the terminology ``fast-slow system'' is the most adequate in infinite dimensions as one cannot immediately spot the scale separation from a standard form but we shall nevertheless still use the finite-dimensional terminology as one can then formally refer to the two evolution equations for $u^\epsilon$ and $v^\epsilon$ more easily. 


\section{General Fast-Slow Systems in Infinite Dimensions}
\label{sec:genfs}

In Section~\ref{Subsection:Bates_Difficult:Normal_Hyperbolicity} we have seen that the classical notion of normal hyperbolicity is very restrictive in infinite dimensions. Unfortunately, it is not known if or how the Lyapunov-Perron method or Hadamard's graph transform can be carried out without this condition and thus, slow manifolds have not been constructed in a general infinite-dimensional setting so far. The main results of this section, Theorem~\ref{Thm:Original_and_Modified_Flow_Close} and Corollary~\ref{Cor:Original_and_Slow_Flow_Close}, show that even without the construction of slow manifolds, one can consider the slow flow as a good approximation of the semiflow generated by the fast-slow system. In order to derive these results, we need a weaker version of normal hyperbolicity. The idea behind this condition is that solutions of the fast equation
\[
	\epsilon\partial_tu^{\epsilon}=Au^{\epsilon}+f(u^{\epsilon},v^{\epsilon})
\]
should decay unless the contribution of the slow variable $v^{\epsilon}$ prevents them from doing so. This could be formulated in terms of conditions on the spectrum of  $A+\txtD_xf(x,y)$ or, as we do it later, by the estimate \eqref{Eq:Weak_Normal_Hyperbolicity}. For finite-dimensional fast-slow systems, requiring the spectrum of $A+\txtD_xf(x,y)$ to have an empty intersection with the imaginary axis is equivalent to normal hyperbolicity of the critical manifold. But in infinite dimensions this is clearly not the case, since Section~\ref{Subsection:Bates_Difficult:Normal_Hyperbolicity} shows that classical normal hyperbolicity crucially depends on the operator in the slow variable.\\
Altogether, one could summarize that in this section we derive weaker results under weaker conditions than classical Fenichel theory. In Section~\ref{sec:slowfinal} we will then introduce a suitable stronger notion of normal hyperbolicity in infinite dimensions which will suffice to construct slow manifolds. However, this stronger notion will be more restrictive again and there are examples in which we are still forced to rely on the results of Section~\ref{sec:genfs}.

\subsection{The Fast Equation}\label{Section:General_Fast_Slow:Fast_Equation}
First, we study the equation
\begin{align}\begin{aligned}\label{Eq:Fast_Equation_Nonlinear_IVP}
 \epsilon\partial_t u^{\epsilon}(t)&=Au^{\epsilon}(t) + f(t,u^{\epsilon}(t))\quad(t\in[0,T]),\\
 u^{\epsilon}(0)&=u_0,
 \end{aligned}
\end{align}
under the following assumptions:
\begin{itemize}
	\item $\epsilon\geq0$, $T>0$ are parameters and $u_0\in X_1:=D(A)$ an initial value which satisfies $0=Au_0+f(0,u_0)$ if $\epsilon=0$.
	\item The operator $A\colon X\supset D(A)\to X$ is a closed linear operator on the Banach space $X$ with $D(A)$ being dense in $X$ and with $0\in\rho(A)$. It generates the $C_0$-semigroup $(\txte^{tA})_{t\geq0}\subset\mathcal{B}(X)$.
	\item We write $(\tilde{X}_{\alpha},A_{\alpha})_{\alpha\in[-1,\infty)}$ for the interpolation-extrapolation scale generated by $(X,A)$ and $(X_{\alpha})_{\alpha\in[-1,\infty)}$ for a scale of Banach spaces such that the norms $\|\cdot\|_{X_\alpha}$ and $\|\cdot\|_{\tilde{X}_\alpha}$ are equivalent. Moreover, we take constants $C_A,M_A>0$, $\omega_A\in\R$ such that
	\[
 \|\txte^{tA}\|_{\mathcal{B}(X_1)}\leq M_A\txte^{\omega_A t},\quad  \|\txte^{tA}\|_{\mathcal{B}(X_{\gamma},X_1)}\leq C_At^{\gamma-1}\txte^{\omega_A t}\quad(t>0),
\]
where $\gamma\in(0,1]$ if $(\txte^{tA})_{t\geq0}\subset\mathcal{B}(X)$ is holomorphic and $\gamma=1$ in the general case.
	\item Take again  $\gamma\in(0,1]$ if $(\txte^{tA})_{t\geq0}\subset\mathcal{B}(X)$ is holomorphic and $\gamma=1$ in the general case. Let $\delta\in[1-\gamma,1]$. The nonlinearity $f\colon [0,\infty)\times X_{\delta} \to X$ is continuous and there is an $L_f>0$ such that
	\begin{align*}
		\| f(t,x_1)-f(t,x_2) \|_{X_{\gamma}}&\leq L_f \|x_1-x_2\|_{X_1},\\
		\| f(\cdot,u_1)-f(\cdot,u_2) \|_{C^1([0,t];X_{\delta-1})}&\leq L_f \|u_1-u_2\|_{C^1([0,t];X_{\delta})},
	\end{align*}
	for all $t\in[0,T]$, $x_1,x_2\in X_1$ and $u_1,u_2\in C^1([0,T];X_{\delta})$. Here we assume that $f(t,x)\in X_{\gamma}$ for $(t,x)\in[0,T]\times X_1$ and $f(\cdot,u)\in C^1([0,T];X_{\delta-1})$ for $u\in C^1([0,T];X_{\delta})$.
	\item We define $\omega_f:=\omega_A+(2C_AL_{f})^{\frac{1}{\gamma}} (\frac{1}{\gamma})^{\frac{1-\gamma}{\gamma}}$ if $\gamma\in(0,1)$ and take $\omega_f>\omega_A+C_AL_F$ if $\gamma=1$. According to Remark~\ref{Rem:Gronwall_Specific} the former definition will not be optimal in most cases, but for the sake of simplicity, we make this choice. However, as the optimal choice for $\gamma=1$ has a nice representation, we explicitely mention this case.
\end{itemize}
We work with these assumptions throughout this subsection.
\begin{remark}
	Formally, one has to distinguish the different operators $A_{\alpha}$ and the corresponding semigroups $(\txte^{tA_{\alpha}})_{t\geq0}$ for different values of $\alpha\in[-1,\infty)$. However, the difference is not essential for us. So we will in our notation just write $A$ and $(\txte^{tA})_{t\geq0}$ no matter on which $X_{\alpha}$ we consider them.
\end{remark}
\begin{proposition} \label{Prop:Nonlinear_IVP_Existence}
\begin{enumerate}[(a)] 
	\item \label{Prop:Nonlinear_IVP_Existence:epsilon=0} Assume that $L_f\|A^{-1}\|_{\mathcal{B}(X_{\delta-1},X_{\delta})}<1$. Then Equation \eqref{Eq:Fast_Equation_Nonlinear_IVP} with $\epsilon=0$ has a unique solution $u^0\in C^1([0,T];X_{\delta})$.
	\item \label{Prop:Nonlinear_IVP_Existence:epsilon>0} Equation \eqref{Eq:Fast_Equation_Nonlinear_IVP} with $\epsilon>0$ has a unique strict solution $u^{\epsilon}$, i.e. a solution $u^{\epsilon}\in C^1([0,\infty);X)\cap C([0,\infty);X_1)$ which satisfies \eqref{Eq:Fast_Equation_Nonlinear_IVP} with $\epsilon>0$ for all $t\in[0,\infty)$.
	\end{enumerate}
\end{proposition}
\begin{proof}
\begin{enumerate}[(a)]
	\item Our assumptions imply that 
		\[
			\mathscr{L}\colon C^1([0,T];X_{\delta})\to C^1([0,T];X_{\delta}), u\mapsto -A^{-1}f(\cdot,u)
		\]
		is a contraction. Since $C^1([0,T];X_{\delta})$ is a Banach space, the assertion follows from Banach's fixed point theorem.
	\item For $\eta\in\R$ let $C_b([0,\infty),\txte^{\epsilon^{-1}\eta t};X_1)$ be the space of all $u\in C([0,\infty);X_1)$ such that
	\[
		\|u\|_{C_b([0,\infty),\txte^{\epsilon^{-1}\eta t};X_1)}:=\sup_{t\geq0} \txte^{-\epsilon^{-1}\eta t}\|u(t)\|_{X_1}<\infty.
	\]
	We show that the operator
	\[
		\mathscr{L}(u):= \txte^{\epsilon^{-1}tA}u_0+\epsilon^{-1}\int_0^t \txte^{\epsilon^{-1}(t-s)A}f(s,u(s))\,\txtd s
	\]
	has a unique fixed point in $C_b([0,\infty),\txte^{\eta t};X_1)$ for $\eta$ large enough. By our assumptions it holds for $\eta>\omega_A$ that
	\begin{align*}
		&\quad\|\mathscr{L}(u_1)-\mathscr{L}(u_2)\|_{C_b([0,\infty),\txte^{\epsilon^{-1}\eta t};X_1)}\\
		&=\sup_{t\geq0}\txte^{-\epsilon^{-1}\eta t}\left\| \epsilon^{-1}\int_0^t \txte^{\epsilon^{-1}(t-s)A}(f(s,u_1(s))-f(s,u_2(s))\,\txtd s\right\|_{X_1}\\
		&\leq\sup_{t\geq0}L_fC_A\int_0^t\frac{\txte^{\epsilon^{-1}(t-s)(\omega_A-\eta)}}{(t-s)^{1-\gamma}\epsilon^{\gamma}}\,\txtd s\|u_1-u_2\|_{C_b([0,\infty),\txte^{\epsilon^{-1}\eta t};X_1)}\\
		&\leq \frac{L_f C_A\Gamma(\gamma)}{(\eta-\omega_A)^{\gamma}}\|u_1-u_2\|_{C_b([0,\infty),\txte^{\epsilon^{-1}\eta t};X_1)},
	\end{align*}
	where $\Gamma$ denotes the gamma function. If even $\eta>(L_f C_A\Gamma(\gamma))^{1/\gamma}-\omega_A$, then $\mathscr{L}$ is a contraction. By Banach's fixed point theorem, it follows that $\mathscr{L}$ has a unique fixed point in $C_b([0,\infty),\txte^{\epsilon^{-1}\eta t};X_1)$. Let $u^{\epsilon}$ be this fixed point. Then we have that 
	\[
		u^{\epsilon}(t)=\txte^{\epsilon^{-1}tA}u_0+\epsilon^{-1}\int_{0}^t\txte^{\epsilon^{-1}(t-s)A}f(s,u^{\epsilon}(s))\,\txtd s
	\]
	and which in turn implies that 
	\[
		u^{\epsilon}(t)=u_0+\epsilon^{-1}A\int_0^tu^{\epsilon}(s)\,\txtd s+\epsilon^{-1}\int_{0}^tf(s,u^{\epsilon}(s))\,\txtd s\quad(t\in[0,\infty)),
	\]
	see for example \cite[Proposition 4.1.5]{Lunardi_1995}. Hence, it follows that for all $t\geq0$ we have that 
	\begin{align*}
		\lim_{h\to 0}\frac{u^{\epsilon}(t+h)-u^{\epsilon}(t)}{h}&=\lim_{h\to 0}\frac{1}{h}\left[\int_t^{t+h}\epsilon^{-1}Au^{\epsilon}(s)\,\txtd s+\epsilon^{-1}\int_{t}^{t+h}f(s,u^{\epsilon}(s))\,\txtd s\right]\\
		&=\epsilon^{-1}Au^{\epsilon}(t)+\epsilon^{-1}f(t,u^{\epsilon}(t)),
	\end{align*}
	where to convergence holds in $X$ as $Au^{\epsilon},f(\cdot,u^{\epsilon})\in C([0,\infty);X)$.
	This shows the assertion.
\end{enumerate}
\end{proof}

\begin{remark}\label{Rem:2nd_Lipschitz_unnecessary}
	Note that in the proof of Proposition~\ref{Prop:Nonlinear_IVP_Existence}~\eqref{Prop:Nonlinear_IVP_Existence:epsilon>0} we did not use the estimate 
	\[
		\| f(\cdot,u_1)-f(\cdot,u_2) \|_{C^1([0,T];X_{\delta-1})}\leq L_f \|u_1-u_2\|_{C^1([0,T];X_{\delta})}\;\;\;(u_1,u_2\in C^1([0,T];X_{\delta})),
	\]
	which we assumed for $f$ to hold.
\end{remark}

\begin{proposition}\label{Prop:A_Priori_Estimate_Fast_Equation} Consider the situation of Proposition \ref{Prop:Nonlinear_IVP_Existence}.
\begin{enumerate}[(a)]
\item \label{Prop:A_Priori_Estimate_Fast_Equation:epsilon=0} Suppose that $L_f\|A^{-1}\|_{\mathcal{B}(X_{\delta-1},X_{\delta})}<1$. Let $\epsilon=0$ and let $u^0$ be the solution of \eqref{Eq:Fast_Equation_Nonlinear_IVP} from Proposition~\ref{Prop:Nonlinear_IVP_Existence}~(\ref{Prop:Nonlinear_IVP_Existence:epsilon=0}). Then we have the estimate
\begin{align*}
	\|u^0\|_{C^1([0,T];X_\delta)}\leq \frac{\|A^{-1}\|_{\mathcal{B}(X_{\delta-1},X_{\delta})}}{1-L_f\|A^{-1}\|_{\mathcal{B}(X_{\delta-1},X_{\delta})}}\|f(\cdot,0)\|_{C^1([0,T];X_{\delta-1})}.
\end{align*}
\item \label{Prop:A_Priori_Estimate_Fast_Equation:epsilon>0}Let $\epsilon>0$ and $\eta>\omega_A+C_AL_f(\frac{\gamma}{2L_fC_A})^{\frac{\gamma-1}{\gamma}}$. Then for all $t\geq0$ we have the estimate
\begin{align*}
 \|u^{\epsilon}(t)&\|_{X_1}\leq 2M_A\txte^{\epsilon^{-1}\omega_f t}\|u_0\|_{X_1}\\
 &+2C_A\left(\frac{\txte^\gamma}{\gamma^{1-\gamma}}+\Gamma(\gamma)\left|\frac{\eta-\omega_A}{\eta-\omega_f}\right|^{1-\gamma}\right)\frac{\|\txte^{\epsilon^{-1}\eta(t-\,\cdot\,)}f(\,\cdot\,,0)\|_{L_{\infty}([0,t];X_{\gamma})}}{(\eta-\omega_f)^\gamma},
\end{align*}
where $u^{\epsilon}$ denotes the solution of \eqref{Eq:Fast_Equation_Nonlinear_IVP} from Proposition~\ref{Prop:Nonlinear_IVP_Existence}~(\ref{Prop:Nonlinear_IVP_Existence:epsilon>0}).
\end{enumerate}
\end{proposition}
\begin{proof}
\begin{enumerate}[(a)]
 \item The assertion follows from 
 	\begin{align*}
 		\|u^0&\|_{C^1([0,T];X_\delta)}=\|A^{-1}f(\cdot,u^0)\|_{C^1([0,T];X_\delta)}\\
 		&\leq \|A^{-1}\|_{\mathcal{B}(X_{\delta-1},X_{\delta})} \|f(\cdot,u^0)\|_{C^1([0,T];X_{\delta-1})}\\
 		&\leq\|A^{-1}\|_{\mathcal{B}(X_{\delta-1},X_{\delta})}\big(\|f(\cdot,u^0)-f(\cdot,0)\|_{C^1([0,T];X_{\delta-1})}+\|f(\cdot,0)\|_{C^1([0,T];X_{\delta-1})}\big)\\
 		&\leq\|A^{-1}\|_{\mathcal{B}(X_{\delta-1},X_{\delta})}\big(L_f\|u^0\|_{C^1([0,T];X_{\delta})}+\|f(\cdot,0)\|_{C^1([0,T];X_{\delta-1})}\big).
 	\end{align*}
 \item In a first step we assume that $\omega_A+C_AL_f(\frac{\gamma}{2L_fC_A})^{\frac{\gamma-1}{\gamma}}<\eta=0$. For the solution of \eqref{Eq:Fast_Equation_Nonlinear_IVP} we have the implicit solution formula
 \begin{align*}
  u^{\epsilon}(t)&= \txte^{\epsilon^{-1}tA}u_0+\epsilon^{-1}\int_0^t \txte^{\epsilon^{-1}(t-s)A}f(s,0)\,\txtd s\\
  &\quad+\epsilon^{-1}\int_0^t \txte^{\epsilon^{-1}(t-s)A}(f(s,u^{\epsilon}(s))-f(s,0))\,\txtd s.
 \end{align*}
Therefore, we obtain
{\allowdisplaybreaks{
\begin{align*}
    \|u^{\epsilon}(t)\|_{X_1}&\leq \|\txte^{\epsilon^{-1}tA} \|_{\mathcal{B}(X_1)}  \|u_0\|_{X_1}+\epsilon^{-1}\int_0^t \|\txte^{\epsilon^{-1}(t-s)A}\|_{\mathcal{B}(X_{\gamma},X_1)}\|f(s,0)\|_{X_{\gamma}}\,\txtd s\\
    &\quad+ L_f\epsilon^{-1}\int_0^t \|\txte^{\epsilon^{-1}(t-s)A}\|_{\mathcal{B}(X_{\gamma},X_1)}\|u^{\epsilon}(s)\|_{X_1}\,\txtd s\\
    &\leq M_A\txte^{\epsilon^{-1}\omega_A t}  \|u_0\|_{X_1} + C_A\int_{0}^t\frac{\txte^{\epsilon^{-1}\omega_A(t-s)}}{\epsilon^{\gamma}(t-s)^{1-\gamma}}\,\txtd s\|f(\,\cdot\,,0)\|_{L_{\infty}([0,t];X_{\gamma})}\\
    &\qquad+ C_AL_f\int_0^t \frac{\txte^{\epsilon^{-1}\omega_A(t-s)}}{\epsilon^{\gamma}(t-s)^{1-\gamma}}\|u^{\epsilon}(s)\|_{X_1}\,\txtd s\\
        &= M_A\txte^{\epsilon^{-1}\omega_A t}  \|u_0\|_{X_1} + C_A\int_{0}^t\frac{\txte^{\epsilon^{-1}\omega_A s}}{\epsilon^{\gamma}s^{1-\gamma}}\,\txtd s\|f(\,\cdot\,,0)\|_{L_{\infty}([0,t];X_{\gamma})}\\
    &\qquad+ C_AL_f\int_0^t \frac{\txte^{\epsilon^{-1}\omega_A(t-s)}}{\epsilon^{\gamma}(t-s)^{1-\gamma}}\|u^{\epsilon}(s)\|_{X_1}\,\txtd s.
\end{align*}}}
Now we choose $t_0\geq t$ and apply Lemma~\ref{Lemma:Gronwall_Specific} with $p=2$ together with Corollary~\ref{Cor:Incomplete_Gamma_1}. If $\gamma=1$, then we apply Lemma~\ref{Lemma:Gronwall_Specific} with $p$ close to $1$. Note that
\[
	t\mapsto \txte^{-\epsilon^{-1}\omega_A t}\int_{0}^t\frac{\txte^{\epsilon^{-1}\omega_A s}}{\epsilon^{\gamma}s^{1-\gamma}}\,\txtd s
\]
is non-decreasing since $\omega_A<0$. We get
\begin{align*}
 \|u^{\epsilon}(t)&\|_{X_1}\leq 2M_A\txte^{\epsilon^{-1}\omega_f t}\|u_0\|_{X_1}\\
 &+2C_A\left(\frac{\txte^\gamma}{\gamma^{1-\gamma}}+\Gamma(\gamma)\left|\frac{\omega_A}{\omega_f}\right|^{1-\gamma}\right)\frac{\|f(\,\cdot\,,0)\|_{L_{\infty}([0,t_0];X_{\gamma})}}{|\omega_f|^\gamma},
\end{align*}
Taking $t_0=t$ yields the assertion for $\omega_f=\omega_A+C_AL_f(\frac{\gamma}{2L_fC_A})^{\frac{\gamma-1}{\gamma}}<0$. For arbitrary $\omega_A+C_AL_f(\frac{\gamma}{2L_fC_A})^{\frac{\gamma-1}{\gamma}}<\eta$, we use the transformation $u_{\eta}^{\epsilon}(t):=\txte^{-\epsilon^{-1}\eta t}u^{\epsilon}(t)$. Then $u_{\eta}^{\epsilon}$ satisfies
\begin{align*}
 \epsilon\partial_t u^{\epsilon}_{\eta}(t)&=(A-\eta)u^{\epsilon}_{\eta}(t) + \txte^{-\epsilon^{-1}\eta t}f(t,\txte^{\epsilon^{-1}\eta t}u^{\epsilon}_{\eta}(t))\quad(t\geq0),\\
 u^{\epsilon}_{\eta}(0)&=u_0.
\end{align*}
Our previous argument thus implies
\begin{align*}
 \|u^{\epsilon}_{\eta}(t)&\|_{X_1}\leq 2M_A\txte^{\epsilon^{-1}(\omega_f-\eta) t}\|u_0\|_{X_1}\\
 &+2C_A\left(\frac{\txte^\gamma}{\gamma^{1-\gamma}}+\Gamma(\gamma)\left|\frac{\eta-\omega_A}{\eta-\omega_f}\right|^{1-\gamma}\right)\frac{\|\txte^{-\epsilon^{-1}\eta(\,\cdot\,)}f(\,\cdot\,,0)\|_{L_{\infty}([0,t];X_{\gamma})}}{(\eta-\omega_f)^\gamma}.
\end{align*}
Multiplying with $\txte^{\epsilon^{-1}\eta t}$ again yields the assertion.
\end{enumerate}
\end{proof}

\begin{proposition}\label{Prop:A_Priori_Different_Nonlinearities}
	Let $\tilde{f}\colon X_{\delta}\to X$ satisfy the same assumptions as $f$ and let $\tilde{u}^{\epsilon}$ be the solution of \eqref{Eq:Fast_Equation_Nonlinear_IVP} for $\epsilon>0$ with $f$ being replaced by $\tilde{f}$. Let further $\eta>\omega_A+C_AL_f(\frac{\gamma}{2L_fC_A})^{\frac{\gamma-1}{\gamma}}$. Then we have the estimate
	\begin{align*}
		\|u^{\epsilon}(t)-\tilde{u}^{\epsilon}(t)\|_{X_1}\leq &2C_A\left(\frac{\txte^{\gamma}}{\gamma^{1-\gamma}}+\Gamma(\gamma)\bigg|\frac{\eta-\omega_A}{\eta-\omega_f}\bigg|^{1-\gamma}\right)\\
		&\cdot\frac{\sup\limits_{0\leq s\leq t, x\in X_1} \txte^{\epsilon^{-1}\eta(t-s)}\|f(s,\txte^{\epsilon^{-1}\eta s}x)-\tilde{f}(s,\txte^{\epsilon^{-1}\eta s}x) \|_{X_\gamma}}{(\eta-\omega_f)^{\gamma}}.
	\end{align*}
\end{proposition}
\begin{proof}
	We only treat the case $\eta=0$. For the general case, one can use the same transformation as in the proof of Proposition~\ref{Prop:A_Priori_Estimate_Fast_Equation}~\eqref{Prop:A_Priori_Estimate_Fast_Equation:epsilon>0}. Variation of constants yields
	{\allowdisplaybreaks{
	 \begin{align*}
  \| u^{\epsilon}(t)-\tilde{u}^{\epsilon}(t) \|_{X_1}&\leq\left\| \epsilon^{-1} \int_0^t \txte^{\epsilon^{-1}(t-s)A} (f(s,\tilde{u}^{\epsilon}(s))-\tilde{f}(s,\tilde{u}^{\epsilon}(s)))\,\txtd s \right\|_{X_1}\\
  &\qquad +\left\| \epsilon^{-1} \int_0^t \txte^{\epsilon^{-1}(t-s)A} (f(s,u^{\epsilon}(s))-f(s,\tilde{u}^{\epsilon}(s)))\,\txtd s \right\|_{X_1}\\
  &\leq C_A\int_0^t \frac{\txte^{-\epsilon^{-1}\omega_A(t-s)}}{\epsilon^{\gamma}(t-s)^{1-\gamma}}\,\txtd s\sup_{0\leq r\leq t_0, x\in X_1} \| f(r,x)-\tilde{f}(r,x) \|_{X_\gamma}\\
  &\qquad +  C_AL_f\int_0^t \frac{\txte^{\epsilon^{-1}\omega_A(t-s)}}{\epsilon^{\gamma}(t-s)^{1-\gamma}}\| u^{\epsilon}(s)-\tilde{u}^{\epsilon}(s) \|_{X_1}\,\txtd s\\
  &\leq C_A\int_0^t \frac{\txte^{-\epsilon^{-1}\omega_A s}}{\epsilon^{\gamma}s^{1-\gamma}}\,\txtd s\sup_{0\leq r\leq t_0, x\in X_1} \| f(r,x)-\tilde{f}(r,x) \|_{X_\gamma}\\
  &\qquad +  C_AL_f\int_0^t \frac{\txte^{\epsilon^{-1}\omega_A(t-s)}}{\epsilon^{\gamma}(t-s)^{1-\gamma}}\| u^{\epsilon}(s)-\tilde{u}^{\epsilon}(s) \|_{X_1}\,\txtd s,
 \end{align*}}}
 where $t_0\geq t$. Applying Lemma~\ref{Lemma:Gronwall_Specific} with $p=2$ (or $p$ close to $1$ if $\gamma=1$ together with Corollary~\ref{Cor:Incomplete_Gamma_1} and taking $t_0=t$ yields the assertion.
\end{proof}

\subsection{A Modified Fast Equation}

Under the assumptions of Section~\ref{Section:General_Fast_Slow:Fast_Equation}, we now consider a modified fast equation
\begin{align}
    \begin{aligned}\label{Eq:Modified_Fast}
 \epsilon\partial_tu^{\epsilon,0}(t)&=Au^{\epsilon,0}(t)+ f(t,u^{\epsilon,0}(t))-\epsilon\partial_tA^{-1}f(t,u^0(t)),\\
 u^{\epsilon,0}(0)&=u_0.
 \end{aligned}
\end{align}
where $u^0$ denotes the solution of \eqref{Eq:Fast_Equation_Nonlinear_IVP} from Proposition~\ref{Prop:Nonlinear_IVP_Existence}~(\ref{Prop:Nonlinear_IVP_Existence:epsilon=0}). Since we work with $u^0$, we assume that $\|A^{-1}\|_{\mathcal{B}(X_{\delta-1},X_\delta)}L_f<1$ in this subsection. Even though it is not necessary for all the results, we will assume $\omega_A<\omega_f<0$ from now on.

\begin{lemma}\label{Lemma:Modified_Fast_Equation}
For all $u_0\in X_1$ and all $\epsilon>0$ there is a unique strict solution $$u^{\epsilon,0}\in C^{1}([0,\infty);X)\cap C([0,\infty);X_1)$$ of \eqref{Eq:Modified_Fast}.
\end{lemma}
\begin{proof}
 Let
\[
 f_{\epsilon}\colon [0,T]\times X_{\delta} \to X,\;(t,x)\mapsto f(t,x)-\epsilon\partial_t A^{-1}f(t,u^0(t)).
\]
Since $u^0\in C^1([0,T];X_{\delta})$ by Proposition~\ref{Prop:Nonlinear_IVP_Existence}~\eqref{Prop:Nonlinear_IVP_Existence:epsilon=0} and since $f$ maps $C^1([0,T];X_{\delta})$ to $C^1([0,T];X_{\delta-1})$, it follows that $\partial_t A^{-1}f(\cdot,u^0)\in C([0,T];X_{\delta})$ so that $f_{\epsilon}$ is well-defined. Moreover, we have
\[
	\| f_{\epsilon}(t,x_1)-f_{\epsilon}(t,x_2) \|_{X_\gamma} = \| f(t,x_1)-f(t,x_2) \|_{X_\gamma}\leq L_f \| x_1-x_2 \|_{X_1}
\]
for all $(t,x_1),(t,x_2)\in[0,T]\times X_1$. By Remark \ref{Rem:2nd_Lipschitz_unnecessary} this suffices to apply Proposition \ref{Prop:Nonlinear_IVP_Existence} with $f$ being replaced by $f_{\epsilon}$.
\end{proof}

\begin{proposition}\label{Prop:Modified_Fast_and_Time_Independent}
 Let $u^{\epsilon,0}$ be the solution of \eqref{Eq:Modified_Fast} with $\epsilon>0$ and the $u^0$ solution of \eqref{Eq:Fast_Equation_Nonlinear_IVP} with $\epsilon=0$. Then we have the estimate
\[
 \| u^{\epsilon,0}(t)-u^0(t) \|_{X_1}\leq 2 M_A \txte^{\epsilon^{-1}\omega_ft} \|u_0-u^0(0)\|_{X_1}.
\]
\end{proposition}
\begin{proof}
 Using variation of constants and integration by parts yields
 \begin{align*}
  &\quad\| u^{\epsilon,0}(t)-u^0(t) \|_{X_1} \\
  &  \leq \left\|\txte^{\epsilon^{-1}tA}u_0+\int_0^t\txte^{\epsilon^{-1}A(t-s)}\big[\epsilon^{-1}f(s,u^{\epsilon,0}(s))-\partial_sA^{-1}f(s,u^0(s))\big]\,\txtd s - u^0(t) \right\|_{X_1}\\
  &=\left\|\txte^{\epsilon^{-1}tA}(u_0-u^0(0))+\epsilon^{-1}\int_0^t\txte^{\epsilon^{-1}A(t-s)}\big[f(s,u^{\epsilon,0}(s))-f(s,u^0(s))\big]\,\txtd s  \right\|_{X_1}\\
  &\leq M_A\txte^{\epsilon^{-1}\omega_A t}\|u_0-u^0(0)\|_{X_1}+C_AL_{f}\int_0^t\frac{\txte^{\epsilon^{-1}\omega_A(t-s)}}{\epsilon^{\gamma}(t-s)^{1-\gamma}}\| u^{\epsilon,0}(s)-u^0(s) \|_{X_1}\,\txtd s
 \end{align*}
 Now, the assertion follows from Lemma \ref{Lemma:Gronwall_Specific}.
\end{proof}

\begin{proposition}\label{Prop:Fast_and_Modified_Fast}
 Suppose that $C_A$ is chosen such that additionally to the assumptions of Section~\ref{Section:General_Fast_Slow:Fast_Equation} we also have
\[
\|\txte^{tA}\|_{\mathcal{B}(X_{\delta},X_1)}\leq C_At^{\delta-1}\txte^{\omega_A t}\quad(t>0).
\]
Let $u^{\epsilon}$ be the solution of \eqref{Eq:Fast_Equation_Nonlinear_IVP} and $u^{\epsilon,0}$ the one of \eqref{Eq:Modified_Fast} for $\epsilon>0$. Then we have the estimate
\begin{align*}
	\|u^{\epsilon}(t)-u^{\epsilon,0}(t)\|_{X_1}&\leq \left(\frac{\txte^{\delta}}{\delta^{1-\delta}}+\Gamma(\delta)\bigg|\frac{\omega_A}{\omega_f}\bigg|^{1-\delta}\right)\frac{C_A \|A^{-1}\|_{\mathcal{B}(X_{\delta-1},X_\delta)}}{(1-L_f\|A^{-1}\|_{\mathcal{B}(X_{\delta-1},X_\delta)})}\\
	&\qquad\cdot\,\frac{\epsilon}{|\omega_f|^{\delta}}\|f(t,0)\|_{C^1_b([0,t];X_{\delta-1})}.
\end{align*}
\end{proposition}
\begin{proof}
 Using variation of constants and choosing $t_0\geq t$ yields that
 \begin{align*}
  \|u^{\epsilon}(t)&-u^{\epsilon,0}(t)\|_{X_1}=\left\|\epsilon^{-1}\int_0^t\txte^{\epsilon^{-1}(t-s)A}[f(s,u^{\epsilon}(s))-f(s,u^{\epsilon,0}(s))]\,\txtd s\right.\\
  &\qquad\qquad\left.-\int_0^t \txte^{\epsilon^{-1}(t-s)A}\partial_sA^{-1} f(s,u^0(s))\,\txtd s\right\|_{X_1}\\
  &\leq  C_A\epsilon \|A^{-1}\|_{\mathcal{B}(X_{\delta-1},X_\delta)}\|\partial_tf(\,\cdot,u^0)\|_{L_{\infty}([0,t_0];X_{\delta-1})} \int_0^t\frac{\txte^{\epsilon^{-1}\omega_A(t-s)}}{\epsilon^{\delta}(t-s)^{1-\delta}}\,\txtd s\\
  &\qquad\qquad+CL_f\int_0^t \frac{\txte^{-\epsilon^{-1}\omega_A(t-s)}}{\epsilon^{\gamma}(t-s)^{1-\gamma}}\|u^{\epsilon}(s)-u^{\epsilon,0}(s)\|_{X_1}\,\txtd s
 \end{align*}
Thus, a combination of Lemma~\ref{Lemma:Gronwall_Specific} and Corollary \ref{Cor:Incomplete_Gamma_1} shows that
\begin{align*}
	\|u^{\epsilon}(t)-u^{\epsilon,0}(t)\|_{X_1}&\leq \left(\frac{\txte^{\delta}}{\delta^{1-\delta}}+\Gamma(\delta)\bigg|\frac{\omega_A}{\omega_f}\bigg|^{1-\delta}\right)\frac{C_A\epsilon \|A^{-1}\|_{\mathcal{B}(X_{\delta-1},X_\delta)}}{\omega_f^{\delta}}\\
&\qquad\qquad\cdot\|\partial_tf(\,\cdot,u^0)\|_{L_{\infty}([0,t_0];X_{\delta-1})}
\end{align*}
 Moreover, it follows from Proposition~\ref{Prop:A_Priori_Estimate_Fast_Equation}~\eqref{Prop:A_Priori_Estimate_Fast_Equation:epsilon=0} that
\begin{align*}
 \|\partial_tf(\cdot,u^{0})&\|_{L_{\infty}([0,t_0];X_{\delta-1})}\leq \|f(\cdot,u^{0})\|_{C^1_b([0,t_0];X_{\delta-1})}\\
 &\leq \|f(\cdot,0) \|_{C^1_b([0,t_0];X_{\delta-1})}+ L_f\| u^0\|_{C^1_b([0,t_0];X_{\delta})}\\
 &\leq \frac{1}{1-\|A^{-1}\|_{\mathcal{B}(X_{\delta-1},X_{\delta})}L_f}\|f(t,0)\|_{C^1_b([0,t_0];X_{\delta-1})}
\end{align*}
so that
\begin{align*}
	\|u^{\epsilon}(t)-u^{\epsilon,0}(t)\|_{X_1}&\leq \left(\frac{\txte^{\delta}}{\delta^{1-\delta}}+\Gamma(\delta)\bigg|\frac{\omega}{\omega_f}\bigg|^{1-\delta}\right)\frac{C_A \|A^{-1}\|_{\mathcal{B}(X_{\delta-1},X_\delta)}}{(1-L_f\|A^{-1}\|_{\mathcal{B}(X_{\delta-1},X_\delta)})}\\
	&\qquad\cdot\,\frac{\epsilon}{|\omega_f|^{\delta}}\|f(t,0)\|_{C^1_b([0,t];X_{\delta-1})}.
\end{align*}
\end{proof}

\subsection{Well-posedness of the Full System} \label{Section:General_Fast_Slow:Full_System}

Now we consider the nonlinear fast-slow system
\begin{align}
\begin{aligned}\label{Eq:Nonlinear_Fast_Slow_System}
 \epsilon\partial_t u^{\epsilon}(t) &= Au^{\epsilon}(t)+ f(u^\epsilon(t),v^{\epsilon}(t)),\\
 \partial_t v^{\epsilon}(t) &= Bv^{\epsilon}(t)+g(u^\epsilon(t),v^{\epsilon}(t)),\\
 u^{\epsilon}(0)&=u_0,\quad v^{\epsilon}(0)=v_0.
 \end{aligned}\quad(t\in[0,T])
\end{align}
We assume that
\begin{enumerate}[(i)]
 \item $X,Y$ are Banach spaces, $\epsilon\geq0$, $T>0$ are parameters and $u_0\in X_1=D(A)$, $v_1\in Y_1=D(B)$ are initial values. If $\epsilon=0$, then $u_0$ has to satisfy $0=Au_0+f(u_0,v_0)$.
 \item The closed linear operator $A\colon X\supset D(A)\to X$ generates an exponentially stable $C_0$-semigroup $(\txte^{tA})_{t\geq0}\subset\mathcal{B}(X)$. The closed linear operator $B\colon Y\supset D(B)\to Y$ is the generator of a $C_0$-semigroup $(\txte^{tB})_{t\geq0}\subset\mathcal{B}(Y)$.
 \item The interpolation-extrapolation scales generated by $(X,A)$ and $(Y,B)$ are -- up to equivalence of norms for each fixed $\alpha\in[-1,\infty)$ -- given by $(X_{\alpha})_{\alpha\in[-1,\infty)}$ and $(Y_{\alpha})_{\alpha\in[-1,\infty)}$. If $0\notin\rho(B)$, then $(Y_{\alpha})_{\alpha\in[-1,\infty)}$ shall be equivalent to the interpolation-extrapolation scale generated by $B-\lambda$ for some $\lambda\in\rho(B)$. 
 \item Let $\gamma_X\in(0,1]$ if $(\txte^{tA})_{t\geq0}\subset\mathcal{B}(X)$ is holomorphic and $\gamma_X=1$ otherwise. In addition, we choose $\delta_X\in[1-\gamma_X,1]$. Let further $\delta_Y\in(0,1]$ if $(\txte^{tB})_{t\geq0}\subset\mathcal{B}(Y)$ is holomorphic and $\delta_Y=1$ otherwise. The nonlinearities $f\colon X_{\delta_X}\times Y_{1-\delta_X}\to X$ and $g\colon X_{1}\times Y_{1}\to Y_{\delta_Y}$ are continuous and there are constants $L_f,L_g>0$ such that with
 \begin{align*}
 	\|f(x_1,y_1)-f(x_2,y_2) \|_{\gamma_X}&\leq L_f\big(\|x_1-x_2\|_{X_1}+\|y_1-y_2\|_{Y_1}\big),\\
 	\|f(u_1,v_1)-f(u_2,v_2)\|_{C^1([0,t];X_{\delta_X-1})}&\leq L_f\big(\|u_1-u_2\|_{C^1([0,t];X_{\delta_X})}\\
 	&\qquad\qquad+\|v_1-v_2\|_{C^1([0,t];Y)}\big),\\
 	\|g(x_1,y_1)-g(x_2,y_2) \|_{\delta_Y}&\leq L_g\big(\|x_1-x_2\|_{X_1}+\|y_1-y_2\|_{Y_1}\big)
 \end{align*}
 for all $x_1,x_2\in X_1$, $y_1,y_2\in Y_1$, $u_1,u_2\in C^1([0,t];X_{\delta_X})$ and all $v_1,v_2\in C^1([0,t];Y)\cap C([0,t];Y_{1-\delta_X})$. Here, we assume that
 \begin{align*}
 f(x,y)\in X_{\gamma_X},\;g(x,y)\in Y_{\delta_Y}\quad\text{if}\quad(x,y)\in X_1\times Y_1
 \end{align*}
 as well as
 \begin{align*}
 f(u,v)\in C^1([0,t];X_{\delta_X-1})\quad\text{if}\quad(u,v)&\in C^1([0,t];X_{\delta_X}\times Y)\;\\
 &\text{and}\;v\in C([0,t];Y_{1-\delta_X}).
 \end{align*}
 \item We assume that $f(0,0)=0$ and $g(0,0)=0$.
 \item We choose constants $M_A,M_B,C_A,C_B>0$, $\omega_A<0$ and $\omega_B\in\R$ such that
 \begin{align*}
 	\|\txte^{tA}\|_{\mathcal{B}(X_1)}\leq &M_A \txte^{\omega_A t},\quad \|\txte^{tA}\|_{\mathcal{B}(X_{\gamma_X},X_1)}\leq C_At^{\gamma_X-1} \txte^{\omega_A t},\\
 	& \|\txte^{tA}\|_{\mathcal{B}(X_{\delta_X},X_1)}\leq C_At^{\delta_X-1} \txte^{\omega_A t}
 \end{align*}
 and
 \begin{align*}
  	\|\txte^{tB}\|_{\mathcal{B}(Y_1)}\leq M_B \txte^{\omega_B t},\quad \|\txte^{tB}\|_{\mathcal{B}(Y_{\delta_Y},Y_1)}\leq C_Bt^{\delta_Y-1} \txte^{\omega_B t}
 \end{align*}
 hold for all $t>0$.
 \item Again we define $\omega_f:=\omega_A+(2C_AL_{f})^{\frac{1}{\gamma_X}} (\frac{1}{\gamma_X})^{\frac{1-\gamma_X}{\gamma_X}}$ if $\gamma_X\in(0,1)$ and take $\omega_f>\omega_A+C_AL_F$ if $\gamma_X=1$. Even though it is not necessary for all the results, we will assume 
 \begin{align}\label{Eq:Weak_Normal_Hyperbolicity}\begin{aligned}
 	&\omega_f<0,\\L_f\max\{\|A^{-1}\|_{\mathcal{B}(X_{\gamma_X},X_{1})}&,\|A^{-1}\|_{\mathcal{B}(X_{\delta_X-1},X_{\delta_X})}\}<1
 	\end{aligned}
 \end{align}in the following. Note that $A^{-1}$ exists as a consequence of the Hille-Yosida theorem, since $A$ generates an exponentially stable $C_0$-semigroup. Recall that as described at the beginning of Section~\ref{sec:genfs} this is a weak version of normal hyperbolicity, as it ensures that solutions of the fast equation would decay exponentially if there was no influence of the slow variable $v^{\epsilon}$ in the fast equation.
\end{enumerate}

Note that assumption (v) can in practice very frequently be ensured locally by just moving the point of interest on the critical manifold via a coordinate transformation to the origin and using Taylor expansion, so it is not really a restriction. We work with all the above assumptions for the rest of this paper. Since we also assume global Lipschitz conditions on the nonlinearities, we obtain the following well-posedness results:

\begin{proposition}\label{Prop:Fast_Slow_Well-Posed}
	\begin{enumerate}[(a)]
		\item\label{Prop:Fast_Slow_Well-Posed:epsilon=0} Let $\epsilon=0$. Then \eqref{Eq:Nonlinear_Fast_Slow_System} has a unique strict solution 
		\[
			(u^0,v^0)\in C^1([0,T];X\times Y)\cap C([0,T];X_1\times Y_1).
		\]
		\item\label{Prop:Fast_Slow_Well-Posed:epsilon>0} Let $\epsilon>0$. Then \eqref{Eq:Nonlinear_Fast_Slow_System} has a unique strict solution 
		\[
			(u^{\epsilon},v^{\epsilon})\in C^1([0,T];X\times Y)\cap C([0,T];X_1\times Y_1).
		\]
	\end{enumerate}
\end{proposition}
\begin{proof}
	\begin{enumerate}[(a)]
		\item Let $y\in Y_{1}$. By assumption, it holds that 
		\[
			f_y\colon X_{\delta_X}\to X, x\mapsto f_y(x):=f(x,y)
		\]
		is continuous and satisfies
		\[
			\|f_y(x_1)-f_y(x_2)\|_{X_{\gamma_X}}=\|f(x_1,y)-f(x_2,y)\|_{X_{\gamma_X}}\leq L_f\|x_1-x_2\|_{X_1}.
		\]
		Since we assume $\|A^{-1}\|_{\mathcal{B}(X_{\gamma_X},X_1)} L_f<1$ it follows from Banach's fixed point theorem that there is a unique solution $x\in X_{1}$ of 
		\[
			0= Ax+f_y(x).
		\]
		In the following we write $h^0(y)$ for this solution. Given $y_1,y_2\in Y_{1}$ it holds that
		\begin{align*}
			\|h^0(y_1)-h^0(y_2)\|_{X_1}&=\|A^{-1}f(h^0(y_1),y_1)-A^{-1}f(h^0(y_2),y_2)\|_{X_1}\\
			&\leq L_f\|A^{-1}\|_{\mathcal{B}(X_{\gamma_X},X_1)}\big(\|h^0(y_1)-h^0(y_2)\|_{X_{1}}+\|y_1-y_2\|_{Y_1}\big)
		\end{align*}
		and thus
		\begin{align*}
			\|h^0(y_1)-h^0(y_2)\|_{X_1}&\leq \frac{1}{1-L_f\|A^{-1}\|_{\mathcal{B}(X_{\gamma_X},X_1)}}\|y_1-y_2\|_{Y_1}.
		\end{align*}
		Therefore, the mapping
		\[
			Y_{1}\to Y_{\delta_Y},\,y\mapsto g(h^0(y),y)
		\]
		is continuous. Moreover, we have the estimate
		\begin{align*}
			\|g(h^0(y_1),y_1)-g(h^0(y_2),y_2)\|_{Y_{\delta_Y}}\leq \left(\frac{L_g}{1-L_f\|A^{-1}\|_{\mathcal{B}(X_{\gamma_X},X_1)}}+L_g\right)\|y_1-y_2\|_{Y_1}.
		\end{align*}
		Therefore, it follows from Proposition~\ref{Prop:Nonlinear_IVP_Existence}~\eqref{Prop:Nonlinear_IVP_Existence:epsilon>0} together with Remark~\ref{Rem:2nd_Lipschitz_unnecessary} with $\delta=1$ and $\gamma=\delta_Y$ that there is a unique strict solution
		\[
			v^0\in C^1([0,T];Y)\cap C([0,T];Y_1)
		\]
		of the equation 
		\[
			\partial_tv^0(t)=Bv^0(t)+g(h^0(v^0(t)),v^0(t)),\quad v^0(0)=v_0.
		\]
		Now we take $u^0(t):=h^0(v^0(t))$, i.e. we have that
		\[
		u^0(t)=A^{-1}f(u^0(t),v^0(t)).
		\]
		Proposition~\ref{Prop:Nonlinear_IVP_Existence}~\eqref{Prop:Nonlinear_IVP_Existence:epsilon=0} shows that $u^0\in C^1([0,T];X_{\delta_X})\subset  C^1([0,T];X)$. Moreover, since $h^0\colon Y_1\to X_1$ is Lipschitz continuous, it follows that $u^0\in C([0,T];X_1)$. Altogether, it follows that
		\[
			(u^0,v^0)=(h^0(v^0),v^0)\in C^1([0,T];X\times Y)\cap C([0,T];X_1\times Y_1)
		\]
		is the unique solution of \eqref{Eq:Nonlinear_Fast_Slow_System} with $\epsilon=0$.
		\item The proof is similar to the one of Proposition~\ref{Prop:Nonlinear_IVP_Existence}~\eqref{Prop:Nonlinear_IVP_Existence:epsilon>0}. This time, for some $\eta\in\R$ we consider the space $C_b([0,\infty),\txte^{\eta t};X_1\times Y_1)$ of all $(u,v)\in C([0,\infty);X_1\times Y_1)$ such that
		\[
			\|(u,v)\|_{C_b([0,\infty),\txte^{\eta t};X_1\times Y_1)}:=\sup_{t\geq0} \txte^{-\eta t}\big(\|u(t)\|_{X_1}+\|v(t)\|_{Y_1}\big)<\infty.
		\]
		On this space, we define the operator $\mathscr{L}$ by
		\[
			[\mathscr{L}(u,v)](t):=\begin{pmatrix}
				\txte^{\epsilon^{-1}tA}u_0+\epsilon^{-1}\int_0^t \txte^{\epsilon^{-1}(t-s)A} f(u(s),v(s))\,\txtd s\\
				\txte^{tB}u_0+\int_0^t \txte^{(t-s)B} g(u(s),v(s))\,\txtd s.
			\end{pmatrix}
		\]
		We show that this operator is a contraction on $C_b([0,\infty),\txte^{\eta t};X_1\times Y_1)$ if $\eta$ is large enough. We have that
		\begin{align*}
			&\quad \sup_{t\geq0} \txte^{-\eta t}\epsilon^{-1}\left\|\int_0^t \txte^{\epsilon^{-1}(t-s)A} \big[f(u_1(s),v_1(s))-f(u_2(s),v_2(s))\big]\,\txtd s\right\|\\
			&\leq L_fC_A \sup_{t\geq0}\int_0^t\frac{\txte^{(t-s)(\epsilon^{-1}\omega_A-\eta)}}{\epsilon^{\gamma_{X}}(t-s)^{1-\gamma_X}}\,\txtd s \|(u_1,v_1)-(u_2,v_2)\|_{C_b([0,\infty),\txte^{\eta t};X_1\times Y_1)}\\
			&\leq\frac{L_f C_A\Gamma(\gamma_X)}{(\epsilon\eta-\omega_A)^{\gamma_X}}\|(u_1,v_1)-(u_2,v_2)\|_{C_b([0,\infty),\txte^{\eta t};X_1\times Y_1)}.
		\end{align*}
		Similarly, we have that
		\begin{align*}
			&\quad \sup_{t\geq0} \txte^{-\eta t}\left\|\int_0^t \txte^{(t-s)B} \big[g(u_1(s),v_1(s))-g(u_2(s),v_2(s))\big]\,\txtd s\right\|\\
			&\leq L_gC_B \sup_{t\geq0}\int_0^t\frac{\txte^{(t-s)(\omega_B-\eta)}}{\epsilon^{\delta_{Y}}(t-s)^{1-\delta_Y}}\,\txtd s \|(u_1,v_1)-(u_2,v_2)\|_{C_b([0,\infty),\txte^{\eta t};X_1\times Y_1)}\\
			&\leq\frac{L_g C_B\Gamma(\delta_Y)}{(\eta-\omega_B)^{\delta_Y}}\|(u_1,v_1)-(u_2,v_2)\|_{C_b([0,\infty),\txte^{\eta t};X_1\times Y_1)}.
		\end{align*}
		Therefore, we have that
		\begin{align*}
		&\|[\mathscr{L}(u,v)](t)\|_{C_b([0,\infty),\txte^{\eta t};X_1\times Y_1)}\\
		&\;\;\leq\left(\frac{L_f C_A\Gamma(\gamma_X)}{(\epsilon\eta-\omega_A)^{\gamma_X}}+\frac{L_g C_B\Gamma(\delta_Y)}{(\eta-\omega_B)^{\delta_Y}}\right)\|(u_1,v_1)-(u_2,v_2)\|_{C_b([0,\infty),\txte^{\eta t};X_1\times Y_1)}.
		\end{align*}
		In particular, if $\eta$ is large enough then $\mathscr{L}$ is a contraction. Thus, there is a unique fixed point $(u^{\epsilon},v^{\epsilon})\in C_b([0,\infty),\txte^{\eta t};X_1\times Y_1)$. By the same line of arguments as in the proof of Proposition~\ref{Prop:Nonlinear_IVP_Existence}~\eqref{Prop:Nonlinear_IVP_Existence:epsilon>0} it now follows that 
			\[
			(u^{\epsilon},v^{\epsilon})\in C^1([0,T];X\times Y)\cap C([0,T];X_1\times Y_1).
		\]
		and that it solves \eqref{Eq:Nonlinear_Fast_Slow_System} with $\epsilon>0$.
	\end{enumerate}
\end{proof}

\begin{remark}\label{Rem:Remarks_on_Flows}
	\begin{enumerate}[(a)]
		\item\label{Rem:Remarks_on_Flows:Critical_Manifold} In the proof of Propisition \eqref{Prop:Fast_Slow_Well-Posed} we introduced the mapping
		\[
			h^0\colon Y_1\to X_1,\;y\mapsto h^0(y),
		\]
		where $h^0(y)$ is the unique solution of 
		\[
			0=Ah^0(y)+f(h^0(y),y).
		\]
		In particular, this mapping describes the critical manifold $S_0$ over $Y_1$ by
		\[
			S_0:=\{(h^0(y),y):y\in Y_1\}\subset X\times Y.
		\]
		Note that Proposition~\ref{Prop:Nonlinear_IVP_Existence}~\eqref{Prop:Nonlinear_IVP_Existence:epsilon=0} shows that if $v^0\in C^1([0,T];Y)\cap C([0,T];Y_{1-\delta_X})$, then $h(v^0)\in C^1([0,T];X_{\delta_X})$.
		\item Since \eqref{Eq:Nonlinear_Fast_Slow_System} is autonomous, the solutions $(u^0,v^0)$ and $(u^{\epsilon},v^{\epsilon})$ are given by semiflows, i.e. continuous mappings
		\[
			T_{\epsilon}\colon [0,T]\times X_1\times Y_1\to X_1\times Y_1,\quad T_{0}\colon [0,T]\times S_0\to S_0.
		\]
		We write
		\[
			\begin{pmatrix}u^{\epsilon}(t)\\v^{\epsilon}(t)\end{pmatrix}=T_{\epsilon}(t)\begin{pmatrix}u_0\\v_0\end{pmatrix},\quad \begin{pmatrix}u^{0}(t)\\v^{0}(t)\end{pmatrix}=T_{0}(t)\begin{pmatrix}h^0(v_0)\\v_0\end{pmatrix}.
		\]
	\end{enumerate}
\end{remark}

\subsection{Extended Slow Flow}
One of our aims is to show that the semiflow of the fast-slow system $(T_{\epsilon}(t))_{t\geq0}$ behaves similarly to the slow flow $(T_{0}(t))_{t\geq0}$. However, the slow flow is only defined on the critical manifold $S_0$ while $(T_{\epsilon}(t))_{t\geq0}$ is defined on $X_1\times Y_1$. Thus, we will compare $(T_{\epsilon}(t))_{t\geq0}$ to an extension $(T_{\epsilon,0}(t))_{t\geq0}$ of the slow flow to $X_1\times Y_1$. This extension will approach the slow flow at an exponential rate and on the critical manifold it will coincide with the slow flow. This extended flow will be generated by the equation
\begin{align}
\begin{aligned}\label{Eq:Nonlinear_Slow_Subsystem_Extended}
 \epsilon\partial_t u^{\epsilon,0}(t) &= Au^{\epsilon,0}(t)+ f(u^{\epsilon,0}(t),v^{0}(t))-\epsilon \partial_tA^{-1}f(h^0(v^0(t)),v^{0}(t)),\\
\partial_t v^{0}(t) &= Bv^{0}(t)+g(h^0(v^0(t)),v^{0}(t)),\\
 u^{\epsilon,0}(0)&=u_0,\quad v^{0}(0)=v_0.
 \end{aligned}
\end{align}
In this equation, the slow variable satisfies the equation of the slow subsystem. The fast variable however satisfies the equation of the fast-slow system with an additional drift in the direction of the slow flow.

\begin{proposition}
 There is a unique solution 
 \[
 (u^{\epsilon,0},v^0)\in C^1([0,T];X\times Y)\cap C([0,T];X_1\times Y_1)
 \]
 of \eqref{Eq:Nonlinear_Slow_Subsystem_Extended} given by a semiflow $(T_{\epsilon,0}(t))_{t\geq0}$ on $X_1\times Y_1$. The critical manifold $S_0$ is invariant under $T_{\epsilon,0}(t)$ for all $t\geq0$. Moreover, the restriction of $(T_{\epsilon,0}(t))_{t\geq0}$ to the critical manifold coincides with the slow flow, i.e. $(T_{\epsilon,0}(t)\vert_{S_0})_{t\geq0}=(T_{0}(t))_{t\geq0}$.
\end{proposition}
\begin{proof}
In the proof of Proposition~\ref{Prop:Fast_Slow_Well-Posed}~\eqref{Prop:Fast_Slow_Well-Posed:epsilon=0} it was shown that there is a unique solution
		\[
			v^0\in C^1([0,T];Y)\cap C([0,T];Y_1)
		\]
		of the equation 
		\[
			\partial_tv^0(t)=Bv^0(t)+g(h^0(v^0(t)),v^0(t)),\quad v^0(0)=v_0
		\]
for all $v_0\in Y_1$. We define
		\[
			f_{\epsilon,v^0}\colon [0,T]\times X_{\delta_X}\to X,\, x\mapsto f(x,v^0(t))-\epsilon\partial_t 		A^{-1} f(h(v^0(t)), v^0(t)).
		\]
Since $v^0\in C^1([0,T];Y)\cap C([0,T];Y_1)$, it follows from Remark~\ref{Rem:Remarks_on_Flows}~\eqref{Rem:Remarks_on_Flows:Critical_Manifold} that
		\[
			[0,T]\times X_{\delta_X}\to X,\, (t,x)\mapsto A^{-1}\partial_t f(h^0(v^0(t)), v^0(t)) 
		\]
and therefore also $f_{\epsilon,v^0}$ is continuous. Moreover, we have the estimate
		\[
			\| f_{\epsilon,v^0}(t,x_1)-f_{\epsilon,v^0}(t,x_2) \|_{X_{\gamma_X}}=\|f(x_1,v^0(t))-f(x_2,v^0(t))\|_{X_{\gamma_X}}\leq L_f \|x_1-x_2\|_{X_1}.
		\]
Now Proposition~\ref{Prop:Nonlinear_IVP_Existence}~\eqref{Prop:Nonlinear_IVP_Existence:epsilon>0} together with Remark~\ref{Rem:2nd_Lipschitz_unnecessary} shows that there is a unique solution $u^{\epsilon,0}\in C^1([0,T];X)\cap C([0,T];X_1)$ of
\[
	\epsilon\partial_t u^{\epsilon,0}(t) = Au^{\epsilon,0}(t)+ f(u^{\epsilon,0}(t),v^{0}(t))-\epsilon \partial_tA^{-1}f(h^0(v^0(t)),v^{0}(t)),\quad u^{\epsilon,0}(0)=u_0.
\]
The desired solution is given by $(u^{\epsilon,0},v^0)$. Since \eqref{Eq:Nonlinear_Slow_Subsystem_Extended} is autonomous, the solution is given by a semiflow $(T_{\epsilon,0}(t))_{t\geq0}$. Note that if $(u_0,v_0)\in S_0$, then the slow flow with initial value $v_0$ solves \eqref{Eq:Nonlinear_Slow_Subsystem_Extended}. Therefore, the critical manifold is invariant under $T_{\epsilon,0}(t)$ for all $t\geq0$ and $(T_{\epsilon,0}(t))_{t\geq0}$ coincides with $(T_{0}(t))_{t\geq0}$ on the critical manifold.
\end{proof}

\begin{proposition}\label{Prop:Slow_Flow_Attracts_Extension}
For all $t\geq0$ it holds that
	\[
		\left\|T_{\epsilon,0}(t)\begin{pmatrix}u_0\\ v_0 \end{pmatrix}-T_{0}(t)\begin{pmatrix}h^0(v_0)\\ v_0 \end{pmatrix}\right\|_{X_1\times Y_1}\leq 2M_A\txte^{\epsilon^{-1}\omega_f t}\|u_0-h^0(v_0)\|_{X_1}.
	\]
\end{proposition}
\begin{proof}
	Since the second components of $T_{\epsilon,0}(t)(u_0,v_0)^T$ and $T_{0}(t)(h^0(v_0),v_0)^T$ are equal, we only have to estimate $\|u^{\epsilon,0}(t)-u^0(t)\|_{X_1}$. But it was shown in Proposition~\ref{Prop:Modified_Fast_and_Time_Independent} that 
		\[
		\|u^{\epsilon,0}(t)-u^0(t)\|_{X_1}\leq 2M_A\txte^{\epsilon^{-1}\omega_f t}\|u_0-h^0(v_0)\|_{X_1}.
	\]
	This shows the assertion.
\end{proof}

\subsection{Approximation by the Slow Flow}

\begin{theorem}\label{Thm:Original_and_Modified_Flow_Close}
 There are constants $C,c>0$ such that
 \[
  \left\| T_{\epsilon}(t)\begin{pmatrix}u_0 \\ v_0 \end{pmatrix}- T_{\epsilon,0}(t)\begin{pmatrix}u_0 \\ v_0 \end{pmatrix}\right\|_{X_1\times Y_1}\leq C \txte^{(\omega_B+c)t} \big(\epsilon \|v_0\|_{Y_1}+\epsilon^{\delta_Y}\| u_0-h^0(v_0)\|_{X_1}\big)
 \]
holds for all $(u_0,v_0)^T\in X_1\times Y_1$, all $t\geq0$ and all $\epsilon\in(0,1]$.
\end{theorem}
\begin{remark}
	Before we turn to the proof we briefly give a rough idea on how large $C$ and $c$ have to be. Actually, we have all the ingredients to explicitly give formulas for these constants and we could also give them by keeping track of the constants in the proof of Theorem~\ref{Thm:Original_and_Modified_Flow_Close}. However, these formulas would be quite involved and probably not sharp. Thus, we refrain from giving precise constants here.\\
	The constant $C>0$ should not be very large unless $\delta_Y,\gamma_X$ or $\omega_f$ are close to $0$. If either of these values tends to $0$, then $C$ will tend to $\infty$. $C$ is basically constructed from the constants which were explicitely computed in Proposition~\ref{Prop:A_Priori_Estimate_Fast_Equation}~\eqref{Prop:A_Priori_Estimate_Fast_Equation:epsilon>0} (with $\epsilon=1$ and $\gamma=\delta_Y$), Proposition~\ref{Prop:Fast_and_Modified_Fast} and Proposition~\ref{Prop:Modified_Fast_and_Time_Independent}.\\
	For $c$ we are a little bit more precise, even though our rough estimate for $c$ can probably still be improved: The constant $c$ can be taken to be
	\begin{align*}
	c&=1+2[L_gC_B(2+C_1L_f)]^{\frac{1}{\delta_Y}}\big(\frac{2}{\delta_Y}\big)^{\tfrac{1-\delta_Y}{\delta_Y}}\quad&\text{if}\quad\delta_Y\in(0,1),\\
		c&>1+L_gC_B(2+C_1L_f)\quad&\text{if}\quad\delta_Y=1,
	\end{align*}
	where $C_1$ is given by
	\[
		C_1=2C_A\left(\frac{\txte^{\gamma_X}}{\gamma_X^{1-\gamma_X}}+\Gamma(\gamma_X)\bigg|\frac{\omega_A}{\omega_f}\bigg|^{1-\gamma_X}\right)\frac{1}{|\omega_f|^{\gamma_X}}.
	\]
\end{remark}
\begin{proof}[Proof of Theorem~\ref{Thm:Original_and_Modified_Flow_Close}]
 In this proof, we use the notation
 \[
  \begin{pmatrix}u^{\epsilon}(t) \\ v^{\epsilon}(t) \end{pmatrix}=T_{\epsilon}(t)\begin{pmatrix}u_0 \\ v_0 \end{pmatrix},\quad\begin{pmatrix}u^{\epsilon,0}(t) \\ v^{0}(t) \end{pmatrix}=T_{\epsilon,0}(t)\begin{pmatrix}u_0 \\ v_0 \end{pmatrix}.
 \]
Variation of constants shows that
\begin{align*}
 v^{\epsilon}(t)&=\txte^{tB}v_0+\int_0^t \txte^{(t-s)B} g(u^{\epsilon}(s),v^{\epsilon}(s))\,\txtd s\\
 v^{0}(t)&=\txte^{tB}v_0+\int_0^t \txte^{(t-s)B} g(h^0(v^0(s)),v^{0}(s))\,\txtd s.
\end{align*}
Therefore, we have that
	\begin{align}\begin{aligned}\label{Eq:Slow_Variables_Distance}
		&\qquad\|v^{\epsilon}(t)-v^{0}(t)\|_{Y_1}\\
		&\leq L_gC_{B}\int_0^t \frac{\txte^{(t-s)\omega_B}}{(t-s)^{1-\delta_Y}}\big(\|u^{\epsilon}(s)-h^0(v^0(s))\|_{X_1}+\|v^{\epsilon}(t)-v^{0}(t)\|_{Y_1}\big)\,\txtd s.
		\end{aligned}
	\end{align}
The aim is to apply Gronwall's inequality. But before we do this, we first estimate the term $\|u^{\epsilon}(s)-h^0(v^0(s))\|_{X_1}$. Let $\tilde{u}^{\epsilon}$ be the unique strict solution of 
\begin{align*}
 \epsilon\partial_t\tilde{u}^{\epsilon}&=A\tilde{u}^{\epsilon}+f(\tilde{u}^{\epsilon},v^0),\\
 \tilde{u}^{\epsilon}(0)&=u_0,
\end{align*}
which exists by Proposition~\ref{Prop:Nonlinear_IVP_Existence}~\eqref{Prop:Fast_Slow_Well-Posed:epsilon>0}. By the triangle inequality, we have
\begin{align*}
	&\qquad\|u^{\epsilon}(s)-h^0(v^0(s))\|_{X_1}\\
	&\leq \|u^{\epsilon}(s)-\tilde{u}^{\epsilon}(s)\|_{X_1}+\|\tilde{u}^{\epsilon}(s)-u^{\epsilon,0}(s)\|_{X_1}+\|u^{\epsilon,0}(s)-h^0(v^0(s))\|_{X_1}
\end{align*}
Using Proposition \ref{Prop:A_Priori_Different_Nonlinearities} with $\eta=0$ we obtain that there is a constant $C_1>0$ such that
\begin{align*}
	\|u^{\epsilon}(s)-\tilde{u}^{\epsilon}(s)\|_{X_1}&\leq C_1 \sup_{0\leq r\leq s, x\in X_1}\|f(x,v^{\epsilon}(r))-f(x,v^{0}(r))\|_{X_{\gamma}}\\
	&\leq C_1L_f\|v^{\epsilon}(t)-v^{0}(t)\|_{Y_1}.
\end{align*}
Proposition~\ref{Prop:Fast_and_Modified_Fast} and Proposition~\ref{Prop:A_Priori_Estimate_Fast_Equation}~\eqref{Prop:A_Priori_Estimate_Fast_Equation:epsilon>0} show that there are constants $C_2,\tilde{C}_2\geq0$ such that
\begin{align*}
	\|\tilde{u}^{\epsilon}(s)-u^{\epsilon,0}(s)\|_{X_1}&\leq \tilde{C}_2\epsilon \|f(0,v^{0}) \|_{C^1([0,s];X_{\delta_X-1})}\\
	&\leq \tilde{C}_2 L_f\epsilon \|v^{0}\|_{C^1([0,s];Y)}\\
	&\leq C_2\epsilon \txte^{\omega_g s} \|v_0\|_{Y_1},
\end{align*}
where
\begin{align*}
\begin{aligned}
	\omega_g&=\omega_B+(2C_BL_g)^{1/\delta_Y}\big(\tfrac{1}{\delta_Y}\big)^{\frac{1-\delta_Y}{\delta_Y}}\quad\text{if }\delta_Y\in(0,1),\\
	\omega_g&>\omega_B+C_BL_g\quad\text{if }\delta_Y=1.
	\end{aligned}
\end{align*}
Moreover, Proposition~\ref{Prop:Modified_Fast_and_Time_Independent} yields
\begin{align*}
	\|u^{\epsilon,0}(s)-h^0(v^0(s))\|_{X_1}\leq 2M_A\txte^{\epsilon^{-1}\omega_f s}\|u_0-h^0(v_0)\|_{X_1}.
\end{align*}
By combining the previous four estimates with \eqref{Eq:Slow_Variables_Distance}, we obtain that there is a constant $C>0$ not depending on $\omega_B,u_0,v_0$ and $\epsilon$ such that
\begin{align*}
	\|v^{\epsilon}(t)-v^{0}(t)\|_{Y_1}&\leq C\int_0^t\frac{\txte^{(t-s)\omega_B}}{(t-s)^{1-\delta_Y}}(\epsilon \txte^{\omega_g s} \|v_0\|_{Y_1}+\txte^{\epsilon^{-1}\omega_f s}\|u_0-h^0(v_0)\|_{X_1})\,\txtd s\\
	&\qquad+L_gC_B(1+C_1L_f)\int_0^t \frac{\txte^{(t-s)\omega_B}}{(t-s)^{1-\delta_Y}}\|v^{\epsilon}(s)-v^{0}(s)\|_{Y_1}\,\txtd s\\
		&\leq C\txte^{\omega_gt}\int_0^t\frac{1}{(t-s)^{1-\delta_Y}}(\epsilon\|v_0\|_{Y_1}+\txte^{(\epsilon^{-1}\omega_f -\omega_g)s}\|u_0-h^0(v_0)\|_{X_1})\,\txtd s\\
	&\qquad+L_gC_B(1+C_1L_f)\int_0^t \frac{\txte^{(t-s)\omega_B}}{(t-s)^{1-\delta_Y}}\|v^{\epsilon}(s)-v^{0}(s)\|_{Y_1}\,\txtd s\\
	&\leq C\txte^{\omega_gt}\bigg(\frac{ t^{\delta_Y}}{\delta_Y} \epsilon\|v_0\|_{Y_1}+\frac{e+\delta_Y}{\delta_Y(\epsilon\omega_g-\omega_f)^{\delta_Y}}\epsilon^{\delta_Y}\| u_0-h^0(v_0)\|_{X_1}\bigg)\\
		&\qquad\qquad+L_gC_B(1+C_1L_f)\int_0^t \frac{\txte^{(t-s)\omega_B}}{(t-s)^{1-\delta_Y}}\|v^{\epsilon}(s)-v^{0}(s)\|_{Y_1}\,\txtd s\\
		&\leq C\txte^{(\omega_g+1)t}\bigg(\epsilon\|v_0\|_{Y_1}+\epsilon^{\delta_Y}\| u_0-h^0(v_0)\|_{X_1}\bigg)\\
	&\qquad\qquad+L_gC_B(1+C_1L_f)\int_0^t \frac{\txte^{(t-s)(\omega_g+1)}}{(t-s)^{1-\delta_Y}}\|v^{\epsilon}(s)-v^{0}(s)\|_{Y_1}\,\txtd s,
\end{align*}
where we used Lemma~\ref{Lemma:Not_Incomplete_Gamma}. Thus, Lemma~\ref{Lemma:Gronwall_Specific} shows that there is a constant $C>0$ not depending on $\omega_B$, $u_0,v_0$ and $\epsilon$ such that
\begin{align*}
	\|v^{\epsilon}(t)-v^{0}(t)\|_{Y_1}\leq C\txte^{(\omega_B+c)t}\big(\epsilon \|v_0\|_{Y_1}+\epsilon^{\delta_Y}\| u_0-h^0(v_0)\|_{X_1}\big)\quad(t\geq0),
\end{align*}
where $c=1+2[L_gC_B(2+C_1L_f)]^{\frac{1}{\delta_Y}}\big(\frac{2}{\delta_Y}\big)^{\frac{1-\delta_Y}{\delta_Y}}$ if $\delta_Y\in(0,1)$ and $c>1+L_gC_B(2+C_1L_f)$ if $\delta_Y=1$.
Using this estimate for the for the slow variable, Proposition \ref{Prop:A_Priori_Different_Nonlinearities} and Proposition \ref{Prop:Fast_and_Modified_Fast} we also obtain for the fast variable
\begin{align*}
 \|u^{\epsilon}(t)-u^{\epsilon,0}(t)\|_{X_1}&\leq \|u^{\epsilon}(t)-\tilde{u}^{\epsilon}(t)\|_{X_1}+\|\tilde{u}^{\epsilon}(t)-u^{\epsilon,0}(t)\|_{X_1}\\
 &\leq C \txte^{(\omega_B+c)t}\big(\epsilon \|v_0\|_{Y_1}+\epsilon^{\delta_Y}\| u_0-h^0(v_0)\|_{X_1}\big).
\end{align*}
Altogether, we obtain the assertion.
\end{proof}

\begin{corollary}\label{Cor:Original_and_Slow_Flow_Close}
  There are constants $C,c>0$ such that
 \begin{align*}
  &\left\| T_{\epsilon}(t)\begin{pmatrix}u_0 \\ v_0 \end{pmatrix}- T_{0}(t)\begin{pmatrix} h^0(v_0) \\ v_0 \end{pmatrix}\right\|_{X_1\times Y_1}\\
  &\qquad\qquad\leq C \big(\epsilon \txte^{(\omega_B+c)t}\|v_0\|_{Y_1}+(\epsilon^{\delta_Y}\txte^{(\omega_B+c)t}+\txte^{\epsilon^{-1}\omega_f t})\| u_0-h^0(v_0)\|_{X_1}\big)
 \end{align*}
holds for all $(u_0,v_0)^T\in X_1\times Y_1$, all $t\in[0,T]$ and all $\epsilon\in(0,1]$.
\end{corollary}
\begin{proof}
	This is a combination of Proposition~\ref{Prop:Slow_Flow_Attracts_Extension}  and Theorem~\ref{Thm:Original_and_Modified_Flow_Close}.
\end{proof}

\section{Slow Manifolds}
\label{sec:slowfinal}

Under additional assumptions on the operator $B$ in the equation of the slow variable, we now prove the existence of a family of slow manifolds $S_{\epsilon,\zeta}$. Unlike in finite dimensions, this family will depend on two parameters. While $\epsilon$ plays the same role as in the finite-dimensional setting, the parameter $\zeta$ is new. As explained in Section~\ref{Section:Bates_Difficult} there might be parts of the slow dynamics which decay faster than other parts in the fast equation evolve. Our idea is to find a certain splitting of the slow variable in a fast and a slow part. The fast part of the slow variable will then be treated together with the fast variable, while the slow manifolds are constructed as graphs over the slow part. The parameter $\zeta$ determines which parts of the slow variables are considered as fast and which parts are considered as slow. In the language of normally hyperbolic invariant manifolds one could say that the stable direction will consist of the fast variable and the fast part of the slow variable, and the center direction will consist of the slow part of the slow variable. The finite dimensional situation will also be recovered as a special case: The family of slow manifolds $S_{\epsilon,\zeta}$ will then not depend on $\zeta$ so that one could omit it in the notation and obtain a family $S_{\epsilon}$ as usual in finite dimensions. More generally, if $B$ generates a $C_0$-group, then the family of slow manifolds will not depend on $\zeta$. We will give applications of our techniques to systems of fast-slow partial differential equations in Section~\ref{sec:examples}. In the next subsection, we make our assumptions more precise.

\subsection{Our approach on how to resolve the issues of Section~\ref{Section:Bates_Difficult}} \label{Section:Bates_Difficult:Our_Approach}
For the problems explained in Section \ref{Subsection:Bates_Difficult:Normal_Hyperbolicity} and Section \ref{Subsection:Bates_Difficult:Fast_Slow_Times}, we assume that for each small $\zeta>0$ we have a splitting of the slow variable space
\[
 Y=Y_{F}^{\zeta}\oplus Y_{S}^{\zeta}
\]
in a fast part $Y_{F}^{\zeta}$ and a slow part $Y_{S}^{\zeta}$ such that
\begin{enumerate}[(i)]
 \item The spaces $Y_{F}^{\zeta}$ and $Y_{S}^{\zeta}$ are closed in $Y$ and the projections $\operatorname{pr}_{Y_F^{\zeta}}$ and $\operatorname{pr}_{Y_S^{\zeta}}$ commute with $B$ on $Y_{1}$.
 \item The space $Y_{F}^{\zeta}\cap Y_1$ is a closed subspace of $Y_1$ and will be endowed with the norm $\|\cdot\|_{Y_1}$.
 \item The space $Y_{S}^{\zeta}\cap Y_1$ is a closed subspace of $Y_1$ and will be endowed with the norm $\|\cdot\|_{Y_1}$. Moreover, the nonlinearity $g$ satisfies
 	\begin{align*}
 		&\|\operatorname{pr}_{Y_S^{\zeta}}[g(x,y_{F},y_{S})-g(\tilde{x},\tilde{y}_{F},\tilde{y}_{S})]\|_{Y_1}\\
 		&\qquad\qquad\leq L_g \zeta^{\delta_Y-1}\big(\|x-\tilde{x}\|_{X_1}+\|y_F-\tilde{y}_F\|_{Y_1}+\|y_S-\tilde{y}_S\|_{Y_1}\big).
 	\end{align*}
 \item The realization of $B$ in $Y_S^{\zeta}$, i.e.
 \[
  B_{Y_S^{\zeta}}\colon Y_S^{\zeta}\supset D(B_{Y_S^{\epsilon}})\to Y_S^{\zeta},\; v\mapsto Bv
 \]
 with 
 \[
 	D(B_{Y_S^{\zeta}}):=\{v_0\in Y_S^{\zeta}\cap D(B):Bv_0\in Y_S^{\zeta} \}
 \]
 generates a $C_0$-group $(\txte^{tB_{Y_S^{\zeta}}})_{t\in\R}\subset\mathcal{B}((Y_S^{\zeta},\|\cdot\|_{Y}))$ which satisfies $\txte^{tB_{Y_S^{\zeta}}}=\txte^{tB}$ on $Y_S^{\zeta}$ for $t\geq0$.
 \item The realization of $B$ in $Y_F^{\zeta}$, i.e.
 \[
  B_{Y_F^{\zeta}}\colon Y_F^{\zeta}\supset D(B_{Y_F^{\epsilon}})\to Y_F^{\zeta},\; v\mapsto Bv
 \]
 with 
 \[
 	D(B_{Y_F^{\zeta}}):=\{v_0\in Y_F^{\zeta}:Bv_0\in Y_F^{\zeta} \}
 \]
 has $0$ in its resolvent set.
 \item The space $Y_F^{\zeta}$ contains the parts of $Y_1$ that decay under the semigroup $(\txte^{tB})_{t\geq0}$ almost as fast as the space $X_1$ under $(\txte^{\zeta^{-1}tA})_{t\geq0}$. The space $Y_S^{\zeta}$ contains the parts of $Y_1$ which do not decay or which only decay slowly under the semigroup $(\txte^{tB})_{t\geq0}$ compared to $X_1$ under $(\txte^{\zeta^{-1}tA})_{t\geq0}$. More precisely, there are constants $C_B,M_B>0$ such that for all $\zeta>0$ small enough there are constants $0\leq N_F^{\zeta}< N_S^{\zeta}$ such that for all $t\geq0$, $y_F\in Y_F^{\zeta}$ and $y_S\in Y_S^{\zeta}$ we have the estimates
 \begin{align*}
  \|\txte^{tB}y_F\|_{Y_1}&\leq  C_Bt^{\delta_Y-1}\txte^{(N_F^{\zeta}+\zeta^{-1}\omega_A)t} \|y_F\|_{Y_{\delta_Y}},\\
   \|\txte^{-tB}y_S\|_{Y_1}&\leq M_B \txte^{-(N_S^{\zeta}+\zeta^{-1}\omega_A) t} \|y_S\|_{Y_1}.
 \end{align*}
 \item We have the estimate
 \begin{align}\label{Eq:Gap_Condition}
 	\frac{2^{\gamma_X}L_fC_A\Gamma(\gamma_X)}{\big(2(\epsilon\zeta^{-1}-1)\omega_A+\epsilon(N_S^{\zeta}+N_F^{\zeta})\big)^{\gamma_X}}+\frac{2^{\delta_Y}L_gC_B\Gamma(\delta_Y)}{(N_S^{\zeta}-N_F^{\zeta})^{\delta_Y}}+\frac{2\zeta^{\delta_Y-1}L_gM_B\Gamma(\delta_Y)}{N_S^{\zeta}-N_F^{\zeta}}<1,
 \end{align}
 which will be needed for an application of Banach's fixed point theorem.
\end{enumerate}
These conditions might seem very restrictive at first. However, in many applications it is possible to find such a decomposition. In many cases, it can be obtained by using Riesz projections corresponding to $B$. This can for example be done if $B$ is a parabolic operator on a bounded domain. If $B$ generates a group, then it will even suffice to take $Y_F^{\zeta}=\{0\}$ and $Y_S^{\zeta}=Y$ for small $\epsilon$. In particular, one can always find such a decomposition if the equation for the slow variable is given by an ordinary differential equation.\\
Besides the parameters $\epsilon$ and $\zeta$, the quantity $N^{\zeta}_S-N^{\zeta}_F$ also plays a certain role. It measures how far one can seperate the decay properties of the fast and the slow part in the slow variable. In many situations this number corresponds to size of spectral gaps in the real part of the spectrum of $B$ as one approaches $-\infty$. For example, if $B$ is the Laplace operator $\Delta$ on $L_2([0,2\pi])$ with Dirichlet boundary conditions, then the eigenvalues are of the form $-k^{2}$. The gaps between two consecutive different eigenvalues will then be given by $2k+1$, i.e. it will behave almost like the square root of the size of the eigenvalues times a constant. In such a situation, $N^{\zeta}_S-N^{\zeta}_F$ will behave like $C\zeta^{-\frac{1}{2}}$ as $\zeta\to0$. If $B$ generates a group, then it will hold that $N^{\zeta}_S-N^{\zeta}_F$ behaves  like $\zeta^{-1}$.\\
We use this splitting to rewrite the fast-slow system \eqref{Eq:Nonlinear_Fast_Slow_System} as
\begin{align}
\begin{aligned}\label{Eq:Nonlinear_Fast_Slow_System:Splitting}
 \epsilon\partial_t u^{\epsilon}(t) &= Au^{\epsilon}(t)+ f(u^\epsilon(t),v^{\epsilon}_F(t),v^{\epsilon}_S(t)),\\
 \partial_t v^{\epsilon}_F(t) &= Bv_F^{\epsilon}(t)+\operatorname{pr}_{Y_F^{\zeta}}g(u^\epsilon(t),v^{\epsilon}_F(t),v^{\epsilon}_S(t)),\\
  \partial_t v^{\epsilon}_S(t) &= Bv_S^{\epsilon}(t)+\operatorname{pr}_{Y_S^{\zeta}}g(u^\epsilon(t),v^{\epsilon}_F(t),v^{\epsilon}_S(t)),\\
 u^{\epsilon}(0)&=u_0,\quad v^{\epsilon}_F(0)=\operatorname{pr}_{Y_F^{\zeta}}v_0,\quad v^{\epsilon}_S(0)=\operatorname{pr}_{Y_S^{\zeta}}v_0,
 \end{aligned}\quad(t\in[0,T])
\end{align}
with an abuse of notation: Actually, $f$ and $g$ only depend on two variables, but we use the convention $f(u^\epsilon(t),v^{\epsilon}_F(t),v^{\epsilon}_S(t)):=f(u^\epsilon(t),v^{\epsilon}_F(t)+v^{\epsilon}_S(t))$ as well as $g(u^\epsilon(t),v^{\epsilon}_F(t),v^{\epsilon}_S(t)):=g(u^\epsilon(t),v^{\epsilon}_F(t)+v^{\epsilon}_S(t))$.\\
We should point out that, as already mentioned at the beginning of Section~\ref{sec:genfs}, there are also certain situations in which the space of the slow variable does not admit such a splitting. The main example we have in mind is if $B$ is a parabolic operator such as the Laplacian $\Delta$ on the whole space $\R^n$. If it is considered on $L_p(\R^n)$, then there are no gaps in the spectrum and it will not be possible to find the constants $0\leq N_F^{\zeta}< N_S^{\zeta}$. In such a situation, we will not be able to construct slow manifolds. If $B$ is a parabolic operator on a bounded domain in dimension $n\geq2$, then it admits such a splitting, but the spectral gaps will usually not grow as $\zeta\to0$. In this case, \eqref{Eq:Gap_Condition} will usually not be satisfied. Nonetheless, we can still use the results of Section~\ref{sec:genfs} in both situations to justify that one may reduce the fast-slow system to the slow subsystem.

\subsection{Existence of Slow Manifolds}

Now we want to construct a family of slow manifolds $S_{\epsilon,\zeta}$ which are given as graphs of certain functions
\[
	h^{\epsilon,\zeta}\colon (Y_S^{\zeta}\cap Y_1)\to X_1\times (Y_F^{\zeta}\cap Y_1),
\]
over the slow part of the slow variable, i.e.  we have that
\[
	S_{\epsilon,\zeta}:=\{(h^{\epsilon,\zeta}(v_0),v_0):v_0\in Y_S^{\zeta}\cap Y_1 \}.
\]
 In the following, we write $h^{\epsilon,\zeta}_{X_1}$ for the first and $h^{\epsilon,\zeta}_{Y_F^{\zeta}}$ for the second component.
We use the Lyapunov-Perron method for the construction of slow manifolds, i.e. we construct fixed points of the operator
\begin{align*}
	\mathscr{L}_{v_0,\epsilon,\zeta}\colon &C_{\eta}\to C_{\eta},\\
	&\quad\begin{pmatrix}u\\v_F\\v_S \end{pmatrix}\mapsto \left[t\mapsto\begin{pmatrix} 
		\epsilon^{-1} \int_{-\infty}^t \txte^{\epsilon^{-1}(t-s)A} f(u(s),v_F(s),v_S(s))\,\txtd s \\
		\int_{-\infty}^t \txte^{(t-s)B} \operatorname{pr}_{Y_F^{\zeta}}g(u(s),v_F(s),v_S(s))\,\txtd s\\
		\txte^{tB}v_0+\int_{0}^t \txte^{(t-s)B} \operatorname{pr}_{Y_S^{\zeta}}g(u(s),v_F(s),v_S(s))\,\txtd s
	\end{pmatrix}\right],
\end{align*}
where $v_0\in Y_S^{\zeta}$ and $C_{\eta}:=C((-\infty,0],\txte^{\eta t}; X_1\times (Y_F^{\zeta}\cap Y_1)\times (Y_S^{\zeta}\cap Y_1))$ for 
$$\eta:=\zeta^{-1}\omega_A+\frac{N_S^{\zeta}+N_F^{\zeta}}{2}$$
is the space of all $(u,v_F,v_S)\in C((-\infty,0];X_1\times (Y_F^{\zeta}\cap Y_1)\times (Y_S^{\zeta}\cap Y_1))$ such that
\[
	\|(u,v_F,v_S)\|_{C_{\eta}}:=\sup_{t\leq0 } \txte^{-\eta t} \big(\|u(t)\|_{X_1}+\|v_F(t)\|_{Y_1}+\|v_S(t)\|_{Y_1}\big)<\infty.
\]
Then we obtain the function $h^{\epsilon,\zeta}$ which describes the family of slow manifolds $S_{\epsilon,\zeta}$  by
\[
	h^{\epsilon,\zeta}\colon (Y_S^{\zeta}\cap Y_1)\to X_1\times (Y_F^{\zeta}\cap Y_1),\,v_0\mapsto (u^{v_0}(0),v_F^{v_0}(0))^T,
\]
i.e. $h^{\epsilon,\zeta}$ gives the first two components of the fixed point $(u^{v_0},v_F^{v_0},v_S^{v_0})^T$ of $\mathscr{L}_{v_0,\epsilon,\zeta}$ evaluated at $t=0$.

\begin{proposition}\label{Prop:Slow_Manifold_Existence}
	Let $v_0\in Y_S^{\zeta}\cap Y_1$. Then $\mathscr{L}_{v_0,\epsilon,\zeta}$ has a unique fixed point in $C_{\eta}$.
\end{proposition}
\begin{proof}
	We show that $\mathscr{L}_{v_0,\epsilon,\zeta}$ is a contraction on $C_{\eta}$. So let $(u,v_F,v_S),(\tilde{u},\tilde{v}_F,\tilde{v}_S)\in C_{\eta}$. Since showing that $\mathscr{L}_{v_0,\epsilon,\zeta}$ maps $C_{\eta}$ into $C_{\eta}$ and showing that $\mathscr{L}_{v_0,\epsilon,\zeta}$ is a contraction on $C_{\eta}$ works in a similar way, we only show the latter. For the first component, we have that
	\begin{align*}
		&\quad\sup_{t\leq0} \txte^{-\eta t}\|\operatorname{pr}_{X_1}\big(\mathscr{L}_{v_0,\epsilon,\zeta}(u(t),v_F(t),v_S(t))^T-\mathscr{L}_{v_0,\epsilon,\zeta}(\tilde{u}(t),\tilde{v}_F(t),\tilde{v}_S(t))^T\big)\|_{X_1}\\
		&\leq L_fC_A\int_{-\infty}^t\frac{\txte^{(t-s)(\epsilon^{-1}\omega_A-\eta)}}{\epsilon^{\gamma_X}(t-s)^{1-\gamma_X}}\,\txtd s\|(u-\tilde{u},v_F-\tilde{v}_F,v_S-\tilde{v}_S) \|_{C_{\eta}}\\
		&= \frac{L_fC_A\Gamma(\gamma_X)}{(\epsilon \eta-\omega_A)^{\gamma_X}}\|(u-\tilde{u},v_F-\tilde{v}_F,v_S-\tilde{v}_S) \|_{C_{\eta}}\\
		&= \frac{2^{\gamma_X}L_fC_A\Gamma(\gamma_X)}{\big(2(\epsilon\zeta^{-1}-1)\omega_A+\epsilon(N_S^{\zeta}+N_F^{\zeta})\big)^{\gamma_X}}\|(u-\tilde{u},v_F-\tilde{v}_F,v_S-\tilde{v}_S) \|_{C_{\eta}}.
	\end{align*}
 For the second component, we have that
	\begin{align*}
		&\quad\sup_{t\leq0} \txte^{-\eta t}\|\operatorname{pr}_{Y_F^{\zeta}}\big(\mathscr{L}_{v_0,\epsilon,\zeta}(u(t),v_F(t),v_S(t))^T-\mathscr{L}_{v_0,\epsilon,\zeta}(\tilde{u}(t),\tilde{v}_F(t),\tilde{v}_S(t))^T\big)\|_{Y_1}\\
		&\leq L_gC_B\int_{-\infty}^t\frac{\txte^{(t-s)(\zeta^{-1}\omega_A+N_{F}^{\zeta}-\eta)}}{(t-s)^{1-\delta_Y}}\,\txtd s\|(u-\tilde{u},v_F-\tilde{v}_F,v_S-\tilde{v}_S) \|_{C_{\eta}}\\
		&= \frac{L_gC_B\Gamma(\delta_Y)}{(\eta-\zeta^{-1}\omega_A- N_{F}^{\zeta})^{\delta_Y}}\|(u-\tilde{u},v_F-\tilde{v}_F,v_S-\tilde{v}_S) \|_{C_{\eta}}\\
		&=\frac{2^{\delta_Y}L_gC_B\Gamma(\delta_Y)}{(N_S^{\zeta}-N_F^{\zeta})^{\delta_Y}}\|(u-\tilde{u},v_F-\tilde{v}_F,v_S-\tilde{v}_S) \|_{C_{\eta}}.
	\end{align*}
	Finally, the third component satisfies
	\begin{align*}
				&\quad\sup_{t\leq0} \txte^{-\eta t}\|\operatorname{pr}_{Y_S^{\zeta}}\big(\mathscr{L}_{v_0,\epsilon,\zeta}(u(t),v_F(t),v_S(t))^T-\mathscr{L}_{v_0,\epsilon,\zeta}(\tilde{u}(t),\tilde{v}_F(t),\tilde{v}_S(t))^T\big)\|_{Y_1}\\
				&\leq L_gC_B\int_{0}^t\zeta^{\delta_Y-1}\txte^{(t-s)(\zeta^{-1}\omega_A+N_{S}^{\epsilon}-\eta)}\,\txtd s\|(u-\tilde{u},v_F-\tilde{v}_F,v_S-\tilde{v}_S) \|_{C_{\eta}}\\
		&\leq \frac{\zeta^{\delta_Y-1}L_gM_B\Gamma(\delta_Y)}{\zeta^{-1}\omega_A+N_{S}^{\zeta}-\eta}\|(u-\tilde{u},v_F-\tilde{v}_F,v_S-\tilde{v}_S) \|_{C_{\eta}}\\
		&=\frac{2\zeta^{\delta_Y-1}L_gM_B\Gamma(\delta_Y)}{N_S^{\zeta}-N_F^{\zeta}}\|(u-\tilde{u},v_F-\tilde{v}_F,v_S-\tilde{v}_S) \|_{C_{\eta}}.
	\end{align*}
	Thus, if \eqref{Eq:Gap_Condition} is satisfied, then $\mathscr{L}_{v_0,\epsilon,\zeta}$ is a contraction. Hence, it has a unique fixed point in this case.
\end{proof}

\begin{proposition}
	Consider the situation of Proposition~\ref{Prop:Slow_Manifold_Existence} and let $(u^{v_0},v_F^{v_0},v_S^{v_0})^T$ be the unique fixed point of $\mathscr{L}_{v_0,\epsilon,\zeta}$. The mapping 
	\[
		h^{\epsilon,\zeta}\colon (Y_S^{\zeta}\cap Y_1)\to X_1\times (Y_F^{\epsilon}\cap Y_1),\,v_0\mapsto (u^{v_0}(0),v_F^{v_0}(0))^T
	\]
	is Lipschitz continuous.
\end{proposition}
\begin{proof}
Let $v_0,\tilde{v}_0\in Y_S^{\zeta}\cap Y_1$ and let $(u,v_F,v_S)\in C_{\eta}$ and $(\tilde{u},\tilde{v}_F,\tilde{v}_S)\in C_{\eta}$ be the fixed points of $\mathscr{L}_{v_0,\epsilon,\zeta}$ and $\mathscr{L}_{\tilde{v}_0,\epsilon,\zeta}$, respectively. As in the proof of Proposition~\ref{Prop:Slow_Manifold_Existence} it follows that
	\begin{align*}
		\sup_{t\leq0} \txte^{-\eta t}\|u(t)-\tilde{u}(t)\|_{X_1}&< \frac{2^{\gamma_X}L_fC_A\Gamma(\gamma_X)\|(u-\tilde{u},v_F-\tilde{v}_F,v_S-\tilde{v}_S) \|_{C_{\eta}}}{\big(2(\epsilon\zeta^{-1}-1)\omega_A+\epsilon(N_S^{\zeta}+N_F^{\zeta})\big)^{\gamma_X}},\\
		\sup_{t\leq0} \txte^{-\eta t}\|v_F(t)-\tilde{v}_F(t)\|_{Y_1}&\leq \frac{2^{\delta_Y}L_gC_B\Gamma(\delta_Y)\|(u-\tilde{u},v_F-\tilde{v}_F,v_S-\tilde{v}_S) \|_{C_{\eta}}}{(N_S^{\zeta}-N_F^{\zeta})^{\delta_Y}},\\
		\sup_{t\leq0} \txte^{-\eta t}\|v_S(t)-\tilde{v}_S(t)\|_{X_1}&\leq M_B\|v_0-\tilde{v}_0\|_{Y_1}\\
		&\quad+\frac{2\zeta^{\delta_Y-1}L_gM_B\Gamma(\delta_Y)\|(u-\tilde{u},v_F-\tilde{v}_F,v_S-\tilde{v}_S) \|_{C_{\eta}}}{N_S^{\zeta}-N_F^{\zeta}}.
	\end{align*}
	Thus, if
	\[
		L:=\tfrac{2^{\gamma_X}L_fC_A\Gamma(\gamma_X)}{\big(2(\epsilon\zeta^{-1}-1)\omega_A+\epsilon(N_S^{\zeta}+N_F^{\zeta})\big)^{\gamma_X}}+\tfrac{2^{\delta_Y}L_gC_B\Gamma(\delta_Y)}{(N_S^{\zeta}-N_F^{\zeta})^{\delta_Y}}+\tfrac{2\zeta^{\delta_Y-1}L_gM_B\Gamma(\delta_Y)}{N_S^{\zeta}-N_F^{\zeta}}<1
	\]
	then we may sum up the three estimates, substract $L\|(u-\tilde{u},v_F-\tilde{v}_F,v_S-\tilde{v}_S) \|_{C_{\eta}}$ and divide by $1-L$. This gives the Lipschitz continuity. 
\end{proof}

\subsection{Distance to the Critical Manifold}\label{Sec:Slow_Manifolds:Distance_Critical_Manifold}
\begin{proposition}\label{Prop:Distance_Critical_Manifold}
	Consider the situation of Proposition~\ref{Prop:Slow_Manifold_Existence} and choose $c\in(0,1)$. There is a constant $C>0$ such that for all $\epsilon,\zeta>0$ small enough which satisfy $\epsilon<c\zeta\frac{(L_fC_A\Gamma(\gamma_X))^{1/\gamma_X}+w_A}{w_A}$ and all $v_0\in Y^{\zeta}_S$ it holds that
	\[
		\left\| \begin{pmatrix} h^{\epsilon,\zeta}_{X_1}(v_0)-h^0(v_0) \\
		 h^{\epsilon,\zeta}_{Y^{\zeta}_F}(v_0)\end{pmatrix} \right\|_{X_1\times Y_1}\leq C\left(\epsilon +\frac{1}{(N_S^{\zeta}-N_F^{\zeta})^{\delta_y}}\right)\|v_0\|_{Y_1}
	\]
\end{proposition}
\begin{proof}
	Let $(\bar{u},\bar{v}_F,\bar{v}_S)\in C_{\eta}$ be the unique fixed point of $\mathscr{L}_{v_0,\epsilon,\zeta}$, i.e. $(\bar{u},\bar{v}_F,\bar{v}_S)=(h^{\epsilon,\zeta}_{X_1}(\bar{v}_S),h^{\epsilon,\zeta}_{Y^{\zeta}_F}(\bar{v}_S),\bar{v}_S)$. Since $(\bar{u},\bar{v}_F,\bar{v}_S)$ solves \eqref{Eq:Nonlinear_Fast_Slow_System:Splitting} on $(-\infty,0]$ we have that $\bar{v}_S\in C^1((-\infty,0],\txte^{\eta t}; Y)$ and
	\[
		\sup_{t\leq0} \txte^{-\eta t}\big(\|\bar{v}_S(t)\|_{Y_1}+\|\partial_t\bar{v}_S(t)\|_Y\big)\leq L(\|A\|_{\mathcal{B}(X_1,X)}+L_f)\|v_0\|_{Y_1}.
	\]
	Moreover, we have that
	\begin{align}\begin{aligned}\label{Eq:Slow_Manifold_Estimate_Second_Component}
		\|h^{\epsilon,\zeta}_{Y^{\zeta}_F}(\bar{v}_S(t))\|_{Y_1}&=\left\|\int_{-\infty}^t\txte^{(t-s)B}\operatorname{pr}_{Y_{F}^{\zeta}}g(h^{\epsilon,\zeta}_{X_1}(\bar{v}_S(s)),h^{\epsilon,\zeta}_{Y^{\zeta}_F}(\bar{v}_S(s)),\bar{v}_S(s))\,\txtd s \right\|_{Y_1}\\
		&\leq L_gC_B\txte^{\eta t} \|(h^{\epsilon,\zeta}_{X_1}(\bar{v}_S),h^{\epsilon,\zeta}_{Y^{\zeta}_F}(\bar{v}_S),\bar{v}_S)\|_{C_\eta}\int_{-\infty}^t \frac{\txte^{(t-s)(\zeta^{-1}\omega_A+N_{F}^{\zeta}-\eta)}}{(t-s)^{1-\delta_Y}}\,\txtd s \\
		&\leq \frac{L L_g C_B\Gamma(\delta_Y)\txte^{\eta t}}{(\eta-\zeta^{-1}\omega_A-N_f^{\zeta})^{\delta_Y}} \|v_0 \|_{Y_1}.
	\end{aligned}\end{align}
	Furthermore, integration by parts shows that for $t_0\leq t\leq0$ it holds that
	\begin{align*}
		 h^{\epsilon,\zeta}_{X_1}(\bar{v}_S(t))-h^0(\bar{v}_S(t))&=\epsilon^{-1}\int_{-\infty}^{t}\txte^{\epsilon^{-1}(t-s)A}f(h^{\epsilon,\zeta}_{X_1}(\bar{v}_S(s)),h^{\epsilon,\zeta}_{Y^{\zeta}_F}(\bar{v}_S(s)),\bar{v}_S(s))\,\txtd s\\
		 &\qquad+A^{-1}f(h^0(\bar{v}_S(t)),0,\bar{v}_S(t))\\
		 &=\epsilon^{-1}\int_{-\infty}^{t_0}\txte^{\epsilon^{-1}(t-s)A}f(h^{\epsilon,\zeta}_{X_1}(\bar{v}_S(s)),h^{\epsilon,\zeta}_{Y^{\zeta}_F}(\bar{v}_S(s)),\bar{v}_S(s))\,\txtd s\\
		 &\quad+\txte^{\epsilon^{-1}(t-t_0)A}A^{-1}f(h^0(\bar{v}_S(t_0)),0,\bar{v}_S(t_0))\\
		 &\quad+\int_{t_0}^t\txte^{\epsilon^{-1}(t-s)A}A^{-1}\partial_sf(h^0(\bar{v}_S(s)),0,\bar{v}_S(s))\,\txtd s\\
		 &\quad+\epsilon^{-1}\int_{t_0}^t\txte^{\epsilon^{-1}(t-s)A}\big[f(h^{\epsilon,\zeta}_{X_1}(\bar{v}_S(s)),h^{\epsilon,\zeta}_{Y^{\zeta}_F}(\bar{v}_S(s)),\bar{v}_S(s))\\
		 &\qquad\qquad\qquad-f(h^0(\bar{v}_S(s)),0,\bar{v}_S(s))\big]\,\txtd s.
	\end{align*}
	Therefore, we obtain
	{\allowdisplaybreaks{
	\begin{align*}
		&\quad\| h^{\epsilon,\zeta}_{X_1}(\bar{v}_S(t))-h^0(\bar{v}_S(t)) \|_{X_1}\\
		\leq& L_fC_A \txte^{\epsilon^{-1}\omega_A(t-t_0)+\eta t_0} \|(h^{\epsilon,\zeta}_{X_1}(\bar{v}_S),h^{\epsilon,\zeta}_{Y^{\zeta}_F}(\bar{v}_S),\bar{v}_S)\|_{C_\eta} \int_{-\infty}^{t_0} \frac{\txte^{(\epsilon^{-1}\omega_A-\eta)(t_0-s)}}{\epsilon^{\gamma_X}(t-s)^{1-\gamma_X}}\,\txtd s\\
		&+L_fM_A\txte^{\epsilon^{-1}\omega_A(t-t_0)+\eta t_0}\|A^{-1}\|_{\mathcal{B}(X_{\gamma_X},X_1)}\| (h^0(\bar{v}_S),0,\bar{v}_S)\|_{C_{\eta}}\\
		&+ \epsilon \txte^{\eta t}L_f C_A \|A^{-1}\|_{\mathcal{B}(X_{\delta_X-1},X_{\delta_X})}\int_{t_0}^t\frac{\txte^{(\epsilon^{-1}\omega_A-\eta)(t-s)}}{\epsilon^{\delta_X}(t-s)^{1-\delta_X}}\,\txtd s\, \\
		&\qquad\qquad\cdot\,\sup_{s\leq0}\big(\txte^{-\eta s}(\|\bar{v}_S(s)\|_Y+\|\partial_s\bar{v}_S(s)\|_Y)\big)\\
		&+\frac{L L_g C_B\Gamma(\delta_Y)}{(\eta-\zeta^{-1}\omega_A-N_f^{\zeta})^{\delta_Y}}\txte^{\eta t}\int_{t_0}^t\frac{\txte^{(\epsilon^{-1}\omega_A-\eta)(t-s)}}{\epsilon^{\gamma_X}(t-s)^{1-\gamma_X}}\,\txtd s \|v_0\|_{Y_1}\\
		& + L_fC_A\int_{t_0}^t\frac{\txte^{\epsilon^{-1}\omega_A(t-s)}}{\epsilon^{\gamma_X}(t-s)^{1-\gamma_X}}\| h^{\epsilon,\zeta}_{X_1}(\bar{v}_S(s))-h^0(\bar{v}_S(s)) \|_{X_1}\,\txtd s\\
		\leq& C\|v_0\|_{Y_1}\bigg(\frac{1}{(\epsilon\eta-\omega_A)^{\gamma_X}}+1\bigg)\txte^{(\eta-\epsilon^{-1}\omega_A)t_0}\txte^{\epsilon^{-1}\omega_At}\\
		&+C\|v_0\|_{Y_1}\bigg(\frac{\epsilon}{(\epsilon\eta-\omega_A)^{\delta_X}}+\frac{1}{(\epsilon\eta-\omega_A)^{\gamma_X}(\eta-\zeta^{-1}\omega_A-N_f^{\zeta})^{\delta_y}}\bigg)\txte^{\eta t}\\
		& + L_fC_A\int_{t_0}^t\frac{\txte^{\epsilon^{-1}\omega_A(t-s)}}{\epsilon^{\gamma_X}(t-s)^{1-\gamma_X}}\| h^{\epsilon,\zeta}_{X_1}(\bar{v}_S(s))-h^0(\bar{v}_S(s)) \|_{X_1}\,\txtd s
	\end{align*}}}
	Now, Lemma~\ref{Lemma:Gronwall_Specific} applied to
	\[
		v(r):=\| h^{\epsilon,\zeta}_{X_1}(\bar{v}_S(r+t_0))-h^0(\bar{v}_S(r+t_0)) \|_{X_1}\quad(r\in[0,t-t_0])
	\]
	yields that
	\begin{align*}
	&\quad\frac{1}{C\|v_0\|_{Y_1}}\| h^{\epsilon,\zeta}_{X_1}(\bar{v}_S(t))-h^0(\bar{v}_S(t)) \|_{X_1}\\
	&\leq \left(\big(\tfrac{1}{(\epsilon\eta-\omega_A)^{\gamma_X}}+1\big)+\tfrac{\epsilon}{(\epsilon\eta-\omega_A)^{\delta_X}} +\tfrac{1}{(\epsilon\eta-\omega_A)^{\gamma_X}(\eta-\zeta^{-1}\omega_A-N_f^{\zeta})^{\delta_y}}\right)\txte^{\eta t_0+\epsilon^{-1}\omega_f(t-t_0)}\\
	&+\left(\tfrac{\epsilon}{(\epsilon\eta-\omega_A)^{\delta_X}} +\tfrac{1}{(\epsilon\eta-\omega_A)^{\gamma_X}(\eta-\zeta^{-1}\omega_A-N_f^{\zeta})^{\delta_y}}\right)\int_{t_0}^{t} (\eta-\epsilon^{-1}\omega_A)\txte^{\eta s}\txte^{\epsilon^{-1}\omega_f(t-s)}\,\txtd s\\
	&=\left(\big(\tfrac{1}{(\epsilon\eta-\omega_A)^{\gamma_X}}+1\big)+\tfrac{\epsilon}{(\epsilon\eta-\omega_A)^{\delta_X}} +\tfrac{1}{(\epsilon\eta-\omega_A)^{\gamma_X}(\eta-\zeta^{-1}\omega_A-N_f^{\zeta})^{\delta_y}}\right)\txte^{\eta t_0+\epsilon^{-1}\omega_f(t-t_0)}\\
	&+\left(\tfrac{\epsilon}{(\epsilon\eta-\omega_A)^{\delta_X}} +\tfrac{1}{(\epsilon\eta-\omega_A)^{\gamma_X}(\eta-\zeta^{-1}\omega_A-N_f^{\zeta})^{\delta_y}}\right)\tfrac{\eta-\epsilon^{-1}\omega_A}{\eta-\epsilon^{-1}\omega_f}(\txte^{t\eta}-\txte^{\eta t_0+\epsilon^{-1}\omega_f(t-t_0)})
	\end{align*}
	Note that since $\eta>\zeta^{-1}\omega_A$, it follows from $\epsilon<c\zeta\frac{(L_fC_A\Gamma(\gamma_X))^{1/\gamma_X}+w_A}{w_A}$ that
	\[
		\eta>\zeta^{-1}\omega_A>c\epsilon^{-1}((L_fC_A\Gamma(\gamma_X))^{1/\gamma_X}+w_A)>c\epsilon^{-1}\omega_f.
	\]
	Hence, choosing $t=0$ and letting $t_0\to-\infty$ shows that
	\begin{align*}
		\| h^{\epsilon,\zeta}_{X_1}(v_0)-h^0(v_0) \|_{X_1}&\leq C\left(\tfrac{\epsilon}{(\epsilon\eta-\omega_A)^{\delta_X}} +\tfrac{1}{(\epsilon\eta-\omega_A)^{\gamma_X}(\eta-\zeta^{-1}\omega_A-N_f^{\zeta})^{\delta_y}}\right)\\
	&\qquad\qquad\qquad\cdot\,\tfrac{\eta-\epsilon^{-1}\omega_A}{\eta-\epsilon^{-1}\omega_f}\|v_0\|_{Y_1}.
	\end{align*}
	Since $\eta=\zeta^{-1}\omega_A+\frac{N_S^{\zeta}+N_F^{\zeta}}{2}$, it follows that
	\begin{align*}
		\| h^{\epsilon,\zeta}_{X_1}(v_0)-h^0(v_0) \|_{X_1}\leq C\left(\epsilon +\frac{1}{(N_S^{\zeta}-N_F^{\zeta})^{\delta_y}}\right)\|v_0\|_{Y_1}
	\end{align*}
	for some constant $C>0$. Moreover, \eqref{Eq:Slow_Manifold_Estimate_Second_Component} turns into 
		\begin{align*}
	\|h^{\epsilon,\zeta}_{Y^{\zeta}_F}(\bar{v}_S(t))\|_{Y_1}\leq C\frac{1}{(N_S^{\zeta}-N_F^{\zeta})^{\delta_y}} \|v_0 \|_{Y_1}.
	\end{align*}
	Altogether, we obtain the assertion.
\end{proof}

\subsection{Differentiability of the Slow Manifolds}\label{Sec:Slow_Manifolds:Differentiability}
Now, we suppose that the nonlinearities $f\colon X_1\times Y_1\to X_{\gamma_X}$ and $g\colon X_1 \times Y_1\to Y_{\delta_Y}$ are continuously differentiable such that
\begin{align}\label{Eq:Nonlinearities_Differentiable}
	\|\txtD f(x,y)\|_{\mathcal{B}(X_1\times Y_1,X_{\gamma_X})}\leq L_f,\quad \|\txtD g(x,y)\|_{\mathcal{B}(X_1\times Y_1,Y_{\delta_Y})}\leq L_g.
\end{align}
The aim is to show that
\[
	(Y_{S}^{\zeta},\|\cdot\|_{Y_1})\to (X_1,\|\cdot\|_{X_1})\times (Y_{F}^{\zeta},\|\cdot\|_{Y_{\delta_Y}}),\,v_0\mapsto (h_{X_1}^{\epsilon,\zeta}(v_0),h_{Y_F^{\zeta}}^{\epsilon,\zeta}(v_0))
\]
is differentiable.

\begin{proposition}\label{Prop:Slow_Manifolds:Differentiability}
	Under the general assumptions in this section and the differentiability assumptions in this subsection, the slow manifold $S_{\epsilon,\zeta}$ is differentiable.
\end{proposition}
\begin{proof}
	Given $v_0\in Y_S^{\zeta}$ we write $U(\,\cdot\,,v_0):=(u(\cdot,v_0),v_F(\cdot,v_0),v_S(\cdot,v_0))\in C_{\eta}$ for the fixed point of $\mathscr{L}_{v_0,\epsilon,\zeta}$. Fix $v_0,\tilde{v}_0\in Y_S^{\zeta}$. Effectively, any classical approach to show smoothness~\cite{Fenichel1,HirschPughShub,WigginsIM} is based around estimates, which show that the derivative exists as the best local linear approximation of the graph of the manifold. We follow this strategy and write
	\begin{align*}
		U(\,\cdot\,\tilde{v_0})-U(\,\cdot\,v_0)-T[U(\,\cdot\,\tilde{v_0})-U(\,\cdot\,v_0)]=\begin{pmatrix} 0 \\ 0 \\ \txte^{B(\cdot)}(\tilde{v}_0-v_0) \end{pmatrix}+ I(\tilde{v}_0,v_0),
	\end{align*}
	where
	\begin{align*}
		T\colon C_{\eta}\to C_{\eta},\,z\mapsto \left[t\mapsto \begin{pmatrix} \epsilon^{-1}\int_{-\infty}^t \txte^{\epsilon^{-1}(t-s)A}\txtD f(U(s,v_0))z(s)\,\txtd s \\ \int_{-\infty}^t \txte^{(t-s)B}\operatorname{pr}_{Y_F^{\zeta}}\txtD g(U(s,v_0))z(s)\,\txtd s \\ 0 \end{pmatrix}\right]
	\end{align*}
	 and $I(\tilde{v}_0,v_0)=(I_1,I_2,I_3)^T(\tilde{v}_0,v_0)$ where
	 \begin{align*}
	 	I_1(\tilde{v}_0,v_0)&=\bigg[t\mapsto \epsilon^{-1}\int_{-\infty}^t \txte^{\epsilon^{-1}(t-s)A}\big(f(U(s,\tilde{v}_0))-f(U(s,v_0))\\
	 	&\qquad\qquad\qquad-\txtD f(U(s,v_0))[U(s,\tilde{v_0})-U(s,v_0)]\big)\,\txtd s\bigg],\\
	 	I_2(\tilde{v}_0,v_0)&=\bigg[t\mapsto\int_{-\infty}^t \txte^{(t-s)B}\operatorname{pr}_{Y_F^{\zeta}}\big(g(U(s,\tilde{v}_0))-g(U(s,v_0))\\
	 	&\qquad\qquad\qquad-\txtD g(U(s,v_0))[U(s,\tilde{v_0})-U(s,v_0)]\big)\,\txtd s\bigg],\\
	 	I_3(\tilde{v}_0,v_0)&=0.
	 \end{align*}
	The aim is to show that $\|T\|_{\mathcal{B}(C_{\eta})}<1$ and that $$\|I(\tilde{v_0},v_0)\|_{X_1\times (Y_{F}^{\zeta}\cap Y_1)\times Y_{S}^{\zeta}}=o(\|\tilde{v}_0-v_0\|_{Y_1}) \quad\text{as}\quad \tilde{v}_0\to v_0.$$ Then we have
	\[
		U(0,\tilde{v_0})-U(0,v_0)=(1-T)^{-1}\begin{pmatrix} 0 \\ 0 \\ \txte^{B(\cdot)}(\tilde{v}_0-v_0) \end{pmatrix}+o(\|\tilde{v}_0-v_0\|_{Y_1})
	\]
	as $\tilde{v}_0\to v_0$ so that $U(0,\cdot\,)=(h_X^{\epsilon,\zeta},h_{Y_F^{\zeta}}^{\epsilon,\zeta},\operatorname{id}_{Y_S^{\zeta}})$ is differentiable. The fact that $\|T\|_{\mathcal{B}(X_1\times (Y_{F}^{\zeta}\cap Y_1)\times Y_{S}^{\zeta})}<1$ follows from the same computation as the one for showing that $\mathscr{L}_{v_0,\epsilon,\zeta}$ is a contraction in Proposition~\ref{Prop:Slow_Manifold_Existence}. Concerning $I$ one can treat both its components similarly. Hence, we only carry out the usual argument for the first component. By our assumptions on $f$, for all $\sigma>0$ there is an $N>0$ such that
	\begin{align*}
	&\quad\txte^{-\eta t} \bigg\|\epsilon^{-1}\int_{-\infty}^{\min\{-N,t\}} \txte^{\epsilon^{-1}(t-s)A}\big(f(U(s,\tilde{v}_0))-f(U(s,v_0))\\
	&\qquad\qquad\qquad\qquad-\txtD f(U(s,v_0))[U(s,\tilde{v_0})-U(s,v_0)]\big)\,\txtd s \bigg\|_{X_1}\\
	&\leq 2L_fC_A\|U(\cdot,\tilde{v_0})-U(\cdot,v_0)\|_{C_\eta} \int_{-\infty}^{\min\{-N,t\}}\frac{\txte^{(\epsilon^{-1}\omega_A-\eta)(t-s)}}{\epsilon^{\gamma}_X(t-s)^{1-\gamma_X}}\,\txtd s\leq \frac{\sigma}{2}\|\tilde{v}_0-v_0\|_{Y_1}
	\end{align*}
	for all $t\leq0$. Having fixed such an $N>0$, we obtain that
	\begin{align*}
	&\quad\txte^{-\eta t}\bigg\|\epsilon^{-1}\int_{\min\{-N,t\}}^{t} \txte^{\epsilon^{-1}(t-s)A}\big(f(U(s,\tilde{v}_0))-f(U(s,v_0))\\
	&\qquad\qquad\qquad-\txtD f(U(s,v_0))[U(s,\tilde{v_0})-U(s,v_0)]\big)\,\txtd s \bigg\|_{X_1}\\
	&\leq C_A\|U(\cdot,\tilde{v_0})-U(\cdot,v_0)\|_{C_\eta} \int_{\min\{-N,t\}}^{t}\frac{\txte^{(\epsilon^{-1}\omega_A-\eta)(t-s)}}{\epsilon^{\gamma}_X(t-s)^{1-\gamma_X}}\\
	&\int_0^1\big\| \txtD f\big(r U(s,\tilde{v}_0)-(1-r)U(s,v_0)\big)-\txtD f(U(s,v_0))\big\|_{\mathcal{B}(X_1\times (Y_F^{\zeta}\cap Y_1)\times Y_S^{\zeta}, X_{\gamma_X})}\,\txtd r\,\txtd s\\
	&\leq C\|\tilde{v_0}-v_0\|_{Y_1} \int_{\min\{-N,t\}}^{t}\frac{\txte^{(\epsilon^{-1}\omega_A-\eta)(t-s)}}{\epsilon^{\gamma}_X(t-s)^{1-\gamma_X}}\\
	&\int_0^1 \big\| \txtD f\big(r U(s,\tilde{v}_0)-(1-r)U(s,v_0)\big)-\txtD f(U(s,v_0))\big\|_{\mathcal{B}(X_1\times (Y_F^{\zeta}\cap Y_1)\times Y_S^{\zeta}, X_{\gamma_X})}\,\txtd r\,\txtd s.
	\end{align*}
	By dominated convergence and the continuity of the integrand, it follows that the integral is smaller than $\frac{\sigma}{2C}$ if $\tilde{v}_0$ is close enough to $v_0$. Thus, for all $\sigma>0$ there is a $\tilde{\sigma}>0$ such that for all $\tilde{v}_0\in Y_S^{\zeta}$ with $\|\tilde{v}_0-v_0\|_{Y_1}<\tilde{\sigma}$ and all $t\leq 0$ it holds that
	\begin{align*}
		&e^{-\eta t}\bigg\|\epsilon^{-1}\int_{-\infty}^{t} \txte^{\epsilon^{-1}(t-s)A}\big(f(U(s,\tilde{v}_0))-f(U(s,v_0))\\
		&\qquad\qquad-\txtD f(U(s,v_0))[U(0,\tilde{v_0})-U(0,v_0)]\big)\,\txtd s \bigg\|_{X_1}<\sigma\|\tilde{v}_0-v_0\|_{Y_1}.
	\end{align*}
	A similar computation can be carried out for the second component of $I$. Thus, we have that
	\[
	\|I(\tilde{v_0},v_0)\|_{C_{\eta}}=o(\|\tilde{v}_0-v_0\|_{Y_1})\quad\text{as}\quad \tilde{v_0}\to v_0
	\]
	which shows the differentiability of the slow manifolds.
\end{proof}

\subsection{Attraction of Trajectories}
Consider the situation of Proposition~\ref{Prop:Slow_Manifolds:Differentiability} and let $(h_{X_1}^{\epsilon,\zeta}(v_0),h_{Y_F^{\zeta}}^{\epsilon,\zeta}(v_0),v_0)\in S_{\epsilon,\zeta}$. Let $(u,v_F,v_S)$ be the solution of \eqref{Eq:Nonlinear_Fast_Slow_System:Splitting} with initial value $(h_{X_1}^{\epsilon,\zeta}(v_0),h_{Y_F^{\zeta}}^{\epsilon,\zeta}(v_0),v_0)$ and let $(u^{\epsilon},v_F^{\epsilon},v_S^{\epsilon})$ be the solution of \eqref{Eq:Nonlinear_Fast_Slow_System:Splitting} with initial value $(u_0,v_{0,F},v_{0,S})$. Since $(u,v_F,v_S)$ is a strict solution, it holds that
\[
	\partial_tu(t)=\epsilon^{-1}Au(t)+\epsilon^{-1}f(u(t),v_F(t),v_S(t))\quad(t\geq0).
\]
On the other hand, since $S_{\epsilon,\zeta}$ is invariant and since it is differentiable, it holds that $u(t)=h_{X_1}^{\epsilon,\zeta}(v_S(t))$ and therefore
\begin{align*}
	\partial_tu(t)&=\partial_th_{X_1}^{\epsilon,\zeta}(v_S(t))=\big(\txtD h_{X_1}^{\epsilon,\zeta}(v_S(t))\big)[\partial_tv_S(t)]\\
	&=\big(\txtD h_{X_1}^{\epsilon,\zeta}(v_S(t))\big)[Bv_S(t)+\operatorname{pr}_{Y_S^{\zeta}}g(u(t),v_F(t),v_S(t))]\quad(t\geq0).
\end{align*}
Combining both equations for $t=0$ and using $$(u(t),v_F(t))=(h_{X_1}^{\epsilon,\zeta}(v_S(t)),h_{Y_{F}^{\zeta}}^{\epsilon,\zeta}(v_S(t)))$$ yields that
\begin{align}
\begin{aligned}\label{Eq:Slow_Manifold_Representation_1}
	h_{X_1}^{\epsilon,\zeta}(v_0)&=\epsilon A^{-1}\big(\txtD h_{X_1}^{\epsilon,\zeta}(v_0)\big)[Bv_0+\operatorname{pr}_{Y_S^{\zeta}}g(h_{X_1}^{\epsilon,\zeta}(v_0),h_{Y_{F}^{\zeta}}^{\epsilon,\zeta}(v_0),v_0)]\\
	&\qquad-A^{-1}f(h_{X_1}^{\epsilon,\zeta}(v_0),h_{Y_{F}^{\zeta}}^{\epsilon,\zeta}(v_0),v_0).
	\end{aligned}
\end{align}
Similarly, it holds that
\begin{align}\begin{aligned}\label{Eq:Slow_Manifold_Representation_2}
	h_{Y_F^{\zeta}}^{\epsilon,\zeta}(v_0)&=B_{Y_F^{\zeta}}^{-1}\big(\txtD h_{Y_F^{\zeta}}^{\epsilon,\zeta}(v_0)\big)[Bv_0+\operatorname{pr}_{Y_S^{\zeta}}g(h_{X_1}^{\epsilon,\zeta}(v_0),h_{Y_{F}^{\zeta}}^{\epsilon,\zeta}(v_0),v_0)]\\
	&\qquad-B_{Y_F^{\zeta}}^{-1}f(h_{X_1}^{\epsilon,\zeta}(v_0),h_{Y_{F}^{\zeta}}^{\epsilon,\zeta}(v_0),v_0).
	\end{aligned}
\end{align}
Note that \eqref{Eq:Slow_Manifold_Representation_1} and \eqref{Eq:Slow_Manifold_Representation_2} hold for arbitrary $v_0\in Y_S^{\zeta}$. In particular, they also hold for $v_0=v_S^{\epsilon}(t)$.
In addition, the differentiability of $h_{X_1}^{\epsilon,\zeta}$ and $h_{Y_F^{\zeta}}^{\epsilon,\zeta}$ shows that
\begin{align}
	\partial_t h_{X_1}^{\epsilon,\zeta}(v_S^{\epsilon}(t))=\big(\txtD h_{X_1}^{\epsilon,\zeta}(v_S^{\epsilon}(t))\big)[Bv_S^{\epsilon}(t)+\operatorname{pr}_{Y_S^{\zeta}}g(u^{\epsilon}(t),v_F^{\epsilon}(t),v_S^{\epsilon}(t))], \label{Eq:Slow_Manifold:Some_Derivative_1} \\
	\partial_t h_{Y_F^{\zeta}}^{\epsilon,\zeta}(v_S^{\epsilon}(t))=\big(\txtD h_{Y_F^{\zeta}}^{\epsilon,\zeta}(v_S^{\epsilon}(t))\big)[Bv_S^{\epsilon}(t)+\operatorname{pr}_{Y_S^{\zeta}}g(u^{\epsilon}(t),v_F^{\epsilon}(t),v_S^{\epsilon}(t))]. \label{Eq:Slow_Manifold:Some_Derivative_2}
\end{align}
\begin{proposition}\label{Prop:Slow_Manifold:Attraction}
	Consider the situation of Proposition~\ref{Prop:Slow_Manifold_Existence} together with the assumptions of this subsection. If $\zeta$ and $\epsilon$ are small enough, then there are constants $C,c>0$ we have the estimate
	\begin{align*}
		\left\|\begin{pmatrix} u^{\epsilon}(t)-h^{\epsilon,\zeta}_{X_1}(v^{\epsilon}_S(t))\\
						v^{\epsilon}_F(t)-h^{\epsilon,\zeta}_{Y_F^{\zeta}}(v^{\epsilon}_S(t))\end{pmatrix}\right\|_{X_1\times Y_1}\leq C\txte^{-ct}\left\|\begin{pmatrix} u_0-h^{\epsilon,\zeta}_{X_1}(v_{0,S})\\
						v_{0,F}-h^{\epsilon,\zeta}_{Y_F^{\zeta}}(v_{0,S})\end{pmatrix}\right\|_{X_1\times Y_1},
	\end{align*}
	i.e. solutions of \eqref{Eq:Nonlinear_Fast_Slow_System} approach the solutions on the slow manifold at an exponential rate.
\end{proposition}
\begin{proof}
It holds that
{\allowdisplaybreaks{
\begin{align*}
	u^{\epsilon}(t)-h^{\epsilon,\zeta}_{X_1}(v^{\epsilon}_S(t))&=\txte^{\epsilon^{-1}t A}\big(u_0-h^{\epsilon,\zeta}_{X_1}(v_{0,S})\big)+\txte^{\epsilon^{-1}t A}h^{\epsilon,\zeta}_{X_1}(v_{0,S})-h^{\epsilon,\zeta}_{X_1}(v^{\epsilon}_S(t))\\
&\quad+\epsilon^{-1}\int_0^t\txte^{\epsilon^{-1}(t-s)A}f(u^{\epsilon}(s),v_F^{\epsilon}(s),v_S^{\epsilon}(s))\,\txtd s\\
	&=\txte^{\epsilon^{-1}t A}\big(u_0-h^{\epsilon,\zeta}_{X_1}(v_{0,S})\big)-\int_0^t\partial_s\big(\txte^{\epsilon^{-1}(t-s)A}h^{\epsilon,\zeta}_{X_1}(v^{\epsilon}_S(s))\big)\,\txtd s\\
	&\quad+\epsilon^{-1}\int_0^t\txte^{\epsilon^{-1}(t-s)A}f(u^{\epsilon}(s),v_F^{\epsilon}(s),v_S^{\epsilon}(s))\,\txtd s\\
	&=\txte^{\epsilon^{-1}t A}\big(u_0-h^{\epsilon,\zeta}_{X_1}(v_{0,S})\big)+\int_0^t\txte^{\epsilon^{-1}(t-s)A}\epsilon^{-1}Ah^{\epsilon,\zeta}_{X_1}(v^{\epsilon}_S(s))\,\txtd s\\
	&\quad-\int_0^t\txte^{\epsilon^{-1}(t-s)A}\partial_s[h^{\epsilon,\zeta}_{X_1}(v^{\epsilon}_S(s))]\,\txtd s\\
	&\quad+\epsilon^{-1}\int_0^t\txte^{\epsilon^{-1}(t-s)A}f(u^{\epsilon}(s),v_F^{\epsilon}(s),v_S^{\epsilon}(s))\,\txtd s
\end{align*}}}
Combining this with \eqref{Eq:Slow_Manifold_Representation_1} and \eqref{Eq:Slow_Manifold:Some_Derivative_1} yields
\begin{align*}
	u^{\epsilon}&(t)-h^{\epsilon,\zeta}_{X_1}(v^{\epsilon}_S(t))=\txte^{\epsilon^{-1}t A}\big(u_0-h^{\epsilon,\zeta}_{X_1}(v_{0,S})\big)\\
&+\int_0^t\txte^{\epsilon^{-1}(t-s)A}\big(\txtD h_{X_1}^{\epsilon,\zeta}(v_S^{\epsilon}(s))\big)\big[\operatorname{pr}_{Y_S^{\zeta}}g(u^{\epsilon}(s),v_F^{\epsilon}(s),v_S^{\epsilon}(s))\\
&\qquad\qquad\qquad-\operatorname{pr}_{Y_S^{\zeta}}g(h^{\epsilon,\zeta}_{X_1}(v^{\epsilon}_S(s)),h^{\epsilon,\zeta}_{Y_F^{\zeta}}(v^{\epsilon}_S(s)),v_S^{\epsilon}(s))\big]\,\txtd s\\
	&+\epsilon^{-1}\int_0^t\txte^{\epsilon^{-1}(t-s)A}\big[f(u^{\epsilon}(s),v_F^{\epsilon}(s),v_S^{\epsilon}(s))-f(h^{\epsilon,\zeta}_{X_1}(v^{\epsilon}_S(s)),h^{\epsilon,\zeta}_{Y_F^{\zeta}}(v^{\epsilon}_S(s)),v_S^{\epsilon}(s))\big]\,\txtd s.
\end{align*}
Similarly, it holds that
\begin{align*}
	v_F^{\epsilon}(t)&-h^{\epsilon,\zeta}_{Y_F^{\zeta}}(v^{\epsilon}_S(t))=\txte^{tB}\big(v_{0,F}-h^{\epsilon,\zeta}_{Y_F^{\zeta}}(v_{0,S})\big)\\
&+\int_0^t\txte^{(t-s)B}\big(\txtD h_{Y_F^{\zeta}}^{\epsilon,\zeta}(v_S^{\epsilon}(s))\big)\big[\operatorname{pr}_{Y_S^{\zeta}}g(u^{\epsilon}(s),v_F^{\epsilon}(s),v_S^{\epsilon}(s))\\
&\qquad\qquad\qquad-\operatorname{pr}_{Y_S^{\zeta}}g(h^{\epsilon,\zeta}_{X_1}(v^{\epsilon}_S(s)),h^{\epsilon,\zeta}_{Y_F^{\zeta}}(v^{\epsilon}_S(s)),v_S^{\epsilon}(s))\big]\,\txtd s\\
	&+\int_0^t\txte^{(t-s)B}\big[f(u^{\epsilon}(s),v_F^{\epsilon}(s),v_S^{\epsilon}(s))-f(h^{\epsilon,\zeta}_{X_1}(v^{\epsilon}_S(s)),h^{\epsilon,\zeta}_{Y_F^{\zeta}}(v^{\epsilon}_S(s)),v_S^{\epsilon}(s))\big]\,\txtd s.
\end{align*}
Thus, if we define $$\varphi(t):=\|u^{\epsilon}(t)-h^{\epsilon,\zeta}_{X_1}(v^{\epsilon}_S(t))\|_{X_1}+\|v^{\epsilon}_F(t)-h^{\epsilon,\zeta}_{Y_F^{\zeta}}(v^{\epsilon}_S(t))\|_{Y_1},$$
then we obtain
\begin{align*}
	\varphi(t)\leq& M_A\txte^{\epsilon^{-1}\omega_A t} \| u_0-h^{\epsilon,\zeta}_{X_1}(v_{0,S})\|_{X_1}+ M_B\txte^{(\zeta^{-1}\omega_A+N_F^{\zeta}) t}\|v_{0,F}-h^{\epsilon,\zeta}_{Y_F^{\zeta}}(v_{0,S})\|_{Y_1}\\
	&+C_AL_f(\epsilon L_g+1)\int_0^t\frac{\txte^{\epsilon^{-1}(t-s)\omega_A}}{\epsilon^{\gamma_X}(t-s)^{1-\gamma_X}}\varphi(s)\,\txtd s\\
	&+C_B(L_g^2+L_f)\int_0^t\frac{\txte^{(t-s)(\zeta^{-1}\omega_A+N_F^{\zeta})}}{(t-s)^{1-\delta_y}}\varphi(s)\,\txtd s.
\end{align*}
Now we can apply Lemma~\ref{Lemma:Gronwall_Specific_Sum} and obtain that there are constants $C,c>0$ such that
\[
	\varphi(t)\leq C\txte^{-ct}\left\|\begin{pmatrix} u_0-h^{\epsilon,\zeta}_{X_1}(v_{0,S})\\
						v_{0,F}-h^{\epsilon,\zeta}_{Y_F^{\zeta}}(v_{0,S})\end{pmatrix}\right\|_{X_1\times Y_1}.
\]
This is the assertion.
\end{proof}

\subsection{An Approximation of the Slow Flow}
In Section~\ref{Sec:Slow_Manifolds:Distance_Critical_Manifold} we measured the distance of the slow manifolds to the subset $S_{0,\zeta}$ of the critical manifold given by
\[
	S_{0,\zeta}:=\{ (h^0(v_0),v_0)\in S_0 : \operatorname{pr}_{Y_F^{\zeta}}v_0=0 \}
\]
In many cases $S_{0,\zeta}$ will not be invariant under the slow flow. Thus, one might wonder how meaningful the result in Section~\ref{Sec:Slow_Manifolds:Distance_Critical_Manifold} is. However, our aim is not to reduce the fast-slow system \eqref{Eq:Nonlinear_Fast_Slow_System} with $\epsilon>0$ to the slow subsystem \eqref{Eq:Nonlinear_Fast_Slow_System} with $\epsilon=0$, but to the reduced slow subsystem:
\begin{align}
	\begin{aligned}\label{Eq:Reduced_Slow_Subsystem}
		0&=Au^{0}_{\zeta}(t)+f(u^{0}_{\zeta}(t),v^0_{\zeta}(t)),\\
		0&=\operatorname{pr}_{Y_F^{\zeta}}v^0_{\zeta}(t),\\
		\partial_tv^0_{\zeta}(t)&=Bv^0_{\zeta}(t)+\operatorname{pr}_{Y_S^{\zeta}}g(u^{0}_{\zeta}(t),v^0_{\zeta}(t)),\\
		v^0_{\zeta}(0)&=\operatorname{pr}_{Y_S^{\zeta}}v_0.
	\end{aligned}
\end{align}
Obviously, $S_{0,\zeta}$ is invariant under the reduced slow flow generated by \eqref{Eq:Reduced_Slow_Subsystem}.

\begin{proposition}\label{Prop:Reduced_Slow_Flow}
	For all $T>0$ there is a constant $C>0$ such that for all $t\in[0,T]$ and all $\zeta>0$ small enough it holds that
	\begin{align*}
		\|v^0(t)-v^0_{\zeta}(t)\|_{Y_1}\leq C\left(\|\operatorname{pr}_{Y_F^{\zeta}}v_0\|_{Y_1}+\frac{\|v_0\|_{Y_1}}{(\omega_B-\zeta^{-1}\omega_A-N_F^{\zeta})^{\delta_Y}}\right).
	\end{align*}
\end{proposition}
\begin{proof}
	Variation of constants shows that
	\begin{align*}
		\|v^0(t&)-v^0_{\zeta}(t)\|_{Y_1}\leq M_B\txte^{(\zeta^{-1}\omega_A+N_F^{\zeta})t}\|\operatorname{pr}_{Y_F^{\zeta}}v_0\|_{Y_1}\\
		&\quad+L_gC_B\int_0^t\frac{\txte^{(\zeta^{-1}\omega_A+N_F^{\zeta})(t-s)}}{(t-s)^{\delta_Y}}\left(\|(h^0(v^0(s))\|_{X_1}+\|v^0(s)\|_{Y_1}\right)\,\txtd s\\
		&\quad+L_gC_B\int_0^t\frac{\txte^{\omega_B(t-s)}}{(t-s)^{\delta_Y}}\left(\|(h^0(v^0(s))-h^0(v^0_{\zeta}(s))\|_{X_1}+\|v^0(s)-v^0_{\zeta}(s)\|_{Y_1}\right)\,\txtd s\\
		&\leq C\txte^{\omega_B t}\left(\|\operatorname{pr}_{Y_F^{\zeta}}v_0\|_{Y_1}+\frac{\|v_0\|_{Y_1}}{(\omega_B-\zeta^{-1}\omega_A-N_F^{\zeta})^{\delta_Y}}\right)\\
		&\quad+\tfrac{L_F\|A^{-1}\|_{\mathcal{B}(X_{\delta_X-1},X_{\delta_X})}L_gC_B}{1-L_F\|A^{-1}\|_{\mathcal{B}(X_{\delta_X-1},X_{\delta_X}}}\int_0^t\frac{\txte^{\omega_B(t-s)}}{(t-s)^{\delta_Y}}\|v^0(s)-v^0_{\zeta}(s)\|_{Y_1}\,\txtd s
	\end{align*}
	Now the assertion follows from Lemma~\ref{Lemma:Gronwall_Specific}.
\end{proof}

\begin{corollary}\label{Cor:Convergence_of_Flows}
	Consider the situation of Proposition~\ref{Prop:Slow_Manifold:Attraction} For all $T>0$ there is a constant $C>0$ such that for all $t\in[0,T]$ and all $\epsilon, \zeta>0$ satisfying the usual assumptions  it holds that
	\begin{align*}
		\left\|\begin{pmatrix}u^{\epsilon}(t)-h^0(v^{0}_{\zeta}(t))\\ v^{\epsilon}(t)-v^{0}_{\zeta}(t) \end{pmatrix}\right\|_{Y_1}
		&\leq C\bigg(\|\operatorname{pr}_{Y_F^{\zeta}}v_0\|_{Y_1}+\big(\epsilon+\tfrac{1}{(\omega_B-\zeta^{-1}\omega_A-N_F^{\zeta})^{\delta_Y}}\big)\|v_0\|_{Y_1}\\
		&\qquad\qquad+(\epsilon^{\delta_Y}+\txte^{\epsilon^{-1}\omega_f t})\|u_0-h^0(v_0)\|_{X_1}\bigg).
	\end{align*}
	In particular, for initial values on the slow manifold it holds that
		\begin{align*}
		\left\|\begin{pmatrix}u^{\epsilon}(t)-h^0(v^{0}_{\zeta}(t))\\ v^{\epsilon}(t)-v^{0}_{\zeta}(t) \end{pmatrix}\right\|_{Y_1}\leq C\left(\epsilon+\tfrac{1}{(\omega_B-\zeta^{-1}\omega_A-N_F^{\zeta})^{\delta_Y}}+\tfrac{1}{(N_S^{\zeta}-N_F^{\zeta})^{\delta_Y}}\right)\|v_0\|_{Y_1}.
	\end{align*}
\end{corollary}
\begin{proof}
	The first estimate is a combination of  Corollary~\ref{Cor:Original_and_Slow_Flow_Close} and Proposition~\ref{Prop:Reduced_Slow_Flow}. For the second estimate, we use the first estimate together with Proposition~\ref{Prop:Distance_Critical_Manifold} and the triangle inequality
	\begin{align*}
	\|\operatorname{pr}_{Y_F^{\zeta}}v_0\|_{Y_1}&\leq \|\operatorname{pr}_{Y_F^{\zeta}}v_0-h^{\epsilon,\zeta}_{Y_F^{\zeta}}(v_0)\|_{Y_1}+\|h^{\epsilon,\zeta}_{Y_F^{\zeta}}(v_0)\|_{Y_1},\\
	\|u_0-h^0(v_0)\|_{X_1}&\leq \|u_0-h^{\epsilon,\zeta}_{X_1}(v_0)\|_{X_1}+\|h^{\epsilon,\zeta}_{X_1}(v_0)-h^0(v_0)\|_{X_1}.
	\end{align*}
\end{proof}

\begin{remark}
\begin{enumerate}[(a)]
	\item Note that we do not need the existence of slow manifolds for the first estimate in Corollary~\ref{Cor:Convergence_of_Flows}.
	\item If the initial values are not on the slow manifold, then it looks like the term $\|\operatorname{pr}_{Y_F^{\zeta}}v_0\|_{Y_1}$ might prevent the trajectories of semiflow generated by the fast-slow system from converging to the ones of the reduced slow flow as $\epsilon,\zeta\to0$. However, sometimes it holds that $\|\operatorname{pr}_{Y_F^{\zeta}}v_0\|_{Y_1}\to 0$ as $\zeta\to0$ uniformly in $v_0$ running through certain sets. For example, if one takes $v_0$ from a bounded set in $Y_2$, then this will hold in many situations.
\end{enumerate}
\end{remark}

\section{Three Examples}
\label{sec:examples}

\subsection{The Spatial Stommel Model}
Now we apply our methods to a version of Stommel's box model for oceanic circulation in the North Atlantic (see \cite{Stommel_1961} and \cite[(6.2.4)]{Berglund_Gentz_2006}) in which we add diffusion in both variables. These equations are then given by
\begin{align}
	\begin{aligned}\label{Eq:Stommel}
	\epsilon \partial_t u^{\epsilon}&=\Delta u^{\epsilon}-u^{\epsilon}+1-\epsilon u^{\epsilon}[1+\eta^2((u^{\epsilon})^2-(w^{\epsilon})^2)],\\
	\partial_t w^{\epsilon}&= \Delta w^{\epsilon}+\mu- w^{\epsilon}[1+\eta^2((u^{\epsilon})^2-(w^{\epsilon})^2)],\\
	u^{\epsilon}(0)&=u_0,\quad w^{\epsilon}(0)=v_0,
	\end{aligned}
\end{align}
where $\mu,\eta>0$ are certain parameters. We study this system with periodic boundary conditions, i.e. on the $\mathbb{T}$, and partly also in $\mathbb{T}^n$ and $\R^n$, but we will carry out our arguments on the torus. For this example, we work with toroidal Bessel potential spaces $H^{s}_2(\mathbb{T}^n)$. The space $H^s_2(\mathbb{T}^n)$ with $s\geq0$ is defined as the space of all $f\in L_2(\mathbb{T}^n)$ such that
		\[
			\|f\|_{	H^s_2(\mathbb{T}^n) }:= \left\|x\mapsto  \sum_{k\in\Z^n} (1+|k|^2)^{\frac{s}{2}}\hat{f}(k)\txte^{ikx}\right\|_{L_2(\mathbb{T}^n)}<\infty,
		\]
				 where $\hat{f}(k)$ denotes the $k$-th Fourier coefficient. Let us formulate our main results for the diffusive Stommel model.
\begin{theorem}\label{Thm:Stommel:Approximation_Slow_Flow}
	Let $\mathbb{E}\in\{\mathbb{T},\mathbb{R}\}$, i.e. let $\mathbb{E}$ either be the torus or the real line. Let further $s\geq0$ and $\delta_Y\in(\tfrac{1}{2},1)$ such that $2s+4(1-\delta_Y)>n$ and let $T>0$ be fixed. We write $(u^{\epsilon},w^{\epsilon})$ for the strict solution of \eqref{Eq:Stommel} with $\epsilon>0$ and $(u^{0},w^{0})$ for corresponding slow flow. Then for all $R>0$ there are constants $\epsilon_0>0$ and $C,c>0$ such that that for all $\epsilon\in(0,\epsilon_0]$, $u_0\in H^{s+2}_2(\mathbb{E}^n)$ with $\|u_0\|_{H^{s+2}_2(\mathbb{E}^n)}\leq R$ and $v_0\in H^{s+2+2(1-\delta_Y)}_2(\mathbb{E}^n)$ with $\|u_0\|_{H^{s+2+2(1-\delta_Y)}_2(\mathbb{E}^n)}\leq R$ it holds that
	\[
		\sup_{0\leq t \leq T(R)}\big(\|u^{\epsilon}(t)-u^0(t)\|_{H^{s+2}_2(\mathbb{E}^n)}+\|w^{\epsilon}(t)-v^0(t)\|_{H^{s+2+2(1-\delta_Y)}_2(\mathbb{E}^n)}\big)\leq C(\epsilon^{\delta_Y}+\txte^{-c\epsilon^{-1}t}),
	\]
	where $T(R)$ is defined by
	\begin{align*}
		T(R):=\inf\big\{t\in[0,T]: \max\{&\|u^0(t)\|_{H^{s+2}_2(\mathbb{E}^n)},\|w^0(t)\|_{H^{s+2+2(1-\delta_Y)}_2(\mathbb{E}^n)},\\
		&\|u^{\epsilon}(t)\|_{H^{s+2}_2(\mathbb{E}^n)},\|w^{\epsilon}(t)\|_{H^{s+2+2(1-\delta_Y)}_2(\mathbb{E}^n)} \}> R\big\}.
	\end{align*}
\end{theorem}

\begin{theorem}\label{Thm:Stommel:Slow_Manifolds}
	Let $n=1$, $s\geq0$ and $\delta_Y\in(\tfrac{1}{2},1)$ such that $2s+4(1-\delta_Y)>1$ and let $T>0$ be fixed. Then for all $R>0$ there are $\zeta_0>0$ and a family of finite-dimensional slow manifolds $S_{\epsilon,\zeta}\subset H^{s+2}_2(\mathbb{T})\times H^{s+2+2(1-\delta_Y)}_2(\mathbb{T}) $ with $0<\zeta\leq\zeta_0$ and $0<\epsilon\leq c\frac{(L_fC_A\Gamma(\gamma_X))^{1/\gamma_X}+\omega_A}{\omega_A}\zeta$ for constants $\omega_f,\omega_A$ which we define later and some $c\in(0,1)$ such that the following assertions hold:
	\begin{enumerate}[(a)]
		\item For each $\zeta\in(0,\zeta_0]$ there is a splitting
			\[
				H^{s+2(1-\delta_Y)}_2(\mathbb{T})=Y_F^{\zeta}\oplus Y_S^{\zeta},
			\]
			where $Y_S^{\zeta}$ is the projection of $H^{s+2(1-\delta_Y)}_2(\mathbb{T})$ to the $k$-th Fourier modes with $|k|$ being smaller than a certain number $k(\zeta)$ depending on $\zeta$. $Y_F^{\zeta}$ is to projection to the remaining Fourier modes.
		\item Let $B_{Y_S^{\zeta}}(0,R)$ be defined as
		\[
				B_{Y_S^{\zeta}}(0,R):=\{ f\in Y_S^{\zeta}: \|f\|_{H^{s+2+2(1-\delta_Y)}_2(\mathbb{T})}< R\}	.
		\]
		Then $S_{\epsilon,\zeta}$ is given as the graph of a differentiable mapping
			\begin{align*}
				h^{\epsilon,\zeta}&\colon (B_{Y_S^{\zeta}}(0,R),\|\cdot\|_{H^{s+2+2(1-\delta_Y)}_2(\mathbb{T})})\\
				&\qquad\qquad\to H^{s+2}_2(\mathbb{T})\times (Y_F^{\zeta}\cap H^{s+2+2(1-\delta_Y)}_2(\mathbb{T}),\|\cdot\|_{H^{s+2+2(1-\delta_Y)}_2(\mathbb{T})}).
			\end{align*}
		\item $S_{\epsilon,\zeta}$ is locally invariant under the semiflow generated by \eqref{Eq:Stommel}, i.e. the semiflow can only leave $S_{\epsilon,\zeta}$ through its boundary.
		\item Let 
			\[
				S_{0,\zeta}:=\{ (u,w)\in S_0: w\in B_{Y_S^{\zeta}}(0,R) \}
			\]
			be the submanifold of the critical manifold which consists of all points whose slow components are elements of $B_{Y_S^{\zeta}}(0,R)$. Then there is a constant $C>0$ depending on $R$ such that
			\[
				\operatorname{dist}(S_{\epsilon,\zeta},S_{0,\zeta})\leq C(\epsilon+\zeta^{\delta_Y/2})\leq C\zeta^{\delta_Y/2}.
			\]
		\item \label{Thm:Stommel:Slow_Manifolds:Flow_Convergence} Suppose that $\|u_0\|_{H^{s+2}_2(\mathbb{T})}\leq R$, $\|v_0\|_{H^{s+2+2(1-\delta_Y)}_2(\mathbb{T})}\leq R$, $\|h^0(v_0)\|_{H^{s+2}_2(\mathbb{T})}\leq R$ and let $(u^0_{\zeta},w^0_{\zeta})$ be the solution of the truncated slow subsystem of the diffusive Stommel model given by 
		\begin{align}
	\begin{aligned}\label{Eq:Stommel_truncated}
	0&=\Delta u^0_{\zeta}-u^0_{\zeta}+1,\\
	\partial_t w^0_{\zeta}&= \operatorname{pr}_{Y_S^{\zeta}}\left[\Delta w^0_{\zeta}+\mu- w^{\epsilon}[1+\eta^2((u^{0}_{\zeta})^2-(w^0_{\zeta})^2)]\right],\\
	u^{0}_{\zeta}&=h^0(\operatorname{pr}_{Y_S^{\zeta}}v_0),\quad w^0_{\zeta}(0)=\operatorname{pr}_{Y_S^{\zeta}}v_0.
	\end{aligned}
\end{align}
Assume that $(u_0,v_0)\in S_{\epsilon,\zeta}$. Then for each $T>0$ there is a constant $C>0$ such that
\[
		\sup_{0\leq t \leq T(R)}\big(\|u^{\epsilon}(t)-u^0_{\zeta}(t)\|_{H^{s+2}_p(\mathbb{T})}+\|w^{\epsilon}(t)-w^{0}_{\zeta}(t)\|_{H^{s+2+2(1-\delta_Y)}_p(\mathbb{T})}\big)\leq C\zeta^{\delta_Y/2},
	\]
	where $T(R)$ is defined by
	\begin{align*}
		T(R):=\inf\big\{t\in[0,T]: \max\{&\|u^0_{\zeta}(t)\|_{H^{s+2}_p(\mathbb{T})},\|w^0_{\zeta}(t)\|_{H^{s+2+2(1-\delta_Y)}_p(\mathbb{T})},\\
		&\|u^{\epsilon}(t)\|_{H^{s+2}_p(\mathbb{T})}, \|w^{\epsilon}(t)\|_{H^{s+2+2(1-\delta_Y)}_p(\mathbb{T})} \}> R\big\}.
	\end{align*}
	\end{enumerate}
\end{theorem}

\begin{remark}
\begin{enumerate}[(a)]
	\item In both theorems the condition $2s+4(1-\delta_Y))>n$ is not essential, but cutoff techniques would get more tedious without this assumption. This condition has the advantage that the nonlinearities are already well-defined and locally Lipschitz continuous in the spaces we work with later on without having to cut them off. Cutoff techniques are then only required to turn local Lipschitz continuity into global Lipschitz continuity.
	\item In Theorem~\ref{Thm:Stommel:Approximation_Slow_Flow} we may allow $\R^n$ or $\mathbb{T}^n$ as underlying domains, as its proof only uses the results of Section~\ref{sec:genfs}, which do not require spectral gaps in the slow variable. Both cases can essentially be derived in the same way and thus, we only explain how to derive Theorem~\ref{Thm:Stommel:Approximation_Slow_Flow}  for $\mathbb{E}=\mathbb{T}$. For $\mathbb{E}=\R$ we would have to replace the integers $\Z$ in the Fourier image by the real numbers $\R$, Fourier series by inverse Fourier transforms and Fourier coefficients by the Fourier transform.
	\item In Theorem~\ref{Thm:Stommel:Slow_Manifolds} we can not replace $\mathbb{T}$ by $\R$, $\R^n$ or $\mathbb{T}^n$ with $n\geq2$, as there are no or only small spectral gaps. Hence, we would not be able to construct slow manifolds or we would not obtain the same convergence results for small $\zeta$.
	\item We could also work in $H^s_p$ with $p\neq2$. But then the proofs would be more complicated since we would have to use Fourier multiplier theorems instead of just Plancherel's theorem.
	\end{enumerate}
\end{remark}
Now we show how our general theory can be applied to derive Theorem~\ref{Thm:Stommel:Approximation_Slow_Flow} and Theorem~\ref{Thm:Stommel:Slow_Manifolds}.

In order to remove constants in the nonlinear terms, we introduce the dummy variable $\tilde{w}$ which takes values in $\R^3$ and satisfies
\[
	\partial_t\tilde{w}^{\epsilon}=0,\quad \tilde{w}^{\epsilon}(0)=(\sqrt{\epsilon},M,\mu)
\]
for some $M>0$ that we choose later. We make the following choices:
\begin{itemize}
	\item The fast variable is given by $u^{\epsilon}$. The slow variable is given by $v^{\epsilon}=(w^{\epsilon},\tilde{w}^{\epsilon}_1,\tilde{w}^{\epsilon}_2,\tilde{w}^{\epsilon}_3)$. As underlying spaces we choose $X=H^{s}_2(\mathbb{T}^n)$ and $Y=H^{s+2(1-\delta_Y)}_2(\mathbb{T}^n)\times\R^3$ such that $2s+4(1-\delta_Y))>n$.
	\item The linear operator in the fast variable is given by
		\[
			A:H^s_2(\mathbb{T}^n)\supset H^{s+2}_2(\mathbb{T}^n)\to H^s_2(\mathbb{T}^n),\,u\mapsto \Delta u-u.
		\]
The linear operator in the slow variable is given by
		\begin{align*}
			B&:H^{s+2(1-\delta_Y)}_2(\mathbb{T}^n)\times \R^3\supset H^{s+2+2(1-\delta_Y)}_2(\mathbb{T}^n)\times\R^3\to H^{s+2(1-\delta_Y)}_2(\mathbb{T}^n)\times\R^3,\\
			&\qquad\qquad(v,z_1,z_2,z_3)^T\mapsto (\Delta v_1, -z_1,-z_2,-z_3)^T
		\end{align*}	
		for some $\delta_Y\in(\tfrac{1}{2},1)$; we compensate the terms $z_j\mapsto -z_j$ from the linear part by inserting maps $z_j\mapsto z_j$ in the nonlinear part defined below.
	\item The Banach scales are given by $$X_{\alpha}=H^{s+2\alpha}_2(\mathbb{T}^n)\quad\text{and}\quad Y_{\alpha}=H^{s+2(1-\delta_Y)+2\alpha}_2(\mathbb{T}^n)\times\R^3.$$
	\item We have already chosen $\delta_Y\in(\tfrac{1}{2},1)$. Moreover, we take  $\gamma_X=1-\delta_Y$ and $\delta_X=1$. Thus, we have to define continuous nonlinearities
	\begin{align*}
		f\colon X_1\times Y\to X,\quad g\colon X_1\times Y_{1}\to Y_{\delta_Y}
	\end{align*}
	satisfying the Lipschitz conditions
	 \begin{align*}
 	\|f(x_1,y_1)-f(x_2,y_2) \|_{X_{1-\delta_Y}}&\leq L_f\big(\|x_1-x_2\|_{X_1}+\|y_1-y_2\|_{Y_1}\big),\\
 	\|f(u_1,v_1)-f(u_2,v_2)\|_{C^1([0,t];X)}&\leq L_f\big(\|u_1-u_2\|_{C^1([0,t];X_{1})}\\
 	&\qquad\qquad+\|v_1-v_2\|_{C^1([0,t];Y)}\big),\\
 	\|g(x_1,y_1)-g(x_2,y_2) \|_{Y_{\delta_Y}}&\leq L_g\big(\|x_1-x_2\|_{X_1}+\|y_1-y_2\|_{Y_1}\big)
 \end{align*}
 With our choices of spaces this translates into 
 	\begin{align}
 	\begin{aligned}\label{Eq:Stommel:Nonlinearity_Condition1}
		f&\colon H^{s+2}_2(\mathbb{T}^n)\times H^{s+2(1-\delta_Y)}_2(\mathbb{T}^n)\times\R^3\to H^{s}_2(\mathbb{T}^n),\\
		g&\colon H^{s+2}_2(\mathbb{T}^n)\times H^{s+2+2(1-\delta_Y)}_2(\mathbb{T}^n)\times\R^3\to H^{s+2}_2(\mathbb{T}^n)\times \R^3
		\end{aligned}
	\end{align}
	and 
		 \begin{align}
		 \begin{aligned}\label{Eq:Stommel:Nonlinearity_Condition2}
 	&\|f(x_1,y_1)-f(x_2,y_2) \|_{H^{s+2(1-\delta_Y)}_2(\mathbb{T}^n)}\\
 	&\qquad\qquad\qquad\leq L_f\big(\|x_1-x_2\|_{H^{s+2}_2(\mathbb{T}^n)}+\|y_1-y_2\|_{H^{s+2+2(1-\delta_Y)}_2(\mathbb{T}^n)\times\R^3}\big),\\
 	&\|f(u_1,v_1)-f(u_2,v_2)\|_{C^1([0,t];H^{s}_2(\mathbb{T}^n))}\\&\qquad\qquad\qquad\leq L_f\big(\|u_1-u_2\|_{C^1([0,t];H^{s+2}_2(\mathbb{T}^n)))}+\|v_1-v_2\|_{C^1([0,t];H^{s+2(1-\delta_Y)}_2(\mathbb{T}^n)\times\R^3))}\big),\\
 	&\|g(x_1,y_1)-g(x_2,y_2) \|_{H^{s+2}_2(\mathbb{T}^n)\times\R^3}\\&\qquad\qquad\qquad\leq L_g\big(\|x_1-x_2\|_{H^{s+2}_2(\mathbb{T}^n)}+\|y_1-y_2\|_{H^{s+2+2(1-\delta_Y)}_2(\mathbb{T}^n)\times\R^3}\big).
 	\end{aligned}
 \end{align}
 Note that if 
 \begin{align}
 	\begin{aligned}\label{Eq:Stommel:Nonlinearity_Condition3}
		f&\colon H^{s+2}_2(\mathbb{T}^n)\times H^{s+2(1-\delta_Y)}_2(\mathbb{T}^n)\times\R^3\to H^{s+2(1-\delta_Y)}_2(\mathbb{T}^n),\\
		g&\colon H^{s+2}_2(\mathbb{T}^n)\times H^{s+2+2(1-\delta_Y)}_2(\mathbb{T}^n)\times\R^3\to H^{s+2}_2(\mathbb{T}^n)\times \R^3
		\end{aligned}
	\end{align}
	are differentiable with
	\begin{align}
	\begin{aligned}\label{Eq:Stommel:Nonlinearity_Condition4}
		\|\txtD f(x,y)\|_{\mathcal{B}(H^{s+2}_2(\mathbb{T}^n)\times H^{s+2(1-\delta_Y)}_2(\mathbb{T}^n)\times\R^3,H^{s+2(1-\delta_Y)}_2(\mathbb{T}^n))}\leq L_f,\\
		\|\txtD g(x,y)\|_{\mathcal{B}(H^{s+2}_2(\mathbb{T}^n)\times H^{s+2+2(1-\delta_Y)}_2(\mathbb{T}^n)\times\R^3,H^{s+2}_2(\mathbb{T}^n)\times \R^3))}\leq L_g,
		\end{aligned}
	\end{align}
	then both \eqref{Eq:Stommel:Nonlinearity_Condition1} and \eqref{Eq:Stommel:Nonlinearity_Condition2} as well as \eqref{Eq:Nonlinearities_Differentiable} are satisfied.
	Since $H^{s}_2(\mathbb{T}^n)$ is a multiplication algebra whenever $2s>n$, it follows that the nonlinearities 
	\begin{align*}
		\tilde{f} &\colon H^{s+2}_2(\mathbb{T}^n)\times H^{s2(1-\delta_Y)}_2(\mathbb{T}^n)\times\R^3\to H^{s+2(1-\delta_Y)}_2(\mathbb{T}^n),\\
		&\qquad\qquad (x,y,z_1,z_2,z_3)\mapsto \tfrac{1}{M}z_2+z_1^2 x[1+\eta^2(x^2-y^2)],\\
		\tilde{g}&\colon H^{s+2}_2(\mathbb{T}^n)\times H^{s+2+2(1-\delta_Y)}_2(\mathbb{T}^n)\times\R^3\to H^{s+2}_2(\mathbb{T}^n)\times \R^3\\
		&\qquad\qquad (x,y,z_1,z_2,z_3)\mapsto \big(z_3-y[1+\eta^2(x^2-y^2)],z_1,z_2,z_3\big),
	\end{align*}
	are well-defined and satisfy \eqref{Eq:Stommel:Nonlinearity_Condition4} locally. In order to obtain these properties globally, we use cutoff techniques. Let $R>0$ be arbitrary and choose $C^1$-functions 
	\begin{align*}
		\chi_1\colon H^{s+2(1-\delta_Y)}_2(\mathbb{T}^n)\to [0,1],\quad
		\chi_2\colon H^{s+2}_2(\mathbb{T}^n)\to [0,1],\quad
		\chi_3\colon H^{s+2+2(1-\delta_Y)}_2(\mathbb{T}^n)\to [0,1]
	\end{align*}
	which equal to $1$ on the ball $B(0,R)$ around $0$ with radius $R$ in their respective topologies and which equal to $0$ in the complement of $B(0,2R)$. For $\sigma>0$ we further choose $\psi_{\sigma}\in C^{\infty}(\R)$ taking values in $[0,1]$ such that
	\begin{align*}
		\psi(z)=1\text{ if }|z|\leq\sigma,\quad\psi(z)=0\text{ if }|z|\geq2\sigma,\quad\text{and}\quad |\psi'(z)|\leq \frac{2}{\sigma}\text{ for }z\in\R.
	\end{align*}
\end{itemize}
Now, the nonlinearities 
	\begin{align*}
		f &\colon H^{s+2}_2(\mathbb{T}^n)\times H^{s+2(1-\delta_Y)}_2(\mathbb{T}^n)\times\R^3\to H^{s+2(1-\delta_Y)}_2(\mathbb{T}^n),\\
		&\qquad\quad (x,y,z_1,z_2,z_3)\mapsto \tfrac{1}{M}z_2+\psi_\sigma(z_1)z_1^2 \chi_2(x)x[1+\eta^2((\chi_2(x)x)^2-(\chi_1(y)y)^2)],\\
		g&\colon H^{s+2}_2(\mathbb{T}^n)\times H^{s+2+2(1-\delta_Y)}_2(\mathbb{T}^n)\times\R^3\to H^{s+2}_2(\mathbb{T}^n)\times \R^3\\
		&\qquad\quad (x,y,z_1,z_2,z_3)\mapsto \big(z_3-\chi_3(y)y[1+\eta^2((\chi_2(x)x)^2-(\chi_3(y)y)^2)],z_1,z_2,z_3\big),
	\end{align*}
	satisfy \eqref{Eq:Stommel:Nonlinearity_Condition4} globally. Moreover, if we choose $\sigma>0$ small enough, then we have
	\begin{align*}
		L_f<\tfrac{2}{M}.
	\end{align*}
With these choices, we may rewrite \eqref{Eq:Stommel} as
\begin{align}
	\begin{aligned}\label{Eq:Stommel_Our_Setting}
	\epsilon \partial_t u^{\epsilon}&=A u^{\epsilon}+f(u^{\epsilon},v^{\epsilon}_1,\sqrt{\epsilon},M,\mu),\\
	\partial_t v^{\epsilon}&= B v^{\epsilon}+g(u^{\epsilon},v^{\epsilon}_1,\sqrt{\epsilon},M,\mu),\\
	u^{\epsilon}(0)&=u_0,\quad v_1^{\epsilon}(0)=v_0.
	\end{aligned}
	\end{align}
	If $\|u^{\epsilon}\|_{H^{s+2}_2(\mathbb{T}^n)},\|v^{\epsilon}_1\|_{H^{s+2+2(1-\delta_Y)}_2(\mathbb{T}^n)}\leq R$ and $\epsilon\leq\sigma^2$, then \eqref{Eq:Stommel} and \eqref{Eq:Stommel_Our_Setting} coincide. This is why we have to introduce $T(R)$ in the statements of Theorem~\ref{Thm:Stommel:Approximation_Slow_Flow} and Theorem~\ref{Thm:Stommel:Slow_Manifolds}. For the proof Theorem~\ref{Thm:Stommel:Approximation_Slow_Flow}  we just have to check whether the assumptions of Section~\ref{Section:General_Fast_Slow:Full_System} are satisfied:
	\begin{enumerate}[(i)]
		\item It is well-known that $X=H^{s}_2(\mathbb{T}^n)$ and $Y=H^{s+2(1-\delta_Y)}_2(\mathbb{T}^n)\times\R^3$ are Banach spaces.
		\item The Laplacian generates a bounded holomorphic $C_0$-semigroup $(\txte^{t\Delta})_{t\geq0}$ on any of the spaces $H^{s+\alpha}_2(\mathbb{T}^n)$, $\alpha\in\R$, which is given by
		\[
			\txte^{t\Delta}f(x)=\sum_{k\in\Z} \txte^{-|k|^2t} \hat{f}(k) \txte^{\txti kx},
		\]
		where $\hat{f}(k)$ denotes the $k$-th Fourier coefficient. Accordingly, $A$ generates an exponentially decaying holomorphic $C_0$-semigroup and $B$ generates a holomorphic $C_0$-semigroup.
		\item It follows from complex interpolation that the spaces $X_{\alpha}=H^{s+2\alpha}_2(\mathbb{T}^n)$ and $Y_{\alpha}=H^{s+2(1-\delta_Y)+2\alpha}_2(\mathbb{T}^n)\times\R^3$ are valid choices for our Banach scales.
		\item We used cutoff techniques in order to ensure that $f$ and $g$ satisfy the continuity assumptions of Section~\ref{Section:General_Fast_Slow:Full_System}.
		\item We introduced the dummy variable $\tilde{w}^{\epsilon}$ the ensure that $f(0,0)=0$ and $g(0,0)=0$.
		\item Theorem~\ref{Thm:Semigroup_in_Scales} ensures that there are constants $M_A,C_A$ and $C_B$  such that
		\begin{align*}
 	\|\txte^{tA}\|_{\mathcal{B}(X_1)}\leq &M_A \txte^{\omega_A t},\quad \|\txte^{tA}\|_{\mathcal{B}(X_{\gamma_X},X_1)}\leq C_At^{\gamma_X-1} \txte^{\omega_A t},\\
 	& \|\txte^{tA}\|_{\mathcal{B}(X_{\delta_X},X_1)}\leq C_At^{\delta_X-1} \txte^{\omega_A t}
 \end{align*}
 and
 \begin{align*}
  	\|\txte^{tB}\|_{\mathcal{B}(Y_1)}\leq M_B \txte^{\omega_B t},\quad \|\txte^{tB}\|_{\mathcal{B}(Y_{\delta_Y},Y_1)}\leq C_Bt^{\delta_Y-1} \txte^{\omega_B t}
 \end{align*}
 hold for all $t>0$. Since $(\txte^{t\Delta})_{t\geq0}$ is a bounded holomorphic semigroup on any of the spaces $H^{s+\alpha}_p(\mathbb{T}^n)$, $\alpha\in\R$, we may take $\omega_A$ to be an arbitrary number larger than $-1$.
 \item We chose $\sigma>0$ such that $L_f<\frac{2}{M}$. If we take $M>8C_A$ and $\omega_A$ close to $-1$, then we have
 \[
 	\omega_f=\omega_A+\frac{\sqrt{2}}{4}<0.
 \]
	\end{enumerate}
	Altogether, all the assumptions of Section~\ref{Section:General_Fast_Slow:Full_System} are satisfied and we obtain Theorem~\ref{Thm:Stommel:Approximation_Slow_Flow}\\
	Let us turn to Theorem~\ref{Thm:Stommel:Slow_Manifolds}. Our task now is to find the splitting
	\[
		Y=Y_F^{\zeta}\oplus Y_S^{\zeta}
	\]
	for any $\zeta>0$ small enough. For the Stommel model we may simply take the truncation to certain Fourier modes. If $-(|k_0|+1)^2<\zeta^{-1}\omega_A\leq-|k_0|^2$ for some $k_0\in\N$, then we take
	\begin{align*}
		\tilde{Y}_S^{\zeta}&:= \operatorname{span}\big\{ [x\mapsto \txte^{\txti kx}] : k\in\Z,\,|k|\leq |k_0|-1\big\},\\
		\tilde{Y}_F^{\zeta}&:= \operatorname{cl}_{H^{s+2(1-\delta_Y)}_p(\mathbb{T})}\big(\operatorname{span}\big\{ [x\mapsto \txte^{\txti kx}] : k\in\Z,\,|k|\geq |k_0|\big\}\big),
	\end{align*}
	where $\operatorname{cl}_{H^{s+2(1-\delta_Y)}_p(\mathbb{T})}A$ means that we take the closure of a set $A\subset H^{s+2(1-\delta_Y)}_p(\mathbb{T})$ in $H^{s+2(1-\delta_Y)}_p(\mathbb{T})$. 
 Now we choose
	\[
		Y_S^{\zeta}:=\tilde{Y}_S^{\zeta}\times\R^3,\quad Y_F^{\zeta}:=\tilde{Y}_F^{\zeta}\times\{0_{\R^3}\}¸.
	\]
	These definitions indeed yield a splitting
	\[
		Y=Y_F^{\zeta}\oplus Y_S^{\zeta}.
	\]
	Let us check the conditions of Section~\ref{Section:Bates_Difficult:Our_Approach}. 
	\begin{enumerate}[(i)]
		\item Since $Y_S^{\zeta}$ is finite-dimensional and since $Y_F^{\zeta}$ is defined as a closure, both spaces are closed. Moreover, in the Fourier image it is easy to see that the their projections commute with $B$.
		\item By our construction $\tilde{Y}_F^{\zeta}$ consists of all $f\in H^{s+2(1-\delta_Y)}_2(\mathbb{T})$ such that $\hat{f}(k)=0$ for all $k\in\Z$ such that $|k|\leq |k_0|-1$. Therefore, the $\tilde{Y}_F^{\zeta}\cap H^{s+2+2(1-\delta_Y)}_2(\mathbb{T})$ consists of all $f\in H^{s+2+2(1-\delta_Y)}_2(\mathbb{T})$ such that $\hat{f}(k)=0$ for all $k\in\Z$ such that $|k|\leq |k_0|-1$. This makes $Y_F^{\zeta}$ a closed subspace of $Y_1=H^{s+2+2(1-\delta_Y)}_2(\mathbb{T})\times\R^3$.
		\item Obviously, $Y_S^{\zeta}$ is a closed subspace of $Y_1$ and thus the same holds trivially for $Y_S^{\zeta}\cap Y_1$. In addition, we know that
		\[
		g\colon X_1\times Y_1\to Y_{\delta_Y}
		\]
		is Lipschitz continuous and Plancherel's theorem yields
		\[
			\| \operatorname{pr}_{Y_S^{\zeta}}\|_{\mathcal{B}(Y_{\delta_Y},Y_1)} \leq \zeta^{\delta_Y-1}.
		\]
		Hence, we obtain that
		\[
		\operatorname{pr}_{Y_S^{\zeta}}g\colon X_1\times Y_1\to Y_{1}
		\]
		is Lipschitz continuous with Lipschitz constant $L_g\zeta^{\delta_Y-1}$.
		\item $Y_S^{\zeta}$ is a finite-dimensional space. Therefore, the realization of $B$ in $Y_S^{\zeta}$ is bounded and thus generates a $C_0$-group $(\txte^{tB_{Y_S^{\zeta}}})_{t\in\R}$. It is obvious that is group coincides with $(\txte^{tB}\vert_{Y_S^{\zeta}})$ for $t\geq0$.
		\item We show that the realization of $B$ in $Y_F^{\zeta}$ has $0$ in its resolvent set by simply giving a formula for the inverse. It is given by
		\begin{align*}
			B_{Y_F^{\zeta}}^{-1}&\colon H^{s+2+2(1-\delta_Y)}_2(\mathbb{T})\times\{0_{\R^3}\}\to H^{s+2(1-\delta_Y)}_2(\mathbb{T})\times\{0_{\R^3}\},\\
			&\left(\sum_{k\in\Z,\atop|k|\geq|k_0|}\hat{f}(k)\txte^{\txti kx},0,0,0\right)\mapsto \left(\sum_{k\in\Z,\atop|k|\geq|k_0|}\frac{\hat{f}(k)\txte^{\txti kx}}{|k|^2},0,0,0\right).
		\end{align*}
		This is well-defined if $\zeta$ is small as $k=0$ does not appear in the sum.
		\item We have already observed that $(\txte^{tB})_{t\geq0}$ is given by
		\[
			\txte^{tB} f = \bigg[x\mapsto \sum_{k\in\Z} \txte^{-|k|^2t}\hat{f}(k)\txte^{\txti kx}\bigg].
		\]
		Thus, Plancherel's theorem shows that for $y_S\in Y_S^{\zeta}$ and $t\geq0$ it holds that
		\[
			\| \txte^{-tB}y_S \|_{H^{s+2(1-\delta_Y)}_2(\mathbb{T})}\leq \txte^{(|k_0|-1)^2t},
		\]
		so that we may take $$N_S^{\zeta}:=-\zeta^{-1}\omega_A-(|k_0|-1)^2.$$ Since $-(|k_0|+1)^2<\zeta^{-1}\omega_A\leq-|k_0|^2$ it holds that $N_S^{\zeta}>0$. Similarly, we can take $$N_F^{\zeta}=-\zeta^{-1}\omega_A-|k_0|^2$$
		so that $N_S^{\zeta}-N_F^{\zeta}=2|k_0|-1\geq 2\sqrt{-\zeta^{-1}\omega_A}-3$. Therefore, we have $N_S^{\zeta}-N_F^{\zeta}>\zeta^{-1/2}$ if $\zeta$ is small and if $\omega_A$ is close to $-1$.
		\item If we take $\zeta>0$ small enough and $\epsilon<c\frac{(L_FC_A\Gamma(\gamma_X))^{1/\gamma_X}+\omega_A}{\omega_A}\zeta$ for some constant $c\in(0,1)$, then \eqref{Eq:Gap_Condition} is satisfied. Note that we need $\delta_Y>\tfrac{1}{2}$ for this to hold true.
	\end{enumerate}
	Altogether, all the assumptions we need to apply our theory are satisfied. The application of our abstract results to the diffusive Stommel model to obtain Theorem~\ref{Thm:Stommel:Slow_Manifolds} is straightforward. We should point out though that for the proof of Theorem~\ref{Thm:Stommel:Slow_Manifolds}~\eqref{Thm:Stommel:Slow_Manifolds:Flow_Convergence} one formally has different initial conditions for $(u^{\epsilon},w^{\epsilon})$ and $(u^{\zeta}_0,w^{\zeta}_0)$ due to our dummy variables: For \eqref{Eq:Stommel} we have $z_2=\sqrt{\epsilon}$ and for \eqref{Eq:Stommel_truncated} we have $z_2=0$. However, the well-posedness \eqref{Eq:Stommel} ensures that the difference of the solutions of \eqref{Eq:Stommel} with $z_2=\sqrt{\epsilon}$ and $z_2=0$ are of the order $O(\sqrt{\epsilon})$ on bounded time intervals. Thus, for the derivation of Theorem~\ref{Thm:Stommel:Slow_Manifolds}~\eqref{Thm:Stommel:Slow_Manifolds:Flow_Convergence} we can just use Corollary~\ref {Cor:Convergence_of_Flows} together with an application of the triangle inequality.

\subsection{The Doubly-Diffusive FitzHugh-Nagumo Equation}

The techniques we used for the Stommel model can also be applied to the doubly-diffusive FitzHugh-Nagumo equation, which has recently been of interest in pattern formation~\cite{CornwellJones}. It is a modification of the classical FitzHugh-Nagumo equation and given by
\begin{align}\label{Eq:FitzHugh-Nagumo}
\begin{aligned}
	\epsilon \partial_t u^{\epsilon}&=\Delta u^{\epsilon}+u^{\epsilon}(1-u^{\epsilon})(u^{\epsilon}-a)-w^{\epsilon},\\
	\partial_tw^{\epsilon}&=\Delta w^{\epsilon}+u^{\epsilon}-\gamma w^{\epsilon},\\
	u^{\epsilon}(0)&=u_0,\quad w^{\epsilon}(0)=v_0
\end{aligned}
\end{align}
where $\gamma>0$ and $a\in(0,\tfrac{1}{2})$. Of course, it is well-known from many works (see~\cite{KuehnBook} and references therein) that at the two fold points of nonlinearity, there is loss of normal hyperbolicity even without the Laplacian terms. Hence, we just illustrate our methods locally at a point on an attracting branch of the critical manifold. We simply select this point as the origin but other points could be treated similarly upon translation of the coordinates locally. Furthermore, compared to the Stommel model we have the additional difficulty that the nonlinearity in the fast variable does not get small as $\epsilon\to0$. However, we have the advantage that we do not have to introduce dummy variables and that all terms are actually linear in the slow variable. The latter property will help us to derive better convergence results, since we can avoid certain cutoffs that would cause problems with different topologies. This way, we obtain:
\begin{theorem}\label{Thm:FHN:Approximation_Slow_Flow}
	Let $\mathbb{E}\in\{\mathbb{T},\mathbb{R}\}$, i.e. let $\mathbb{E}$ either be the torus or the real line. We write $(u^{\epsilon},w^{\epsilon})$ for the strict solution of \eqref{Eq:FitzHugh-Nagumo} with $\epsilon>0$ and $(u^{0},w^{0})$ for corresponding slow flow. Then there are a neighborhood $U\subset H^2_2(\mathbb{E}^n)$ of $0$ which only depends on $a$ and constants $\epsilon_0>0$ and $C,c>0$ such that that for all $\epsilon\in(0,\epsilon_0]$, $u_0\in U$ and $v_0\in H^{2}_2(\mathbb{E}^n)$ it holds that
	\[
		\sup_{0\leq t \leq T(R,U)}\big(\|u^{\epsilon}(t)-u^0(t)\|_{H^{2}_2(\mathbb{E}^n)}+\|w^{\epsilon}(t)-v^0(t)\|_{H^{2}_2(\mathbb{E}^n)}\big)\leq C(\epsilon+\txte^{-c\epsilon^{-1}t}),
	\]
	where $T(R,U)$ is defined by
	\begin{align*}
		T(R,U):=\inf\big\{t\in[0,T]: u^0\notin U \text{ or } u^{\epsilon}\notin U\big\}.
	\end{align*}
\end{theorem}

\begin{theorem}\label{Thm:FHN:Slow_Manifolds}
	There are a neighborhood $U\subset H^2_2(\mathbb{T})$ of $0$ which only depends on $a$, a constant $\zeta_0>0$ and a family of finite-dimensional manifolds $S_{\epsilon,\zeta}\subset H^{2}_2(\mathbb{T})\times H^{2}_2(\mathbb{T}) $ with $0<\zeta\leq\zeta_0$ and $0<\epsilon\leq C\frac{(L_FC_A\Gamma(\gamma_X))^{1/\gamma_X}+\omega_A}{\omega_A}\zeta$ for some $C\in(0,1)$ such that the following assertions hold:
	\begin{enumerate}[(a)]
		\item For each $\zeta\in(0,\zeta_0]$ there is a splitting
			\[
				L_2(\mathbb{T})=Y_F^{\zeta}\oplus Y_S^{\zeta},
			\]
			where $Y_S^{\zeta}$ is the projection of $L_2(\mathbb{T})$ to the $k$-th Fourier modes with $|k|$ being smaller than a certain number $k(\zeta)$ depending on $\zeta$. $Y_F^{\zeta}$ is to projection to the remaining Fourier modes.
		\item The manifolds $S_{\epsilon,\zeta}$ are given as the graph of a differentiable mapping
			\[
				h^{\epsilon,\zeta}\colon (Y_S^{\zeta},\|\cdot\|_{H^{s+2}_2(\mathbb{T})})\to H^{2}_2(\mathbb{T})\times (Y_F^{\zeta}\cap H_2^2(\mathbb{T}),\|\cdot\|_{H^{2}_2(\mathbb{T})}).
			\]
		\item The intersection of $S_{\epsilon,\zeta}$ with $U\times Y$ is a slow manifold which is locally invariant under the semiflow generated by \eqref{Eq:FitzHugh-Nagumo}, i.e. the semiflow can only leave $S_{\epsilon,\zeta}\cap U\times Y$ through its boundary.
		\item Let 
			\[
				S_{0,\zeta,U}:=\{ (u,w)\in S_0: w\in Y_S^{\zeta} \}\cap U\times Y
			\]
			be the intersection of $U\times Y$ with the submanifold of the critical manifold which consists of all points whose slow components are elements of $Y_S^{\zeta}$. Then constant $C>0$ such that
			\[
				\operatorname{dist}(S_{\epsilon,\zeta},S_{0,\zeta})\leq C(\epsilon+\zeta^{1/2})\leq C\zeta^{1/2}.
			\]
		\item  Suppose that $u_0\in U$ and let $(u^0_{\zeta},w^0_{\zeta})$ be the solution of the truncated slow subsystem of \eqref{Eq:FitzHugh-Nagumo} given by 
		\begin{align}
	\begin{aligned}\label{Eq:FitzHugh-Nagumo_truncated}
	0&=\Delta u^0_{\zeta}-u^0_{\zeta}(1-u^0_{\zeta})(u^0_{\zeta}-a)-v^0_{\zeta},\\
	\partial_t w^0_{\zeta}&= \operatorname{pr}_{Y_S^{\zeta}}\left[\Delta w^0_{\zeta}+u^0_{\zeta}-\gamma v^0_{\zeta}\right],\\
	u^{\epsilon}(0)&=h^0(\operatorname{pr}_{Y_S^{\zeta}}v_0),\quad w^{\epsilon}(0)=\operatorname{pr}_{Y_S^{\zeta}}v_0.
	\end{aligned}
\end{align}
Assume that $(u_0,v_0)\in S_{\epsilon,\zeta}\cap U\times Y $. Then for each $T>0$ there is a constant $C>0$ such that
\[
		\sup_{0\leq t \leq T(U)}\big(\|u^{\epsilon}(t)-u^0_{\zeta}(t)\|_{H^{2}_2(\mathbb{T})}+\|w^{0}_{\zeta}(t)-v^0(t)\|_{H^{2}_2(\mathbb{T})}\big)\leq C\zeta^{1/2},
	\]
	where $T(R,U)$ is defined by
	\begin{align*}
		T(U):=\inf\big\{t\in[0,T]: u^0_{\zeta}\notin U \text{ or } u^{\epsilon}\notin U\big\}.
	\end{align*}
	\end{enumerate}
\end{theorem}

\begin{remark}
	One might wonder why we have to introduce the neighborhood $U$ in Theorem~\ref{Thm:FHN:Approximation_Slow_Flow} and Theorem~\ref{Thm:FHN:Slow_Manifolds}. The reason is that we have only treated the attracting case in our general theory. In order to ensure that we stay in this attracting case, we use cutoff techniques to modify the nonlinearity in the fast variable where it would be positive. However, this means that our results are only related to the system \eqref{Eq:FitzHugh-Nagumo} as long as the fast variable stays in the region where we did not modify the nonlinearity.
\end{remark}

Let us give a sketch on how these results can be obtained. Again, we only treat the case $\mathbb{E}=\mathbb{T}$. First, we rescale the slow variable and define $v^{\epsilon}=\frac{2}{a}w^{\epsilon}$ so that \eqref{Eq:FitzHugh-Nagumo} turns into
\begin{align}\label{Eq:FitzHugh-Nagumo_rescaled}
\begin{aligned}
	\epsilon \partial_t u^{\epsilon}&=\Delta u^{\epsilon}+u^{\epsilon}(1-u^{\epsilon})(u^{\epsilon}-a)-\tfrac{a}{2} v^{\epsilon},\\
	\partial_tv^{\epsilon}&=\Delta v^{\epsilon}+\tfrac{2}{a}u^{\epsilon}-\gamma v^{\epsilon},\\
	u^{\epsilon}(0)&=u_0,\quad v^{\epsilon}(0)=\tfrac{2}{a}v_0
\end{aligned}
\end{align}
Now we make the following choices:
\begin{itemize}
	\item As underlying spaces we choose $X=L_2(\mathbb{T}^n)$ and $Y=L_2(\mathbb{T}^n)$.
	\item The linear operator in the fast variable is given by
		\[
			A:L_2(\mathbb{T}^n)\supset H^{2}_2(\mathbb{T}^n)\to L_2(\mathbb{T}^n),\,u\mapsto \Delta u-au.
		\]
The linear operator in the slow variable is given by
		\[
			B:L_2(\mathbb{T}^n)\supset H^{2}_2(\mathbb{T}^n)\to L_2(\mathbb{T}^n),\,u\mapsto \Delta u-\gamma u.
		\]	
	\item The Banach scales are given by $X_{\alpha}=H^{2\alpha}_2(\mathbb{T}^n)$ and $Y_{\alpha}=H^{2\alpha}_2(\mathbb{T}^n)$.
	\item We choose $\gamma_X=\delta_X=\delta_Y=1$. This is the main difference to the Stommel model and will lead to better convergence rates. With these parameters, it suffices to choose a differentiable mapping $f\colon X_1\times Y\to X$ which is also differentiable as a mapping from $X_1\times Y_1$ to $X_1$ such that
	\begin{align*}
		\| Df(x,y) \|_{\mathcal{B}(X_1\times Y,X)}\leq L_f<a,\\
		\| Df(x,y) \|_{\mathcal{B}(X_1\times Y_1,X_1)}\leq L_f<a.
	\end{align*}
	Moreover, for the nonlinearity in the slow variable we may choose a continuous mapping $g\colon X\times Y\to Y$ which is differentiable as a mapping $g\colon X_1\times Y_1\to Y_1$ with bounded derivative. With our choices of spaces this translates into 
 	\begin{align*}
		f&\colon H^{2}_2(\mathbb{T}^n)\times L_2(\mathbb{T}^n)\to L_2(\mathbb{T}^n),\\
		g&\colon L_2(\mathbb{T}^n)\times L_2(\mathbb{T}^n)\to L_2(\mathbb{T}^n)
	\end{align*}
	and 
		\begin{align*}
		\|\txtD f(x,y)\|_{\mathcal{B}(H^{2}_2(\mathbb{T}^n)\times L_2(\mathbb{T}^n),L_2(\mathbb{T}^n))}&\leq L_f<a,\\
		\|\txtD f(x,y)\|_{\mathcal{B}(H^{2}_2(\mathbb{T}^n)\times H^{2}_2(\mathbb{T}^n),H^{2}_2(\mathbb{T}^n))}&\leq L_f<a,\\
		\|\txtD g(x,y)\|_{\mathcal{B}(H^{2}_2(\mathbb{T}^n)\times H^{2}_2(\mathbb{T}^n),H^{2}_2(\mathbb{T}^n)))}&\leq L_g.
	\end{align*}
 For the definition of $f$, we choose a small number $1>\sigma>0$ and a $C^{1}$-function $\chi\colon H^2_2(\mathbb{T}^n)\to [0,1]$ such that $\chi(u)=1$ if $\|u\|_{H^2_2(\mathbb{T}^n)}\leq\sigma^2$, $\chi(u)=0$ if $\|u\|_{H^2_2(\mathbb{T}^n)}\geq2\sigma$ and $\|\txtD\chi\|_{\mathcal{B}(H^2_2(\mathbb{T}^n);\R)}\leq \sigma$. Then we define
 \begin{align*}
 	f&\colon H^{2}_2(\mathbb{T}^n)\times L_2(\mathbb{T}^n)\to L_2(\mathbb{T}^n),\,(u,v)\mapsto -(\chi(u)u)^{3}+(1+a)(\chi(u)u)^{2}-\frac{a}{2}v,\\
 	g&\colon L_2(\mathbb{T}^n)\times L_2(\mathbb{T}^n)\to L_2(\mathbb{T}^n),\,(u,v)\mapsto\tfrac{2}{a}u.
 \end{align*}
 If $\sigma$ is small enough, then it will hold that $L_f<a$.
\end{itemize}
With these choices, the equation
\begin{align*}
\epsilon \partial_t u^\epsilon &= A u^\epsilon + f(u^\epsilon,v^\epsilon),\\
\partial_t v^\epsilon &= B v^\epsilon + g(u^\epsilon,v^\epsilon)
\end{align*}
is equivalent to \eqref{Eq:FitzHugh-Nagumo_rescaled} as long as $\|u^{\epsilon}\|_{H^2_2(\mathbb{T}^n)}\leq\sigma^2$. Concerning the splitting $Y=Y_F^{\zeta}\oplus Y_S^{\zeta}$ we make analogous choices as for the Stommel model. Now, as for the Stommel model one can verify that our theory can be applied.

\subsection{The Maxwell-Bloch Equations}

We consider the Maxwell-Bloch equations in the slow time scale
\begin{align}
\begin{aligned}\label{Eq:Maxwell_Bloch}
	\epsilon \partial_t u_1^{\epsilon}&=\mu w^{\epsilon}u_2^{\epsilon}-(1+i\delta)u_1^{\epsilon},\\
	\epsilon \partial_t u_2^{\epsilon}&=\gamma_{\parallel}(\lambda +1-u_2^{\epsilon})-\frac{\mu}{2}\big(\overline{w^{\epsilon}}u_1^{\epsilon}+w^{\epsilon}\overline{u_1^{\epsilon}}\big),\\
	\partial_tw^{\epsilon}&=-\partial_xw^{\epsilon}+\kappa\left(\tfrac{1}{\mu}u_1^{\epsilon}-w^{\epsilon}\right),\\
		u^{\epsilon}_1(0)&=u_{0,1},\quad u^{\epsilon}_2(0)=u_{0,2},\quad w^{\epsilon}(0)=v_0,
	\end{aligned}
\end{align}
on the one-dimensional torus $\mathbb{T}$. Here, $\gamma_{\parallel},\kappa,\delta,\lambda>0$ are certain parameters and $\mu=\sqrt{\lambda\gamma_{\parallel}}$. The existence of slow manifolds for this system which are given as graphs over a certain subset of the slow variable space has been shown in \cite{MenonHaller} by a direct approach. We want to illustrate that these equations are a special case accessible through our more general methods. 

\begin{theorem}\label{Thm:Maxwell-Bloch:Approximation_Slow_Flow}
	Let $R>0$ be large enough, $T>0$ and $w_0\in C^1(\mathbb{T},\C)$ be fixed. Let further $(u^{\epsilon},w^{\epsilon})$ be the strict solution of \eqref{Eq:Maxwell_Bloch} with $\epsilon>0$ and let $(u^0,w^0)$ be the corresponding slow flow. Then there are a neighborhood $U\subset C^1(\mathbb{T},\C)$ of $w_0$	and constants $\epsilon_0,C,c>0$ such that for all  $\epsilon\in(0,\epsilon_0]$, $u_0\in C^1(\mathbb{T};\C)\times C^1(\mathbb{T};\R) $ with $\|u_{0,1}\|_{C^1(\mathbb{T};\C)}+\|u_{0,2}\|_{C^1(\mathbb{T};\R)}\leq R$ and $v_0\in U$ it holds that
	\[
		\sup_{0\leq t \leq T(R,U)}\big(\|u^{\epsilon}(t)-u^0(t)\|_{C^1(\mathbb{T};\C)\times C^1(\mathbb{T};\R)}+\|w^{\epsilon}(t)-w^0(t)\|_{C^1(\mathbb{T};\C)}\big)\leq C(\epsilon+e^{-c\epsilon^{-1}t}),
	\]
	where $T(R,U)$ is defined by
	\begin{align}\begin{aligned}\label{Eq:Maxwell-Bloch:Admissible_Intervall}
		T(R,U):=\inf\big\{t\in[0,T]:& \max\{\|u^0(t)\|_{C^1(\mathbb{T};\C)\times C^1(\mathbb{T};\R)},\|u^{\epsilon}(t)\|_{C^1(\mathbb{T};\C)\times C^1(\mathbb{T};\R)}\}> R\\
		&\text{ or } w^0(t)\notin U \text{ or } w^{\epsilon}(t)\notin U \big\}.
		\end{aligned}
	\end{align}
\end{theorem}

\begin{theorem}\label{Thm:Maxwell-Bloch:Slow_Manifolds}
	Let $R>0$ be large enough and let $w_0\in C^1(\mathbb{T},\C)$ be fixed. Then there are $\epsilon_0>0$, a neighborhood $U\subset C^1(\mathbb{T},\C)$ of $w_0$ and a family of infinite-dimensional slow manifolds $S_{\epsilon}\subset C^1(\mathbb{T},\C)\times C^1(\mathbb{T},\R)\times C^1(\mathbb{T},\C) $ with $0<\epsilon\leq\epsilon_0$ such that the following assertions hold:
	\begin{enumerate}[(a)]
		\item The slow manifold $S_{\epsilon}$ is given as the graph of a differentiable mapping
			\[
				h^{\epsilon}\colon (U,\|\cdot\|_{C^1(\mathbb{T},\C)})\to C^1(\mathbb{T},\C)\times C^1(\mathbb{T},\R).
			\]
		\item $S_{\epsilon}$ is locally invariant under the semiflow generated by \eqref{Eq:Maxwell_Bloch}, i.e. the semiflow can only leave $S_{\epsilon}$ through its boundary.
		\item Let 
			\[
				S_{0,U}:=\{ (u,w)\in S_0: w\in U \}
			\]
			be the submanifold of the critical manifold which consists of all points whose slow components are elements of $U$. Then there is a constant depending on $R$ such that
			\[
				\operatorname{dist}(S_{\epsilon},S_{0,U})\leq C\epsilon.
			\]
		\item \label{Thm:Maxwell-Bloch:Slow_Manifolds:Flow_Convergence} Suppose that $\|u_0\|_{C^1(\mathbb{T};\C)\times C^1(\mathbb{T};\R)}\leq R$, $v_0\in U$. Assume that $(u_0,v_0)\in S_{\epsilon}$. Then for each $T>0$ there is a constant $C>0$ such that
\[
		\sup_{0\leq t \leq T(R,U)}\big(\|u^{\epsilon}(t)-u^0(t)\|_{C^1(\mathbb{T};\C)\times C^1(\mathbb{T};\R)}+\|w^{\epsilon}(t)-w^0(t)\|_{C^1(\mathbb{T};\C)}\big)\leq C\epsilon,
	\]
	where $T(R,U)$ is again defined by \eqref{Eq:Maxwell-Bloch:Admissible_Intervall}.
	\end{enumerate}
\end{theorem}

First, we rescale \eqref{Eq:Maxwell_Bloch} so that the constants in front of the nonlinearities in the fast variable can be chosen small. We define $\tilde{v}^{\epsilon}:=\sigma^{-1}w^{\epsilon}$ for some $\sigma>0$ and obtain
\begin{align}
\begin{aligned}\label{Eq:Maxwell_Bloch_Rescaled}
	\epsilon \partial_t u_1^{\epsilon}&=\sigma\mu \tilde{v}^{\epsilon}u_2^{\epsilon}-(1+i\delta)u_1^{\epsilon},\\
	\epsilon \partial_t u_2^{\epsilon}&=-\gamma_{\parallel}u_2^{\epsilon}+\gamma_{\parallel}(1+\lambda)-\tfrac{\sigma\mu}{2}\big(\overline{\tilde{v}^{\epsilon}}u_1^{\epsilon}+\tilde{v}^{\epsilon}\overline{u_1^{\epsilon}}\big),\\
	\partial_t\tilde{v}^{\epsilon}&=-\partial_x\tilde{v}^{\epsilon}+\kappa\left(\tfrac{1}{\sigma\mu}u_1^{\epsilon}-\tilde{v}^{\epsilon}\right),\\
		u^{\epsilon}_1(0)&=u_{0,1},\quad u^{\epsilon}_2(0)=u_{0,2},\quad \tilde{v}^{\epsilon}(0)=\tfrac{v_0}{\sigma}.
	\end{aligned}
\end{align}
Straightforward calculation shows that the critical manifold to this rescaled equation is given as the graph of
	\begin{align}\label{Eq:Maxwell-Bloch_Critical_Manifold}
	h^0_{\sigma}\left(\tfrac{v_0}{\sigma}\right)=\begin{pmatrix} \mu(1-i\delta)\frac{(\lambda+1)\sigma v_0}{1+\delta^2+\sigma^2\lambda|v_0|^2}\\
	\frac{(1+\delta^2)(\lambda+1)}{1+\delta^2+\sigma^2\lambda|v_0|^2} \end{pmatrix}.
\end{align}
In particular, $h^0_{\sigma}$ will be bounded in the spaces we choose later with a bound that can be chosen independently of $\sigma$. This fact will be useful for the cutoff procedure of the nonlinearities.\\
As for the Stommel model, we introduce the dummy variable $\tilde{w}^{\epsilon}$ to ensure that the nonlinearities vanish at $0$. This way, we may rewrite \eqref{Eq:Maxwell_Bloch_Rescaled} as
\begin{align}
\begin{aligned}\label{Eq:Maxwell_Bloch_Dummy}
	\epsilon \partial_t u_1^{\epsilon}&=\sigma\mu (\tilde{v}^{\epsilon}-\tfrac{v_0}{\sigma})u_2^{\epsilon}-(1+i\delta)u_1^{\epsilon}+\mu v_0 u_2^{\epsilon},\\
	\epsilon \partial_t u_2^{\epsilon}&=-\tfrac{\mu}{2}\big(\overline{v_0}u_1^{\epsilon}+v_0\overline{u_1^{\epsilon}}\big)-\gamma_{\parallel}u_2^{\epsilon}+\sigma\tilde{w}^{\epsilon}-\tfrac{\sigma\mu}{2}\big(\overline{(\tilde{v}^{\epsilon}-\tfrac{v_0}{\sigma})}u_1^{\epsilon}+(\tilde{v}^{\epsilon}-\tfrac{v_0}{\sigma})\overline{u_1^{\epsilon}}\big),\\
	\partial_t\tilde{v}^{\epsilon}&=-\partial_x\tilde{v}^{\epsilon}+\kappa\left(\tfrac{1}{\sigma\mu}u_1^{\epsilon}-\tilde{v}^{\epsilon}\right),\\
	\partial_t\tilde{w}^{\epsilon}&=0,\\
		u^{\epsilon}_1(0)&=u_{0,1},\quad u^{\epsilon}_2(0)=u_{0,2},\quad \tilde{v}^{\epsilon}(0)=\tfrac{v_0}{\sigma},\quad\tilde{w}^{\epsilon}(0)=\tfrac{(\lambda +1)\gamma_{\parallel}}{\sigma}
	\end{aligned}
\end{align}
Now we make the following choices:
 \begin{itemize}
		\item As base spaces we take 
			\[
				X:=C^1(\mathbb{T};\C)\times C^1(\mathbb{T};\R)\quad\text{and}\quad Y:=C(\mathbb{T};\C)\times\C.
			\]
			Here, we identify $\C=\R\times\R$ and treat it as a real vector space. This way complex conjugation is a differentiable mapping.
		\item The fast variable is given by $u^{\epsilon}:=(u^{\epsilon}_1,u^{\epsilon}_2)$ and the slow variable is given by $v^{\epsilon}:=(\tilde{v}^{\epsilon},\tilde{w}^{\epsilon})$.
		\item The linear operator $A$ of the fast variable is even a bounded operator:
			\[
				A\colon X\to X, \begin{pmatrix}\Re(u_1)\\\Im(u_1) \\u_2\end{pmatrix}\mapsto \begin{pmatrix} -\Re(u_1) + \delta\Im(u_1)+\mu \Re(v_0) u_2\\
				  -\delta \Re(u_1) - \Im(u_1) + \mu \Im(v_0) u_2\\ -\mu \Re(v_0)\Re(u_1)-\mu \Im(v_0)\Im(u_1)-\gamma_{\parallel}\end{pmatrix},
			\]
			i.e. it is given by the multiplication with matrix
			\begin{align*}
				\begin{pmatrix} -1 & \delta & \mu\Re(v_0) \\ -\delta & -1 & \mu\Im(v_0) \\ -\mu\Re(v_0) & \mu\Im(v_0) &-\gamma_{\parallel}\end{pmatrix}
			\end{align*}
			The eigenvalues $\lambda_1,\lambda_2,\lambda_3$ of this matrix have a negative real part. Let
			\[
				K:=|\max\{\Re(\lambda_1),\Re(\lambda_2),\Re(\lambda_3)\}|
			\]
			The linear operator $B$ of the slow variable is given by
			\[
				B\colon Y\supset D(B)\to Y,(v_1,v_2)\mapsto (-\partial_xv_1-\kappa v_1,0),
			\]
			where the domain is given by
			\[
				D(B)=C^1(\mathbb{T};\C)\times \C.
			\]
		\item We choose the parameters $\gamma_X=\delta_Y=1$ and $\delta_X=0$. Thus, we only need the Banach scales for $\alpha\in\{0,1\}$. Since $A$ is a bounded operator, the Banach scale in the fast variable is just given by $X=X_1$. For the fast variable we have $Y_1=C^1(\mathbb{T};\C)\times \C$ endowed with the norm $$\|(v_1,v_2)\|_{\mathbb{Y}_1}=\|v_1\|_{C^1(\mathbb{T};\C)}+|v_2|.$$
		\item The nonlinearities $\tilde{f}$, $\tilde{g}$ are given by
		\begin{align*}
			\tilde{f}&\colon X\times Y_1\to X,\;\begin{pmatrix}(x_1,x_2)^T\\ (y_1,y_2)^T \end{pmatrix}\mapsto \begin{pmatrix} \sigma\mu (y_1-\tfrac{v_0}{\sigma})x_2 \\ \sigma y_2-\tfrac{\sigma\mu}{2}((\overline{y_1-\tfrac{v_0}{\sigma}})x_1-(y_1-\tfrac{v_0}{\sigma})\overline{x_1})\end{pmatrix}
		\end{align*}
		and
				\begin{align*}
			g&\colon X\times Y\to Y,\;\begin{pmatrix}(x_1,x_2)^T\\ (y_1,y_2)^T \end{pmatrix}\mapsto \begin{pmatrix} \tfrac{\kappa}{\sigma\mu}x_1 \\ 0 \end{pmatrix}.
		\end{align*}
		In order to make $\tilde{f}$ globally Lipschitz continuous, we use cutoff functions again. Suppose that the critical manifold is bounded by
		\[
			M:=\sup_{v\in C^1(\mathbb{T};\C), 0<\sigma<1} \left\| h^0_\sigma(v)\right\|_{C^1(\mathbb{T};\C)\times C^1(\mathbb{T};\R)}.
		\]
		Let further $R\geq 2M$ and $\chi_1\colon X\to [0,1]$ be a $C^1$-function (in the real sense) such that $\chi_1(u)=1$ for $\|u\|_{X}\leq 2R$, $\chi_1(u)=0$ for $\|u\|_{X}\geq 2R+2$. Moreover, let $\tilde{K}>0$ large enough and $\chi_2\colon Y_1\to [0,1]$ be a $C^1$-function (in the real sense) such that $\chi_2(v)=1$ for $\|v\|_{Y_1}\leq \frac{K}{2\tilde{K}\mu\sigma}$, $\chi_2(v)=0$ for $\|v\|_{Y_1}\geq \frac{K}{\tilde{K}\mu\sigma}$ and $\|\txtD\chi_2\|_{\mathcal{B}(Y_1,\R)}\leq \tfrac{3\tilde{K}\mu\sigma}{K}$. Now we define 
				\begin{align*}
			f&\colon X\times Y\to X,\;\\
			&\quad\begin{pmatrix}(x_1,x_2)^T\\ (y_1,y_2)^T \end{pmatrix}\mapsto \begin{pmatrix} \sigma\mu (y_1-\tfrac{v_0}{\sigma})x_2\chi_1(x_2)\chi_2(y_1-\tfrac{v_0}{\sigma}) \\ \sigma y_2-\tfrac{\sigma\mu}{2}((\overline{y_1-\tfrac{v_0}{\sigma}})x_1-(y_1-\tfrac{v_0}{\sigma})\overline{x_1})\chi_1(x_1)\chi_2(y_1-\tfrac{v_0}{\sigma})\end{pmatrix}
		\end{align*}
	\end{itemize}
With these choices it holds that \eqref{Eq:Maxwell_Bloch} is given by
\begin{align}
	\begin{aligned}
		\epsilon\partial_tu^{\epsilon}&=Au^{\epsilon}+f(u^{\epsilon},v_1^{\epsilon},\tfrac{\gamma_{\parallel}(1+\lambda)}{\sigma}),\\
		\partial_tv^{\epsilon}&=Bv^{\epsilon}g(u^{\epsilon},v^{\epsilon}),\\
		u^{\epsilon}_1(0)&=u_{0,1},\quad u^{\epsilon}_2(0)=u_{0,2},\quad \tilde{v}^{\epsilon}(0)=\tfrac{v_0}{\sigma},
	\end{aligned}
\end{align}
as long as $\| u^{\epsilon}_1 \|_{C(\mathbb{T};\C)}\leq R$, $\| u^{\epsilon}_2 \|_{C(\mathbb{T};\R)}\leq R$ and $\|\sigma v_1^{\epsilon}-v_0\|_{C(\mathbb{T};\C)}\leq\tfrac{K}{10\tilde{K}\mu}$.\\
Let us now check the conditions of Section~\ref{Section:General_Fast_Slow:Full_System} for this example.
\begin{enumerate}[(i)]
	\item It is well-known that $X=C^1(\mathbb{T};\C)\times C^1(\mathbb{T};\R)$ and $Y=C(\mathbb{T};\C)\times\C$ are Banach spaces.
	\item Since all eigenvalues of $A$ have a negative real part and since $A$ is bounded, it follows that it generates an exponentially stable analytic semigroup. Moreover, it is well-known and straightforward to verify that 
	\[
		\partial_x\colon C(\mathbb{T};\C)\supset C^1(\mathbb{T};\C)\to C(\mathbb{T};\C),\,v\mapsto v
	\]
	generates the translation group $(T(t))_{t\in\R}$ given
	\[
		T(t) v(x)=v(t+x).
	\]
	Therefore, also $B$ generates a $C_0$ group which even is exponentially decaying.
	\item Since $\gamma_X=\delta_Y=1$ and $\delta_X=0$, we only need the spaces $X,Y, X_1,Y_1$ which we already defined. If we wanted, we could complete the scales by adding H\"older spaces, but this is not necessary for our considerations.
	\item The differentiability of $f\colon X\times Y_1\to X$ and $g\colon X\times Y\to Y$ in the real sense is obvious. It is also clear that $g\colon X_1\times Y_1\to Y_1$ is Lipschitz continuous. $f\colon X\times Y_1\to X$ is also globally Lipschitz continuous due to the cutoff. We need the Lipschitz constant of $f$ to be smaller than the decay rate of $\txte^{tA}$, i.e. smaller than $K$. But if $\sigma\to 0$ and $\tilde{K}\to\infty$, then we have that
	\[
		\| \txtD f(x,y) \|_{\mathcal{B}(X\times Y_1,X)} \to0
	\]
	This shows that both Lipschitz conditions on $f$ hold true with small Lipschitz constant $L_f$.
	\item We introduced the dummy variable $\tilde{w}^{\epsilon}$ so that $f(0,0)=0$ and $g(0,0)=0$.
	\item Let $\omega_A\in(-K,0)$ be close to $-K$. Since we have we chose $\gamma_X=\delta_Y=1$ and $\delta_X=0$, the estimates
	 \begin{align*}
 	\|\txte^{tA}\|_{\mathcal{B}(X_1)}\leq &M_A \txte^{\omega_A t},\quad \|\txte^{tA}\|_{\mathcal{B}(X_{\gamma_X},X_1)}\leq C_At^{\gamma_X-1} \txte^{\omega_A t},\\
 	& \|\txte^{tA}\|_{\mathcal{B}(X_{\delta_X},X_1)}\leq C_At^{\delta_X-1} \txte^{\omega_A t}
 \end{align*}
 and
 \begin{align*}
  	\|\txte^{tB}\|_{\mathcal{B}(Y_1)}\leq M_B \txte^{\omega_B t},\quad \|\txte^{tB}\|_{\mathcal{B}(Y_{\delta_Y},Y_1)}\leq C_Bt^{\delta_Y-1} \txte^{\omega_B t}
 \end{align*}
 hold trivially.
	\item Since we can make $L_f$ arbitrarily small by choosing $\sigma$ small and $\tilde{K}$ large enough, we immediately obtain that $\omega_f=\omega_A+C_AL_f<0$.
\end{enumerate}
 Now, the proof of Theorem~\ref{Thm:Maxwell-Bloch:Approximation_Slow_Flow} is a direct application of Corollary~\ref{Cor:Original_and_Slow_Flow_Close}. Concerning Theorem~\ref{Thm:Maxwell-Bloch:Slow_Manifolds} we are in the easy situation that $B$ already generated a $C_0$-group. Thus, we may choose the trivial splitting
 \[
 	Y=Y^{\zeta}_{F}\oplus Y^{\zeta}_{S}:= \{0\}\oplus Y
 \]
 for all $\zeta>0$. Therefore, we may take $\zeta=C\epsilon$ for some $C\in(0,1)$, $N_F^{\zeta}=0$ and $N_S^{\zeta}=-\omega_A\zeta^{-1}-\kappa$. If $\epsilon>0$ is small enough, then all the conditions of Section~\ref{Section:Bates_Difficult:Our_Approach} can easily be verified and Theorem~\ref{Thm:Maxwell-Bloch:Slow_Manifolds} follows from the results in Section~\ref{sec:slowfinal}.

\section{Outlook}
\label{sec:outlook}

We have provided a quite general theory to use time scale separation in infinite-dimensional evolution equations with a focus on slow manifolds. Evidently, there are always further generalizations one could pursue. Examples are trying to weaken the conditions on the linear operators $A$ and $B$, trying to lift the theory into a completely non-standard form setting~\cite{Wechselberger4}, or extending it to quasilinear problems~\cite{Amann_2019}. In addition, the case of loss of invertibility/hyperbolicity of the fast dynamics has been a key focus in many finite-dimensional problems~\cite{KuehnBook}, i.e., in this scenario one has to track invariant slow manifolds through special regions. Therefore, combining our slow manifold theory here with the recent development of the blow-up method for fast-slow PDEs~\cite{EngelKuehn1} is a natural challenge for future work. 

From the viewpoint of applications, several directions are likely to be important. First, one may want to compute the invariant slow manifolds numerically, and we refer to~\cite[Ch.~11]{KuehnBook} for a survey of methods available for computing slow manifolds for finite-dimensional fast-slow systems. In fact, our analytically intermediate approximation~\eqref{Eq:Nonlinear_Slow_Subsystem_Extended} provides a hint, how to prove rigorous error estimates for computational methods based upon the invariance equation and/or iterated asymptotics for infinite-dimensional fast-slow dynamics. Second, working out concrete examples from pattern formation problems will be relevant as this can provide additional insights, which aspects of the theory need extensions, while others are immediately applicable. Third, trying to make many results, which have been obtained only via formal asymptotic matching methods for PDEs, rigorous is likely to be possible since a similar strategy using Fenichel theory has worked already in finite dimensions~\cite{Kaper,KuehnBook}.\medskip   

\textbf{Acknowledgments:} FH and CK acknowledge support of the EU within the TiPES project funded the European Unions Horizon 2020 research and innovation programme under grant agreement No.~820970. CK has also been supported by a Lichtenberg Professorship of the VolkswagenStiftung. FH and CK also acknowledge partial support of the SFB/TR109 ``Discretization in Geometry and Dynamics''.

\bibliographystyle{plain}
\bibliography{Bibliography}
\end{document}